\documentclass[a4paper,notitlepage,twoside,leqno,12pt]{amsart}

\usepackage{mathrsfs}
\usepackage{bbm,pifont,latexsym}
\usepackage{anysize} 
\marginsize{2.7cm}{2.7cm}{2.5cm}{2.5cm}

\usepackage{dcolumn,indentfirst}
\usepackage[pdfpagemode=UseNone,linktocpage=true,pdfstartview=FitH]{hyperref}
\usepackage{amsmath,amssymb,amscd,amsthm,amsfonts,mathrsfs}
\usepackage{color,graphicx,xcolor,graphics,subfigure,extarrows,caption2}
\usepackage{titletoc} 
\usepackage{tikz}
\usepackage{tikz-cd}
\usepackage{mathrsfs}
\usepackage{amsmath, amssymb, amsthm}
\usepackage{amssymb}
\usepackage{enumitem}
\usepackage{setspace}
\usepackage{mathtools}
\usepackage{float}
\newtheorem{thm}{Theorem}[section]
\newtheorem{lema}[thm]{Lemma}
\newtheorem{cor}[thm]{Corollary}
\newtheorem{prop}[thm]{Proposition}

\newtheorem{prob}{Problem}

\theoremstyle{definition}

\newtheorem*{rmk}{Remark}

\newcommand{\D}{\mathbb{D}}

\newcommand{\C}{\mathbb{C}}

\newcommand{\EEE}{\mathcal{E}}
\newcommand{\UUU}{\mathcal{U}}
\newcommand{\DDD}{\mathcal{D}}
\newcommand{\NNN}{\mathcal{N}}

\newcommand{\VVV}{\mathcal{V}}
\newcommand{\Crit}{\textup{Crit}}

\newcommand{\diam}{\textup{diam}}

\newcommand{\pa}{\partial}

\newcommand{\mb}{\mathbb}
\newcommand{\mc}{\mathcal}
\newcommand{\sm}{\setminus}
\newcommand{\tu}{\textup}
\newcommand{\ol}{\overline}

\newcommand{\tb}{\textbf}

\makeatletter\@addtoreset{equation}{section}\makeatother 
\title{Rigidity of McMullen Julia sets}

\author{Yan Gao}
\address{Yan Gao, School of Mathematical Sciences, Shenzhen University, Shenzhen 518061, China}
\email{gyan@szu.edu.cn}

\author{Luxian Yang}
\address{Luxian Yang, School of Mathematical Sciences, Shenzhen University, Shenzhen 518061, China}
\email{lxyang@szu.edu.cn}

\author{Jinsong Zeng}
\address{Jinsong Zeng, School of Mathematical Sciences, Shenzhen University, Shenzhen 518061, China}
\email{jinsongzeng@163.com}

\begin{document}
%\address{School of Mathematical Sciences, Shenzhen University, Shenzhen 518061, China}

%---------------------------------------------------------------------------------------------------------------
%\footnotetext[1]{$^\dag$ the corresponding author.}

\begin{abstract}
We provide a complete  quasisymmetric classification of the Julia sets of postcritically finite McMullen maps $f_\lambda(z)=z^n+\lambda/z^n$ with $\lambda\in\mathbb{C}^*$ and $n\geq 2$, and prove that the quasisymmetry group of each such Julia set is exactly the finite dihedral group generated by the natural symmetries of the map. These results establish quasisymmetric rigidity for all topological classes in this family, including Sierpi\'{n}ski-like carpets, necklaces, and clusters, and provide the first known examples of rigid Julia sets in each of the three classes.
\end{abstract}
% AMS subject classifications (used in AMS journals)

% AMS keywords (used in AMS journals)

% today's date, or fill in whatever date you prefer
%\date{\today}

% acknowledge support, etc
% \thanks{This research was partially supported by NSF grant DOA-123456789.}
% \thanks{We would like to thank our colleagues for their helpful criticism.}

% dedication
% \dedicatory{Dedicated to Professor Donald Knuth on the occasion of his $100$th birthday}

\maketitle

%----------------------------------------------------------------------------------------------------------------
%\vskip1.0cm
%\tableofcontents
%----------------------------------------------------------------------------------------------------------------

%----------------------------------------------------------------------------------------------------------------
\section{Introduction}\label{introduction}
Quasisymmetric geometry of planar fractal sets has attracted considerable interest recently. A homeomorphism $h:X\to Y$ between subsets of the Riemann sphere $\overline{\mathbb{C}}$ is {\bf quasisymmetric} if it extends to an orientation-preserving homeomorphism over $\overline{\mathbb{C}}$ and there exists a homeomorphism $\eta:[0,\infty)\to[0,\infty)$ such that for any distinct $u,v,w\in X$,
\[
\frac{\sigma(h(u),h(v))}{\sigma(h(u),h(w))}\leq \eta\!\left(\frac{\sigma(u,v)}{\sigma(u,w)}\right),
\]
where $\sigma$ is the standard spherical metric.

A central problem in quasiconformal geometry is to classify fractal sets up to quasisymmetric equivalence, and to investigate the quasisymmetric rigidity for canonical representatives of each class \cite{Bonk2006ICM}. A subset of the Riemann sphere is said to be {\bf quasisymmetrically rigid} if every quasisymmetric self-homeomorphism of this set is the restriction of a M\"obius transformation of $\overline{\mathbb{C}}$ or an even more restrictive map. 

Quasisymmetric rigidity has been established for several types of fractal sets, including the standard Sierpi\'{n}ski carpet $S_p$ for odd integer $p\ge3$ \cite{BM1,BM2}, the slit carpet $S_2$ \cite{Mer}, and the Schottky sets \cite{BKM,Mer12}. 

Conformal dynamics provides another rich source of planar fractal sets, including Julia sets of rational maps and limit sets of Kleinian groups. These two families often produce topologically equivalent fractals, which are conjecturally never quasisymmetrically equivalent, with the exception of two cases: Jordan curves and $\overline{\mathbb{C}}$~\cite{LLMM}. 

By \cite{BLM14}, a Sierpi\'{n}ski carpet arising as the Julia set of a {\bf postcritically finite} (i.e., all critical points are eventually periodic, {\bf PCF} for short) rational map is quasisymmetrically rigid, and this result was extended to some non-PCF cases in \cite{QYZ}.
In contrast, many other Julia sets are known to lack quasisymmetric rigidity. Examples include the Basilica Julia set \cite{LM,LMM}, Apollonian gasket Julia sets \cite{LLMM}, and finitely ramified fractal Julia sets \cite{BF}. 

In this paper, we study the quasisymmetric rigidity for Julia sets of PCF rational maps in the \textbf{McMullen family}
$$f_{\lambda}(z)=z^n+\dfrac{\lambda}{z^n},$$
where $\lambda\in\C^*=\C\sm\{0\}$ and $n\geq 2$ is an integer. This family has been shown to exhibit a variety of interesting dynamical and topological phenomena; see, for example \cite{DLU05,QWY12,QRWY15}.

There exist infinitely many PCF parameters $\lambda$ for which the maps $f_\lambda$ have Sierpi\'{n}ski carpet Julia sets. Devaney and Pilgrim posed a question on classifying these Julia sets up to quasisymmetric equivalence \cite{DP}. 
One such map is $f(z)=z^2-1/(16z^2)$. The authors of \cite{BLM14} speculate that the two M\"{o}bius transformations 
$$\xi_1(z)=\tb{i}z\quad\text{and}\quad \xi_2(z)=\frac{1}{4z}$$
generate the quasisymmetry group of the Julia set $J_f$; see Figure \ref{fig:classification}. 

There also exist infinitely many PCF parameters $\lambda$ for which the Julia sets of $f_\lambda$ are not Sierpi\'{n}ski carpets. As will be shown in Theorem \ref{thm1}, the Julia sets of PCF McMullen maps fall into four distinct topological classes: Sierpi\'{n}ski carpets, Sierpi\'{n}ski-like carpets, necklaces, and clusters. Figure \ref{fig:classification} illustrates examples of each class.

The main results of the paper resolve the quasisymmetric classification and rigidity problems for Julia sets of all PCF McMullen maps. In particular, we confirm the above speculation raised in \cite{BLM14}, and answer the question of Devaney and Pilgrim \cite{DP}.

\begin{figure}[H]
\centering
\begin{tikzpicture}
\node at (0,0){ \includegraphics[width=14.cm]{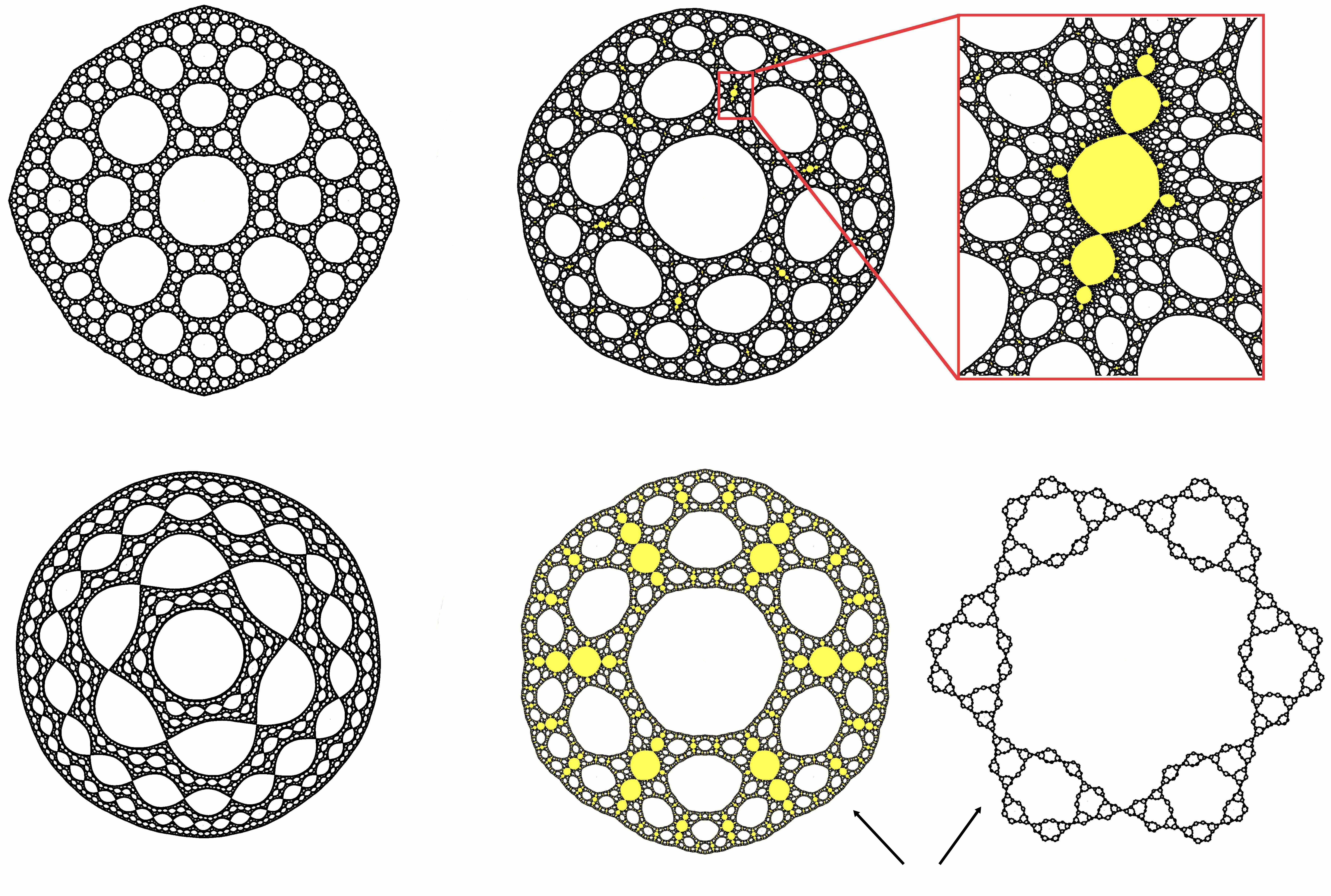}};
\node at (-5.25, 0.25){Sierpi\'{n}ski carpet};
\node at (2.75, 0.25){Sierpi\'{n}ski-like carpet};
\node at (-5.25, -5){Necklace};
\node at (2.85, -5){Clusters};
%\ruler{6}{5}
\end{tikzpicture}
\caption{Classification of PCF McMullen Julia sets.}\label{fig:classification}
\end{figure}
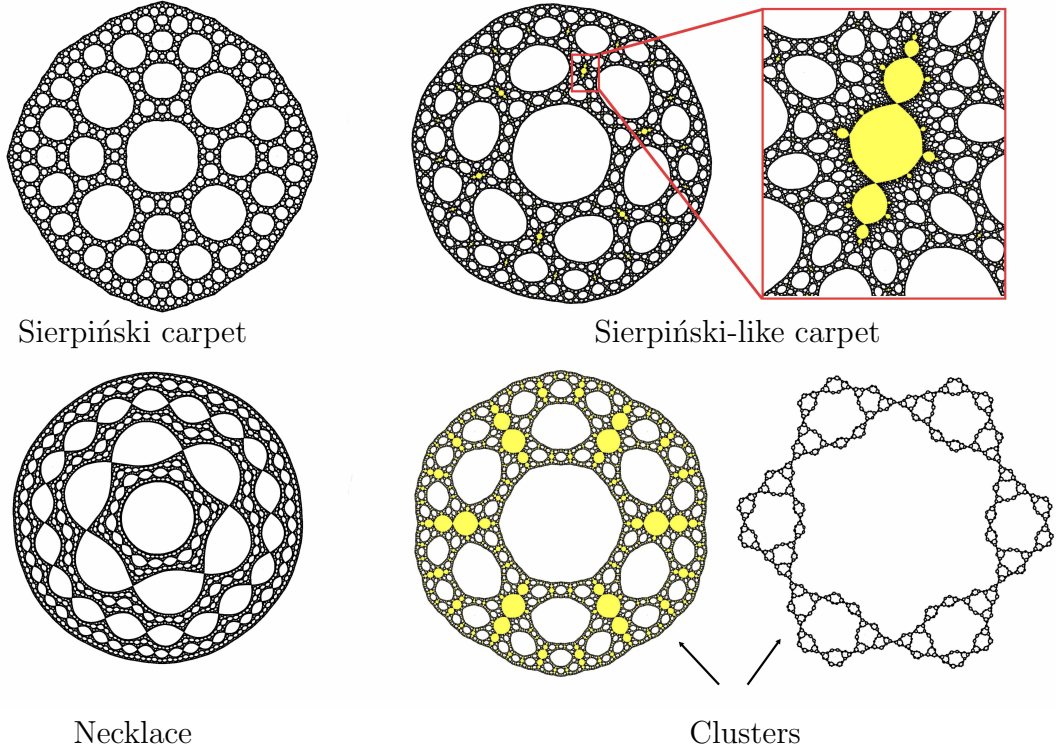

\subsection{Statement of main results}
For a McMullen map $f_\lambda(z)=z^n+\lambda/z^n$, the integer $n\geq 2$ is called its \textbf{exponent}. 
It is easy to check that the Julia set of $f_\lambda$ is invariant under the rotation $z\mapsto e^{\pi \tb{i}/n}z$ and the involution $z\mapsto \sqrt[n]{\lambda}/z$ (Lemma \ref{lem:symmetric}). Hence, the group of all quasisymmetric self-homeomorphisms of the Julia set (called the \textbf{quasisymmetry group}) contains the group $G_\lambda$ generated by these two M\"obius transformations. The following theorem shows that this inclusion is in fact an equality. 

\begin{thm}[Quasisymmetry group]\label{main1}
Let $f_\lambda$ be a PCF McMullen map of exponent $n\ge2$. Then the quasisymmetry group of the Julia set of $f_\lambda$  is exactly $G_\lambda$.
\end{thm}

Ha\"issinsky and Pilgrim proved that for McMullen maps with Cantor-circle Julia sets (non-PCF case), distinct exponents imply quasisymmetric inequivalence~\cite{HP12}. Since the order of $G_\lambda$ is a quasisymmetry invariant, we obtain a corollary to Theorem \ref{main1}. 

\begin{cor}\label{cor:distinct-exponent}
If two PCF McMullen maps have distinct exponents, then their Julia sets are not quasisymmetrically equivalent.
\end{cor}

Theorem \ref{main1} implies that if the Julia set of $f_\lambda$ is a Sierpi\'{n}ski carpet, then it is not quasisymmetrically equivalent to the standard Sierpi\'{n}ski carpet $S_p$. As non-elementary Kleinian limit sets have infinite quasisymmetry groups, Theorem \ref{main1} excludes quasisymmetric equivalence between such limit sets and PCF McMullen Julia sets.

\medskip The following result gives the quasisymmetric classification of Julia sets of PCF McMullen maps. For a set $E\subseteq\overline{\mathbb{C}}$ and $a\in\C^*$, we write $aE=\{az:z\in E\}$.

\begin{thm}[Quasisymmetric classification]\label{main2}
   Let $f_\lambda$ and $f_{\tilde{\lambda}}$ be PCF McMullen maps with the same exponent $n\geq 2$, and let $J$ and $\tilde{J}$ be their Julia sets, respectively. Then $J$ and $\tilde{J}$ are quasisymmetrically equivalent if and only if $\tilde{\lambda}=e^{\tb{i}\theta}\lambda$ with $\theta=2\pi k/(n-1)$ for $k\in\{0, \ldots, n-2\}$. In this case, we have $\tilde{J}=e^{\tb{i}\theta/(2n)}J$.
\end{thm}

Geometrically, Theorem \ref{main2} reflects the rotational symmetry of the parameter space. For a fixed exponent $n$, the connectedness locus $$\Lambda_n=\{\lambda\in\C^*: \text{the Julia set of $f_\lambda$ is connected.}\}$$ is invariant under the rotation $\lambda\mapsto e^{2\pi\tb{i}/(n-1)}\lambda$, see Figure \ref{fig:Parameter}. This parameter rotation corresponds precisely to a rotation of the Julia set by $e^{\pi\tb{i}/(n(n-1))}$, yielding a quasisymmetrically equivalent copy. Theorem \ref{main2} asserts that this is the \textbf{only} source of quasisymmetric equivalence; hence the quasisymmetric moduli space of PCF McMullen Julia sets is the quotient of \(\Lambda_n\) by this \((n-1)\)-fold rotation. In particular, for \(n=2\) the parameter rotation is trivial, so distinct parameters always produce Julia sets that are not quasisymmetrically equivalent.

\medskip As an interesting byproduct of these rigidity results, we obtain the following:

\begin{cor}\label{cor:distinct}
A PCF McMullen map is uniquely determined by its Julia set.
\end{cor}

\begin{proof}[Proof of Corollary \ref{cor:distinct}]
Suppose that $f_{\lambda}$ and $f_{\tilde{\lambda}}$ are two PCF McMullen maps with exponents $n,\tilde{n} \geq 2$ respectively. Assume that $J=\tilde{J}$. Then by Corollary \ref{cor:distinct-exponent}, $n=\tilde{n}$. Applying Theorem \ref{main2} to $\xi=\tu{id}$, we obtain $\tilde{\lambda}=e^{\tb{i}\theta}\lambda$, where $\theta=2k\pi/(n-1)$ for some integer $k$ satisfying $0\leq k<n-1$, and the rotation $z\mapsto e^{\tb{i}\theta/(2n)}z$ belongs to the quasisymmetry group of $J$. By Theorem \ref{main1}, there exists an integer $m$ such that $\theta/(2n)=m\pi/n$. This implies $k=m(n-1)$. The condition $0\leq k< n-1$ forces $k=0$. Hence, $\lambda=\tilde{\lambda}$.
\end{proof}

\begin{rmk}
There exist rational maps outside the McMullen family whose Julia sets coincide with McMullen Julia sets. For example, $g(z)=e^{\tb{i}\pi/n}f_\lambda(z)$ and $f_\lambda$, which do not commute (i.e., $g\circ f_\lambda\neq f_\lambda\circ g$), share the same Julia set (Lemma~\ref{lem:parameter-sym}). 
\end{rmk}

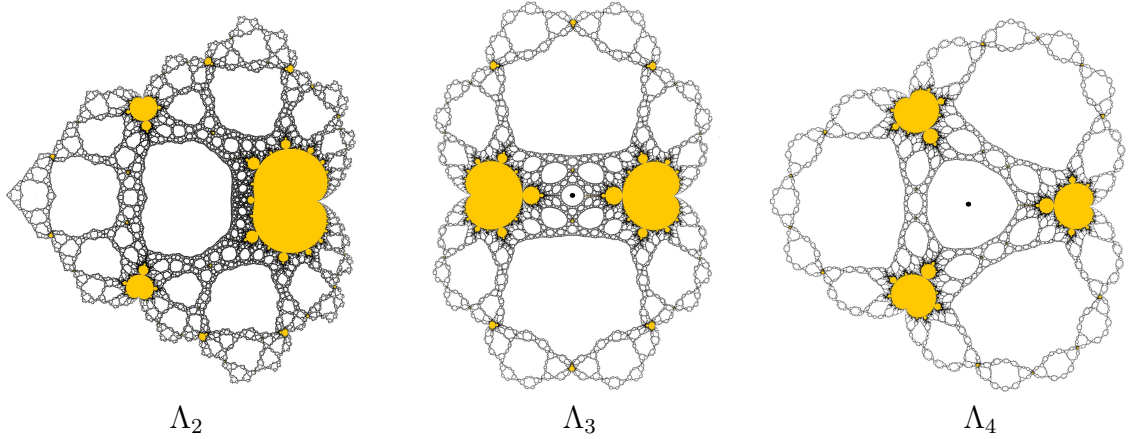
\begin{figure}[http]
\centering
\begin{tikzpicture} 
\node at (0,0){ \includegraphics[width=15cm]{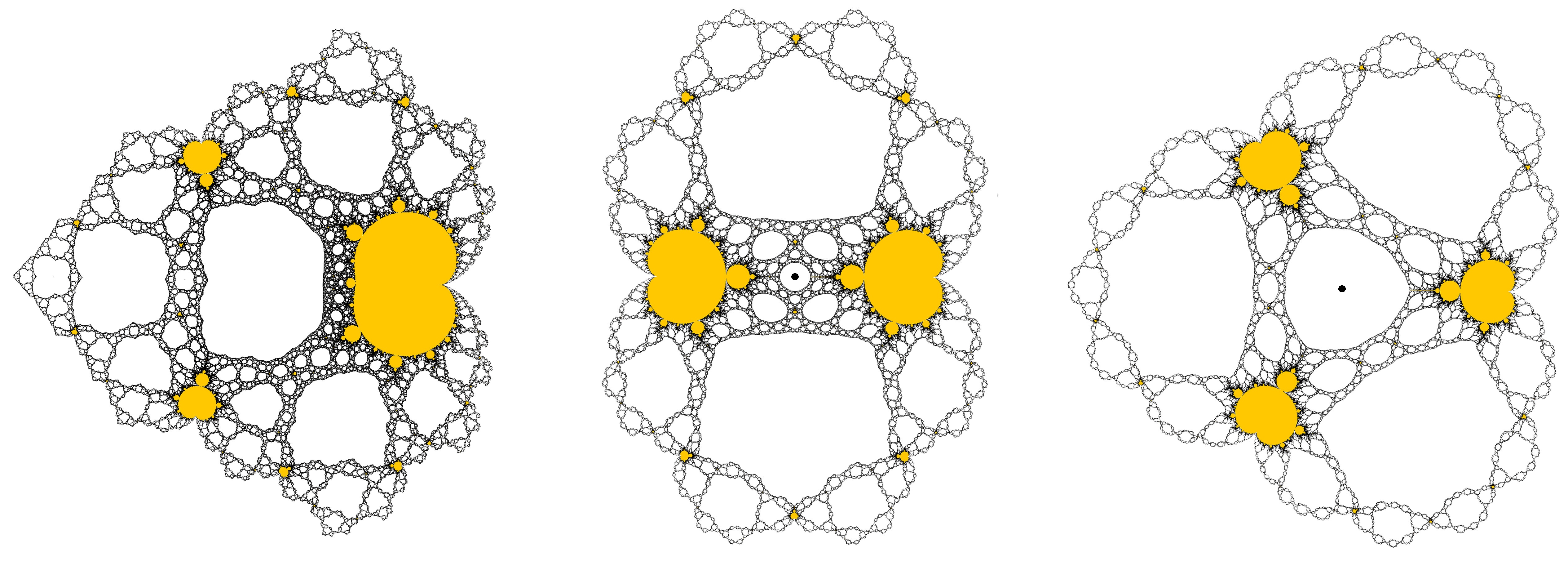}};
\node at (0.2, -2.8){$\Lambda_3$};
\node at (-5, -2.8){$\Lambda_2$};
\node at (5.5, -2.8){$\Lambda_4$};
\end{tikzpicture}
\caption{The connectedness loci of McMullen maps. Here $\Lambda_2$ is the complement of the unbounded hyperbolic component; $\Lambda_3$ is the complement of the union of the two hyperbolic components containing $0$ and $\infty$, respectively, and so is $\Lambda_4$.}\label{fig:Parameter}
\end{figure}

Theorems \ref{main1} and \ref{main2} are both consequences of the following rigidity theorem, which constitutes the main technical achievement of this paper. The sufficiency in Theorem \ref{main2} follows from a simple fact (Lemma \ref{lem:parameter-sym}): if $\tilde{\lambda}=e^{\tb{i}\theta}\lambda$ with $\theta=2\pi k/(n-1)$, then $\tilde{J}=e^{\tb{i}\theta/(2n)}J$.  

\begin{thm}\label{thm:mobius}
   Let $f_\lambda$ and $f_{\tilde{\lambda}}$ be PCF McMullen maps with the same exponent $n\geq 2$, and let $J$ and $\tilde{J}$ be their Julia sets, respectively. Suppose that $\xi: J\to \tilde{J}$ is a quasisymmetry. Then there exists an angle $\theta = 2\pi k/(n-1)$ for some integer $0\le k\le n-2$ such that
\begin{itemize} 
\item[(1)] $\tilde{\lambda} = e^{\textup{\tb{i}}\theta}\lambda$; and\vspace{2pt}
\item[(2)] $\xi(z) = e^{\textup{\tb{i}}\theta/(2n)}z$ up to precomposition with an element of $G_\lambda$.
\end{itemize} 
\end{thm}
\vskip 0.3cm

\subsection{Outline of the proof and main techniques}\label{sec:outline}

In \cite{BLM14}, the authors observe that any quasisymmetry between Sierpi\'{n}ski carpet Julia sets of PCF rational maps forces a hidden relation between the underlying dynamics:

\begin{prop}[{\cite[Theorem 1.4]{BLM14}}]\label{pro:sierpinski}
Let \( J \) and \(\tilde{J} \) be Sierpi\'{n}ski carpet Julia sets of PCF rational maps \( f \) and \(\tilde{f} \), respectively. If \( \xi: J \to \tilde{J} \) is a quasisymmetry, then there exist integers  \( m', l \geq 1 \) and \( m \geq 0 \) such that
\[
\tilde{f}^{m'} \circ \xi = \tilde{f}^{m} \circ \xi \circ f^l  \text{\,\,on } J.
\]

\end{prop}

From this, they prove quasisymmetric rigidity for Sierpi\'nski carpet Julia sets. We observe that the passage from the dynamical relation in Proposition \ref{pro:sierpinski} to the M\"{o}bius property works in wider generality: it applies to any PCF rational map whose Fatou domains are Jordan domains:

\begin{prop}\label{thm:BLM}
Let $f$ and $\tilde{f}$ be PCF rational maps with non-empty Fatou sets such that each Fatou domain of $f$ is a Jordan domain. Suppose $\xi: J\to \tilde{J}$ is a quasisymmetry satisfying the dynamical relation
\begin{equation}\label{eq:26}
\tilde{f}^{m'}\circ \xi = \tilde{f}^m \circ \xi \circ f^l \text{\,\,on\,\,\,} J
\end{equation}
for some integers $m', l\ge 1$ and $m\ge 0$.
Then $\xi$ is the restriction of a M\"obius transformation, and \eqref{eq:26} still holds on $\overline{\mathbb{C}}$ (also denoted by $\xi$). In particular, $\xi$ maps each Fatou domain of $f$ and its center to those of $\tilde{f}$. 
\end{prop}

To prove Theorem~\ref{thm:mobius}, we establish the dynamical relation \eqref{eq:26} for all PCF McMullen maps (Sections~\ref{sec5}--\ref{sec7}). Proposition~\ref{thm:BLM} then implies that $\xi$ is the restriction of a M\"{o}bius transformation. Tracking the images of the critical Fatou domains $B$ and $T$ (containing $\infty$ and $0$, respectively) under the dynamical relation, we find that $\xi$ fixes $0$ and $\infty$ up to an element of $G_\lambda$, whence $\xi(z)=az$. A comparison of Laurent expansions at infinity yields $|a|=1$; comparing local degrees further forces $a^{2n}$ to be an $(n-1)$-th root of unity. This identifies $\xi$ as the rotation $z\mapsto e^{\tb{i}\theta/(2n)}z$ and yields Theorem~\ref{thm:mobius}.
\vskip 0.1cm

The proof of the dynamical relation \eqref{eq:26} proceeds differently for each topological class of Julia sets. We therefore need  the following topological classification theorem for McMullen maps. 
\vskip 0.3cm

A \textbf{continuum} is a compact and connected subset of $\overline{\mathbb{C}}$ containing at least two points. A continuum is called \textbf{full} if its complement in $\overline{\mathbb{C}}$ is connected. A point $z$ in a compact set $E$ is called {\bf buried} in $E$ if it does not lie on the boundary of any component of $\overline{\mathbb{C}}\sm E$. A continuous onto map $\pi:\overline{\mathbb{C}}\to \overline{\mathbb{C}}$ is a \textbf{quotient map} if $\pi^{-1}(z)$ is a singleton or a full continuum for all $z\in\overline{\mathbb{C}}$.
\vskip 0.1cm

These notions allow us to describe the four possible topological classes of Julia sets for PCF McMullen maps precisely.

\begin{thm}\label{thm1}
Let $J$ be the Julia set of a PCF McMullen map $f_\lambda$ with exponent $n\geq 2$. Then exactly one of the following occurs:
\begin{enumerate}[itemsep=2pt,parsep=0pt]
\item[(1)] $J$ is a Sierpi\'{n}ski carpet; 
\item[(2)] $J$ is a \tb{Sierpi\'{n}ski-like carpet}, i.e., the image of a Sierpi\'{n}ski carpet by a quotient map over $\overline{\mathbb{C}}$ that is injective on the buried point set;
\item[(3)] $J$ is a \tb{necklace}, i.e., the image of a Cantor-circle by a quotient map over $\overline{\mathbb{C}}$ that is at most two-to-one on the Cantor-circle; 
\item[(4)] $J$ is a \tb{cluster}, i.e., the Fatou domains $B$ and $T$ (containing $\infty$ and 0, respectively) can be joined by a curve meeting $J$ in countably many points.
\end{enumerate}
Moreover, Case $(3)$ occurs only when $n\geq 3$.
\end{thm}

Figure~\ref{fig:classification} shows examples of each class. We note that all necklace Julia sets of the same exponent are topologically equivalent and have an infinite homeomorphism group. In contrast, cluster Julia sets of McMullen maps exhibit a stronger rigidity: they are topologically rigid (hence also quasisymmetrically rigid); see Theorem~\ref{thm:intersection-I}.

The cases of Sierpi\'nski carpet and necklace Julia sets are well understood \cite{DLU05,QWY12,QRWY15}. Theorem \ref{thm1} completes the classification by showing that the remaining Julia sets of PCF McMullen maps are either Sierpi\'nski-like or cluster.  
To this end, we use the tool of {Fatou chains} developed in \cite{CGZ} (see Section \ref{sec:fatouchain}), instead of the Yoccoz puzzle technique in previous works.
 \vspace{3pt}

Building on the classification of Theorem~\ref{thm1}, we establish the dynamical relation by treating each topological class separately. The outline is as follows. 
\vspace{3pt}

For Sierpi\'nski carpet Julia sets, relation \eqref{eq:26} has been established in \cite{BLM14}, and the proof depends essentially on the quasisymmetric rigidity of Schottky maps \cite{Mer14} and the uniformization of Sierpi\'{n}ski carpets \cite{Bon11}.\vspace{3pt}

For Sierpi\'nski-like carpet and necklace Julia sets, we use a {lifting argument}: blowing up $J$ to a Sierpi\'nski carpet and a standard Cantor-circle, respectively; lifting $\xi$ to a quasisymmetry $\tilde{\xi}$ between the blown-up Julia sets of model maps $g$; then showing that $\tilde{\xi}$ and $g$ satisfy \eqref{eq:26}, which projects down to $\xi$ and the original McMullen maps.\vspace{3pt}

Specifically, when $J$ is a Sierpi\'nski-like carpet, such a model map exists by \cite[Theorem~1.6]{CGZ}. The heart of the argument is to prove that the lifted map $\tilde{\xi}$ is quasisymmetric; this relies on a distortion estimate (Lemma~\ref{lem:distortion}) controlling the diameters of iterates. Once $\tilde{\xi}$ is known to be quasisymmetric, Proposition~\ref{pro:sierpinski} applies and yields the required dynamical relation (Theorem~\ref{thm:sierpinski-like}).\vspace{3pt}

When $J$ is a necklace, we prove that it can be blown up to a standard Cantor-circle, the Julia set of a model map $g$\,(Theorem~\ref{thm:blow-upII}). Strikingly, although the Cantor-circle admits infinitely many quasisymmetries, the rule that ``equivalent points are mapped to equivalent points'' forces the lift $\tilde{\xi}$ to be a rotation, which therefore satisfies the desired relation (Theorem~\ref{thm:necklace}).\vspace{3pt}

For cluster Julia sets, we introduce the {connecting number}, a topological invariant: simply count how many components connect two Fatou domains. We prove that $B$ and $T$ are the only pair with the maximal count, so any homeomorphism must preserve these two domains. Moreover, the count records how deep each Fatou domain sits in the cluster. Using this, we argue inductively that the homeomorphism respects the entire internal structure, then label the Fatou domains to match the two dynamics, and pass to the limit to obtain the conjugacy (Theorem~\ref{thm:intersection-I}).
\vskip 0.3cm

\subsection{Further problems and generalizations}
The arguments in this paper lead to several natural generalizations. Theorem~\ref{thm:sierpinski-like} is stated for general PCF rational maps with Sierpi\'{n}ski-like carpet Julia sets. Together with Proposition~\ref{thm:BLM}, this gives quasisymmetric rigidity for such Julia sets in general. Similarly, the necklace case admits a natural generalization via {folding rational maps}, introduced and studied extensively in \cite[Section~8]{CPT2}; it seems likely that the Cantor-circle lifting argument applies to these maps as well, and their quasisymmetric rigidity is the subject of our ongoing work.

In a broader context, Cui together with the first and third authors have shown that the dynamics of any PCF rational map with non-empty Fatou set decomposes into three types of subsystems: {exact subsystems}, {annular subsystems}, and {cluster subsystems}~\cite[Theorem~1.7]{CGZ}. Sierpi\'{n}ski-like carpets arise as special cases of exact subsystems; necklaces are the simplest instances of annular subsystems; and cluster Julia sets of McMullen maps together with polynomial Julia sets are special cases of cluster subsystems. In general, a cluster Julia set is one in which every pair of Fatou domains can be joined by a curve meeting the Julia set in at most countably many points. Our results therefore cover the simplest representatives of each subsystem type.

This motivates the following problem:

\begin{prob}
    If a PCF rational map has non-empty Fatou set and its Julia set is not a cluster, does it necessarily exhibit quasisymmetric rigidity?
\end{prob}

In contrast, cluster Julia sets behave quite differently. As shown by Lyubich and Merenkov \cite{LM}, the Basilica Julia set admits quasisymmetric self-maps that are not M\"{o}bius; similarly, Apollonian gasket Julia sets \cite{LLMM} and finitely ramified fractal Julia sets \cite{BF} are known to possess nontrivial quasisymmetry groups. To date, the cluster Julia sets of PCF McMullen maps remain the first known examples in the cluster family to exhibit quasisymmetric rigidity.
\vspace{8pt}

The paper is organized as follows. Section \ref{sec2} proves some basic properties of McMullen maps. Section \ref{Fatou chains} studies the symmetry of Fatou chains. Section \ref{sec3} presents the blow-up theorems for Sierpi\'{n}ski-like and necklace Julia sets and gives the proof of Theorem \ref{thm1}. Sections~\ref{sec5}, \ref{sec6} and \ref{sec7} establish the dynamical relation for PCF McMullen maps whose Julia sets are Sierpi\'nski-like carpets, necklaces and clusters, respectively. Section \ref{sec8} completes the proof of the rigidity theorems (Theorem \ref{thm:mobius}).

\medskip\noindent\textbf{Acknowledgements.}
We would like to thank Guizhen Cui and Yueyang Wang for valuable discussions and suggestions. The first author is supported by the  National Key R\&D Program of China (Grant no. 2021YFA1003203), the NSFC (Grant no. 12131016,12322104). The second author is supported by the NSFC (Grant 
no. 12526609). The third author is supported by the NSFC (Grant 
no. 12271115).

\section{Topology, symmetries and Fatou domains}\label{sec2}
This section develops the topological properties and symmetries of McMullen maps. We establish a rotation-invariant continuum lemma, and employ the dihedral symmetry to control the geometry of the Julia sets and Fatou domains.

\subsection{A topological lemma}\label{sec:topo}
 We begin with a topological lemma to be used repeatedly below. We say that a set $E\subseteq\overline{\mathbb{C}}$ \textbf{separates} two points $x,y\in\overline{\mathbb{C}}$ if $x$ and $y$ lie in different connected components of $\overline{\mathbb{C}}\setminus E$.

\begin{lema}\label{lem:separate}
Let $E\subseteq \ol{\mb{C}}\setminus\{0,\infty\}$ be a continuum and let $m\geq 2$ be an integer. 
\begin{itemize}
\item[(1)] If $E\cap e^{\textup{\tb{i}}\theta}E\neq \emptyset$ for some $\theta\in \mathbb{R}\setminus 2\pi\mathbb{Z}$, then $K=\overline{\bigcup_{k\geq 0}e^{\textup{\tb{i}}k\theta}E}$ separates $0$ and $\infty$. 

\item[(2)] If $E$ separates $0$ and $\infty$, then for any $\theta\in \mb{R}$, $E\cap e^{ \textup{\tb{i}}\theta} E\neq \emptyset$. 

\item[(3)] If $e^{\textup{\tb{i}}\theta}E=E$ for some $\theta\in\mathbb{R}\setminus 2\pi\mathbb{Z}$, then $E$ separates 0 and $\infty$. Moreover, for any component $U$ of $\ol{\mb{C}}\setminus E$ containing neither $0$ nor $\infty$, $e^{\textup{\tb{i}}k\theta}U\cap U\neq\emptyset$ if and only if $e^{\textup{\tb{i}}k\theta}= 1$.
\item[(4)] Suppose there exists a dense subset $A$ of $E$ such that any two points of $E$ can be joined by an open arc within $A$. If $\omega A\cap A=\emptyset$ for any $\omega\neq 1$ with $\omega^m=1$, then
$\omega_0 E\cap E=\emptyset$, where $\omega_0\neq 1$ is a non-primitive $m$-th root of unity.

\end{itemize}
\end{lema}

\begin{proof}
(1) Since $E\cap e^{\tb{i}\theta}E\neq\emptyset$, there exist points $z_0, w_0\in E$ with $w_0=e^{\tb{i}\theta}z_0$. Note that $z_0\neq0,\infty$ because $E\subseteq\overline{\mathbb{C}}\setminus\{0,\infty\}$.

If $\theta/(2\pi)\notin\mathbb{Q}$, then the set $\{e^{\tb{i}k\theta}z_0:k\geq0\}$ is dense in the circle $C=\{z:|z|=|z_0|\}$. Since $e^{\tb{i}k\theta}z_0\in e^{\tb{i}k\theta}E\subseteq K$ for every $k\geq0$ and $K$ is closed, we obtain $C\subseteq K$. Hence $K$ separates $0$ and $\infty$.

If $\theta/(2\pi)\in\mathbb{Q}\setminus\mathbb{Z}$, write $\theta=2\pi p/q$ with $\gcd(p,q)=1$ and $q\geq2$. 
Then $K=\bigcup_{k=0}^{q-1}e^{\tb{i}k\theta}E$ and $K=e^{\tb{i}\theta}K$. Assume, to the contrary, that $0$ and $\infty$ lie in the same component of $\overline{\mathbb{C}}\setminus K$. Then there exists an arc $\gamma\subseteq\overline{\mathbb{C}}\setminus K$ connecting $0$ and $\infty$. Set
$$\Gamma=\gamma\cup e^{\tb{i}\theta}\gamma\cup\cdots\cup e^{\tb{i}(q-1)\theta}\gamma.$$
Since $e^{\tb{i}\theta}K=K$ and $\gamma\cap K=\emptyset$, we have $\Gamma\cap K=\emptyset$; hence $K$ is contained in a component $W$ of $\overline{\mathbb{C}}\setminus\Gamma$. The invariance $e^{\tb{i}\theta}K=K$ and $e^{\tb{i}\theta}\Gamma=\Gamma$ forces $e^{\tb{i}\theta}W=W$. Since $W$ is simply connected and invariant, $W$ contains a fixed point of $z\mapsto e^{\tb{i}\theta}z$, contradicting that $0,\infty\in \Gamma$. Therefore, $K$ separates $0$ and $\infty$. 

\medskip\noindent(2) Let $r=\tu{min}\{|z|: z\in E\}$ and $R=\tu{max}\{|z|: z\in E\}$. If $r=R$, then $E$ is a circle centered at 0, and the conclusion is obvious. Assume $r<R$, and consider the closed annulus $A=\{z: r\leq |z|\leq  R\}$. If there exists some $\theta$ such that $E\cap e^{\tb{i}\theta}E=\emptyset$, then $e^{\textbf{i}\theta}E$ would lie entirely in a component of $A\setminus E$, which is a bounded simply connected domain. This is impossible, as such a domain cannot separate $0$ and $\infty$. 

\medskip\noindent(3) By Statement (1), $E$ clearly separates $0$ and $\infty$. If $\theta/(2\pi)\in\mathbb{R}\sm\mathbb{Q}$, then the circle $\{z\in\mathbb{C}:|z|=|z_0|\}$ is contained in $E$ provided that $z_0\in E$. This implies $E$ is a circle or a closed annulus centered at $0$, so the conclusion is trivial in this case. 

If $\theta/(2\pi)\in\mathbb{Q}\setminus\mathbb{Z}$, write $\theta=2\pi p/q$ with $\gcd(p,q)=1$ and $q\geq2$. For the sake of contradiction, assume that $e^{\tb{i}k_0\theta}U\cap U\not=\emptyset$ for some integer $k_0$ with $e^{\tb{i}k_0\theta}\neq 1$. Then there exists a point $z\in U$ such that $e^{\tb{i}k_0\theta}z\in U$. Choose an arc $\gamma\subseteq U$ joining $z$ and $e^{\tb{i}k_0\theta} z$. By Statement (1), $$\Gamma=\gamma\cup e^{\tb{i}k_0\theta}\gamma\cup \cdots\cup e^{\tb{i}k_0(q-1)\theta}\gamma$$ separates $0$ and $\infty$. Since $E=e^{\tb{i}k_0\theta}E$ and $\gamma\cap E=\emptyset$, we have $\Gamma\cap E=\emptyset$. Hence $\Gamma\subseteq U$. This contradicts the fact that $U$ does not separate $0$ and $\infty$.  

\medskip\noindent(4) Suppose that $\omega_0E\cap E\not=\emptyset$ for some non-primitive $m$-th root of unity $\omega_0$. Then there exist $z, \omega_0 z\in E$ joined by an open arc $\gamma\subseteq A$. 
By Statement (1), $$\Gamma=\overline{\gamma}\cup \omega_0\overline{\gamma}\cup \cdots\cup \omega_0^{m-1}\overline{\gamma}$$
 separates $0$ and $\infty$. Since $\omega_0$ is non-primitive, there exists $\omega$ such that $\omega^m=1$ and $\omega\neq \omega_0^i$ for any integer $i\geq 0$. By Statement (2), $\Gamma\cap \omega\Gamma\neq\emptyset$. Hence $\ol{\gamma}\cap \omega\omega_0^j\ol{\gamma}\neq \emptyset$ for some integer $j\geq1$. If the intersection occurs only at the endpoints of $\gamma$, then $\omega=\omega_0^i$ for some integer $i\geq 1$, a contradiction. Thus $\gamma\cap \omega\omega_0^j\gamma\neq\emptyset$. Since $\gamma\subseteq A$ and $(\omega\omega_0^j)^m=1$ with $\omega\omega_0^j\neq 1$, the hypothesis yields $(\omega\omega_0^j)A\cap A=\emptyset$, a contradiction.
\end{proof}

\subsection{Symmetries of McMullen maps}\label{sec:symmetry}

The results in this subsection apply to all McMullen maps
$f_\lambda(z)=z^n+{\lambda}/{z^n},$
where $n\geq 2$ is an integer and $\lambda\in\mathbb{C}^*$. For simplicity, we write $f=f_\lambda$.

Let $J$ denote the Julia set of $f$. Note that $\infty$ is a superattracting fixed point of $f$ and that $f^{-1}(\infty)=\{0,\infty\}$. We write $B$ for the immediate basin of $\infty$, and let $T$ be the component of $f^{-1}(B)$ containing $0$. The map $f$ has $2n$ simple critical points
$$\tu{Crit}=\{\sqrt[2n]{\lambda}\,\omega : \omega^{2n}=1\},$$
called the \textbf{free} critical points, and two \textbf{free} critical values $\pm v=\pm2\sqrt{\lambda}$. Moreover, $f$ satisfies $f(-z)=(-1)^n f(z)$; hence $f(v)=f(-v)$ when $n$ is even and $f(v)=-f(-v)$ when $n$ is odd. 

Recall that $G_\lambda$ denotes the group generated by the rotation $z\mapsto e^{\pi \tb{i}/n} z$ and the involution $z\mapsto \sqrt[n]{\lambda}/z$, which is the dihedral group $D_{2n}$ of order $4n$.

\begin{lema}[$G_\lambda$-invariance of the Julia set]\label{lem:symmetric} Let $\omega$ be a $2n$-th root of unity, and $\tau$ be the involution given by $z\mapsto \tilde{\omega}/z$, where $\tilde{\omega}^n=\lambda$. Then %Let $f$ be a McMullen map with exponent $n\geq2$, and $\omega$ be a $2n$-th root of unity. Then
\begin{enumerate}
\item[(1)] $J=\omega J$, $B=\omega B$, and $T=\omega T$. 
\item[(2)] $\tau(B)=T,\tau(T)=B$ and $\tau(J)=J$.
\end{enumerate} 
\end{lema}
\begin{proof}
The proof relies on the symmetry satisfied by McMullen maps $f$, namely
\begin{equation}\label{symmetry}
f(\omega z)=\omega^n f(z)
\end{equation}
and $f\circ\tau(z)=f(z)$.
This implies that escaping to infinity under iteration of $f$ is preserved by both maps $z\mapsto \omega z$ and $\tau$. Consequently, the attracting basin of infinity, and therefore  its boundary, which is the Julia set $J$, is invariant under these transformations. The remaining statements follow from a direct verification using the definitions of $B$ and $T$.
\end{proof}

\begin{lema}\label{lem:non-sym-domain}
Let $U$ be a Fatou domain of $f$ such that $U\neq B, T$ and $f(U)\neq T$. Then $\omega\ol{U}\cap\ol{U}=\emptyset$ for any $\omega\neq 1$ with $\omega^{2n}=1$. 
\end{lema}
\begin{proof} Since $\omega J=J$ for any $\omega\neq 1$ with $\omega^{2n}=1$ (Lemma \ref{lem:symmetric}), the set $\omega U$ is also a Fatou domain of $f$. By Lemma \ref{lem:separate}\,(3), $U\cap\omega U=\emptyset$. Thus, $E=\ol{U}$ and $A=U$ satisfies the assumption of Lemma \ref{lem:separate}\,(4). We then obtain that $\omega \ol{U}\cap\ol{U}=\emptyset$ when $\omega$ is non-primitive. 

If $\omega$ is primitive, then $\omega^n=-1$. By symmetry \eqref{symmetry}, we have $f(\omega\overline{U})=-f(\overline{U})$. Since $f(U)\neq B, T$ by assumption, we have $\omega^kf(U)\cap f(U)=\emptyset$ whenever $\omega^k\neq 1$ by Lemma \ref{lem:separate}\,(3). Note that $-1$ is a non-primitive $2n$-th root of unity for $n\geq 2$, we apply Lemma \ref{lem:separate}\,(4) to $E=\ol{f(U)}$ and $A=f(U)$ with $\omega_0=-1$ to obtain $$(-\overline{f(U)})\cap \overline{f(U)}=\emptyset.$$Since $f(\omega U)=-f(U)$, it follows that $$f(\omega \overline{U}\cap \overline{U})\subset f(\omega \overline{U})\cap f(\overline{U})=(-\overline{f(U)})\cap\overline{f(U)}=\emptyset.$$  Hence, $\omega \ol{U}\cap\ol{U}=\emptyset$. Now the proof is complete. 
\end{proof}

The following lemma implies that the parameter space of McMullen maps also has the rotational symmetries; see Figure \ref{fig:Parameter}. 
\begin{lema}\label{lem:parameter-sym} Let $f=f_\lambda$ be a McMullen map with exponent $n\geq 2$. \vspace{3pt}

(1) The Julia sets of $g(z):=\omega f(z)$ and $f$ coincide for any $2n$-th root $\omega$ of unity.

(2) Suppose $\tilde{\lambda}=e^{{\bf i}\theta}\lambda$ with $\theta={2\pi k}/{(n-1)}, k=1,\ldots,n-2$. Let $\tilde{J}$ be the Julia set of $f_{\tilde{\lambda}}$. Then $\tilde{J}=e^{{\bf i}\theta/(2n)}J$.
\end{lema}
\begin{proof}
(1) By the symmetry \eqref{symmetry},
$$g^2(z)=\omega f(\omega f(z))=\omega\cdot\omega^n f^2(z)=\omega^{n+1}f^2(z).$$
By induction, $g^k(z)=\omega^{a_k}f^k(z)$ for all $k\geq1$, where $a_k=(n^k-1)/(n-1)$. Consequently, a point $z$ tends to $\infty$ under iteration of $f$ if and only if it does so under $g$. Hence $f$ and $g$ have the same attracting basin of $\infty$, and therefore $J_g=J$.

\medskip\noindent(2) Let $\xi(z)=e^{\tb{i}\theta/(2n)}z$ be a rotation. The conjugate map
$$h(z)=\xi\circ f\circ\xi^{-1}(z)=e^{\tb{i}\theta/(2n)}\left(\frac{z^n}{e^{\tb{i}\theta/2}}+\frac{\lambda e^{\tb{i}\theta/2}}{z^n}\right)$$
equals $e^{-\pi\tb{i}k/n}f_{\tilde{\lambda}}(z)$, where $\tilde{\lambda}=e^{\tb{i}\theta}\lambda$. By Statement~(1) the Julia sets of $h$ and $f_{\tilde{\lambda}}$ coincide. Therefore $\tilde{J}=\xi(J)$.
\end{proof}

\subsection{Fatou domains $B$ and $T$}\label{}
We now assume that the McMullen map $f$ is PCF with exponent $n\geq2$. Then $B\neq T$ and $J$ is connected and locally connected. Since both $B$ and $T$ are simply connected, there exist unique Riemann mappings  $$\phi_B: B\to\overline{\mathbb{C}}\setminus\overline{\mathbb{D}}\ \tu{ and }\ \phi_T: T\to{\mathbb{D}}$$ satisfying $\phi'_B(\infty)>0$ and $\phi'_T(0)>0$. 

The maps $\phi_B$ and $\phi_T$ are called \textbf{B\"{o}ttcher maps} of $f$ on $B$ and $T$, respectively. We define $\phi$ on $B\cup T$ by $\phi|_B=\phi_B$ and $\phi|_T=\phi_T$. 
 By the uniqueness of Riemann maps and Lemma \ref{lem:symmetric}, we have $\phi(\omega z)=\omega \phi(z)$. 
The preimages of radial rays in $\mathbb{D}$ and $\mathbb{C}\setminus\overline{\mathbb{D}}$ of $\phi$ are called \textbf{internal rays} in $T$ and $B$, respectively.

The Yoccoz puzzle technique was employed in \cite{QWY12} to prove that for a McMullen map with connected Julia set and exponent $n \geq 3$, the basin $B$ is a Jordan domain. Their approach fails for $n = 2$, leaving open whether $\partial B$ is a Jordan curve in this case. When $f$ is PCF, by using Lemma \ref{lem:separate}, instead of the Yoccoz puzzle technique, we obtain the following result: both $B$ and $T$ are Jordan domains.

\begin{lema}\label{lem:jordan-domains}
Let $f$ be a PCF McMullen map with exponent $n\geq 2$. Then the Fatou domains $B$ and $T$ are Jordan domains. 
\end{lema}

\begin{proof}
%[Proof of Lemma \ref{lem:jordan-domains}]
We argue by contradiction and assume $\partial B$ is not a Jordan curve. Since $B$ is simply connected and $\partial B$ is locally connected, $K=\ol{\mb{C}}\setminus B$ is a full continuum and each component of $\tu{int}(K)$ is a Jordan domain. Let $\Omega_0$ be the component of $\tu{int}(K)$ containing $0$.

We claim that $f$ maps any component $\Omega\neq\Omega_0$ of $\tu{int}(K)$ onto a component of $\tu{int}(K)$. In fact, the set $f(\Omega)$ is disjoint from $B$ since $f^{-1}(\infty)=\{0,\infty\}$. Moreover, $f(\partial \Omega)\subseteq \partial B=\pa K$. Consequently, the open connected set $f(\Omega)$ is a component of $\tu{int}(K)$. This proves the claim. 

Since for any point $z\in\Omega_0$, the Julia set $J$ is contained in the closure of the backward orbit $O^-(z)=\{w: f^k(w)=z, k\in\mathbb{N}\}$, there exists a component $\Omega\neq\Omega_0$ of $\tu{int}(K)$ such that $f(\Omega)=\Omega_0$. Otherwise, $f^{-1}(\Omega_0)\subset \Omega_0$ implies that $O^-(z)$ is contained in $\ol{\Omega_0}$. This contradicts that $\pa B\sm \ol{\Omega_0}\neq \emptyset$. 
   
Let $\omega$ be a primitive $2n$-th root of unity. Since $\omega\overline{ B}=\overline{ B}$ (Lemma \ref{lem:symmetric}), it follows from Lemma \ref{lem:separate}\,(3) that
 the $2n$ domains $\omega^k\Omega, k=0,\ldots, 2n-1$ are pairwise disjoint, and are all mapped onto $\Omega_0$ by the symmetry \eqref{symmetry}. Note that the closures of any two components of ${\rm int}(K)$ intersect in at most one point. Then there exists $w\in \partial \Omega_0$ which does not lie in the intersection $\pa \Omega_0 \cap(\bigcup_{k=0}^{2n-1}\omega^k\ol{\Omega})$, and $2n$ distinct points $z_k\in \omega^k\overline{ \Omega}\setminus \overline{\Omega_0},k=0,\ldots,2n-1$ such that $f(z_k)=w$. As $T\subset \Omega_0$ and $f(\partial T)=\partial B$, there also exists a point $z\in \overline{\Omega}_0$ such that $f(z)=w$. Thus, $f^{-1}(w)$ consists of at least $2n+1$ distinct points, a contradiction to the degree of $f$. Therefore, $\partial B$ is a Jordan curve.

We now show, by contradiction, that $\partial T$ is a Jordan curve. Assume that there are two internal rays of $T$ landing at the same point $z \in \partial T$. Since $\partial B$ is a Jordan curve, their images under $f$ form a single internal ray $R$ of $B$ landing at $f(z) \in \partial B$. Clearly, $z$ is a free critical point. Let $W$ be the bounded domain enclosed by these two rays. Then
$$W\subseteq \Omega_0\tu{ and }f(W)\supseteq \overline{\mathbb{C}}\setminus \overline{R}.$$ Thus, $W$ contains a component of $f^{-1}(W)$. By the Lefschetz fixed point theorem, $W$ contains at least one fixed point of $f$.

By symmetry \eqref{symmetry}, all free critical points of $f$ lie on $\partial \Omega_0$, each receiving at least two internal rays of $T$. Similar arguments as above show that each bounded domain enclosed by the two internal rays landing at a free critical point contains a fixed point of $f$. Hence, $\Omega_0$ contains at least $2n$ fixed points of $f$. Moreover, $f: \overline{B} \to \overline{B}$ is conjugate to $z \mapsto z^n$ on $\overline{\mathbb{C}} \setminus \mathbb{D}$, so $f$ has $n$ fixed points on $\overline{B}$. Hence, $f$ would have at least $3n$ fixed points in total. This contradicts that a rational map of degree $2n$ has exactly $2n+1$ fixed points counted with multiplicity and $n\geq 2$. Therefore, $\partial T$ is a Jordan curve.
\end{proof}

\begin{lema}\label{lem:disjoint}
Let $f$ be a PCF McMullen map with exponent $n\geq 2$. Then $\pa B\cap \pa T\neq \emptyset$ if and only if $v\in\partial B$. In this case, $\pa B\cap \pa T=\tu{Crit}$. 
\end{lema}
\begin{proof}
Suppose $\partial B\cap\partial T\neq\emptyset$. For any $z\in\partial B\cap\partial T$, internal rays in $B$ and $T$ landing at $z$ map under $f$ to a single internal ray of $B$; hence $z$ is a free critical point and $f(z)=v\in\partial B$. Thus $\partial B\cap\partial T\subseteq\Crit$.

Conversely, suppose $v\in\partial B$. Then $-v\in\partial B$ by Lemma~\ref{lem:symmetric}. Let $c\in\Crit$ with $f(c)=v$. Since $c$ is a simple critical point, $f$ near $c$ is locally a two-to-one branched cover. The internal ray of $B$ landing at $v$ lifts under $f$ to two rays landing at $c$. Since both $\partial B$ and $\partial T$ are Jordan domains, the two rays lie in distinct components of $f^{-1}(B)=B\cup T$; hence one is in $B$ and the other in $T$. This implies $c\in \partial B\cap\partial T$. The same holds for preimages of $-v$. Thus $\Crit\subset\partial B\cap\partial T$ and $\partial B\cap\partial T\neq\emptyset$.

Combining the two directions yields $\partial B\cap\partial T=\Crit$ in this case.
\end{proof}
\begin{prop}\label{prop:classify-no-bounded-cycles}
    Let $f$ be a PCF McMullen map with $n\geq 2$. Then every Fatou domain eventually iterated onto $B$ is a Jordan domain. If $f$ has no bounded attracting cycles, then exactly one of the following holds:
    \begin{enumerate}
    \item $f^k(v)=\infty$ for some integer $k\geq 0$. Then $k\geq 2$ and $J$ is a Sierpi\'{n}ski carpet.
    \item $v\in\partial B$. This occurs if and only if $\partial B\cap \partial T\neq \emptyset$. In this case, $\partial B\cap\partial T=\tu{Crit}$ and thus $J$ is a cluster.
     \item $v\in\partial T$. Then there exists a bounded Fatou domain $U$ such that $\ol{U}$ and $e^{\pi\textup{\tb{i}}/n}\ol{U}$ meet at a free critical point, and such that $$f^{-1}(\overline{T})=\bigcup_{k=0}^{2n-1}e^{\pi\textup{\tb{i}}k/n}\overline{U}$$ is a degenerate annulus.
        When $n=2$, this case does not occur. 
        \item $f^k(v)\in \partial T$ for some integer $k\geq 1$. Then the components of $$T\cup f^{-1}(T)\cup\cdots\cup f^{-k}(T)$$ are Jordan domains with pairwise disjoint closures; the free critical values $\pm v$ lie on the boundaries of two of them, and the preimage of each such closure is a full continuum whose interior consists of two Jordan domains.
        \item $v$ is buried in $J$. Then $J$ is a Sierpi\'{n}ski carpet.
    \end{enumerate}
\end{prop}
\begin{proof}
We first claim: if $\partial B\cap P_f=\emptyset$, then every Fatou domain eventually mapped onto $B$ is a Jordan domain with pairwise disjoint closures. Here $P_f$ denotes the postcritical set of $f$. Since $\partial B$ is a Jordan curve (Lemma~\ref{lem:jordan-domains}) and $\overline{B}$ contains only one postcritical point, each component of $f^{-i}(\partial B)$ ($i\ge1$) is a Jordan curve; any Fatou domain mapped onto $B$ by $f^i$ is bounded by such a component. For disjointness, if two such domains map to $B$ by $f^i$ and $f^j$ with $i\le j$, forward iteration by $f^i$ sends their boundaries into $\partial B$ and $f^{-(j-i)}(\partial B)$, which are disjoint because $\partial B\cap P_f=\emptyset$.

If $f$ has a bounded attracting cycle, then $\partial B\cap P_f=\emptyset$, and the claim implies that any Fatou domain iterated onto $B$ is a Jordan domain.

 \noindent(1) Since $0$ is the only pole and $v\neq0$, we have $k\ge2$. In this case $\partial B\cap P_f=\emptyset$, so the claim applies. Since $J$ is locally connected, then it is a Sierpi\'{n}ski carpet.

 \noindent(2) It follows from Lemma~\ref{lem:disjoint} and the definition of cluster Julia sets in Theorem~\ref{thm1}.

 \noindent(3) By symmetry $-v\in\partial T$. Lemma~\ref{lem:disjoint} implies $\partial B\cap\partial T=\emptyset$. Let $A=\ol{\mb{C}}\setminus\overline{B\cup T}$. The Riemann-Hurwitz formula implies that every component of $f^{-1}(A)$ is an annulus. 
 Since
    $$\partial B\cup\partial T\cup f^{-1}(\partial T)=\partial f^{-1}(A),$$
    $f^{-1}(A)$ consists of two sub-annuli $A_1$ and $A_2$, each sharing a boundary component with $A$. Furthermore, $\partial A_1\cap\partial A_2$ is precisely the set of $2n$ free critical points. Hence $f^{-1}(T)=A\setminus\ol{A_1\cup A_2}$ and it consists of $2n$ Fatou domains, all of which are Jordan domains. Since $\pm v\notin\ol{f^{-1}(T)}$, Fatou domains iterated into $f^{-1}(T)$ are Jordan domains.

    A comparison of the conformal moduli of $A$, $A_1$ and $A_2$ together with Gr\"{o}tzsch's inequality gives
    $$\tu{mod}(A)>\tu{mod}(A_1)+\tu{mod}(A_2)=\frac{2}{n}\tu{mod}(A),$$
    which forces $n\neq2$. Therefore, this configuration cannot occur when $n=2$.

    \medskip\noindent (4) By symmetry $f^k(-v)=(-1)^n f^k(v)\in\partial T$. Since Case~(4) is distinct from Case~(2), we have $v\notin\partial B$; Lemma~\ref{lem:disjoint} then yields $\partial B\cap\partial T=\emptyset$. Consequently $\partial B\cap P_f\neq\emptyset$ (because $f^{k+1}(v)\in\partial B$), so the initial claim does \emph{not} apply directly. We argue separately.

    \smallskip\noindent\textit{Step~1: Jordan domain property.}
    Since $f^{k}(v)\in\partial T$, the points $\{v,f(v),\ldots,f^{k-1}(v)\}$ lie in the Julia set, hence are disjoint from every Fatou domain $f^{-(i-1)}(T)$ ($i\le k$). Therefore for each $i\le k$, the map $f:f^{-i}(T)\to f^{-(i-1)}(T)$ is a covering; each component of $f^{-i}(T)$ maps homeomorphically onto a component of $f^{-(i-1)}(T)$ and is thus a Jordan domain. Closures at the same level are disjoint because they are bounded by distinct components of $f^{-i}(\partial T)$, which are disjoint Jordan curves.

    \smallskip\noindent\textit{Step~2: Disjointness across levels.}
    Let $U\subseteq f^{-i}(T)$ and $V\subseteq f^{-j}(T)$ with $i<j\le k$. If $\overline{U}\cap\overline{V}\neq\emptyset$, then $f^i(\overline{U}\cap\overline{V})\subseteq\partial T\cap f^{-(j-i)}(\partial T)$. Since the critical orbit $\{v,f(v),\ldots,f^{k-1}(v)\}$ avoids $\bigcup_{\ell=1}^{k}f^{-(\ell-1)}(T)$, $\partial T$ is disjoint from $f^{-(j-i)}(\partial T)$, a contradiction. Hence closures at different levels are also disjoint.

    \smallskip\noindent\textit{Step~3: Location of critical values.}
    Since $\pm v\in f^{-k}(\partial T)$, each lies on the boundary of some component of $f^{-k}(T)$; denote them by $U_+$ and $U_-$. If $U_+=U_-$, then both $\pm v\in\partial U_+$, so $\ol{U_+}\cap -\ol{U_+}\neq \emptyset$, contradicting Lemma \ref{lem:non-sym-domain}. Hence $U_+\neq U_-$.

    \smallskip\noindent\textit{Step~4: Preimage structure.}
    Consider $f^{-1}(\overline{U_+})$. The interior $f^{-1}(U_+)$ consists of $2n$ components, each a Jordan domain. The $n$ simple critical points mapping to $v$ lie on the common boundary of these components; each such critical point is adjacent to exactly two of them. Consequently $f^{-1}(\overline{U_+})$ splits into $n$ connected components. Each component is the closure of two Jordan domains meeting at a single critical point. The same description holds for $f^{-1}(\overline{U_-})$.

    \medskip\noindent (5) This follows directly from the claim, since $v$ buried in $J$ implies $\partial B\cap P_f=\emptyset$.
\end{proof}
\section{Fatou chains}\label{Fatou chains}

This section studies the maximal Fatou chain of a PCF McMullen map with bounded attracting cycles, establishing the topological properties needed for the blow-up theorems in Section~\ref{sec3}. 

In this section we suppose $f$ is such a map with exponent $n\geq 2$; thus the free critical value $v$ or $-v$ (or both) lies in a periodic attracting cycle. Denote by $\mathcal{U}$ the union of immediate basins of all bounded attracting cycles. There are two cases:

{ \textbf{Case 1.} $n$ is even}. Then $\mathcal{U}$ contains exactly one critical value. Indeed, the critical values $\pm v$ have the same image, so at most one domain in $\mathcal{U}$ can contain a critical value. 

{  \textbf{Case 2.} $n$ is odd}. Then the symmetry \eqref{symmetry} implies that $\mathcal{U}$ consists of  one or two Fatou cycles and contains both critical values $\pm v$. Moreover, we have $\mathcal{U}=-\mathcal{U}$. 

\subsection{Maximal Fatou chains}\label{sec:fatouchain}
We now give the definition of maximal Fatou chains. Set $\mc{E}^1=\overline{\mathcal{U}}$. For $k\geq 1$, define $\mc{E}_k^1$ inductively as the union of all components of $f^{-k}(\mc{E}^1)$ that intersect $\mc{E}^1$.  As $f(\mc{E}^1)\subset \mc{E}^1$, we have 
$$\mc{E}^1\subseteq\mc{E}^1_1\subseteq\cdots \subseteq \mc{E}_{k-1}^1\subseteq \mc{E}_{k}^1\subseteq\cdots $$
and each set in this chain has the same number of components.

If $\mc{E}_1^1=\mc{E}^1$, then $\mc{E}_k^1=\mc{E}^1$ for all $k$, and $\mc{E}^1$ is {\bf stable}, i.e., $f(\mc{E}^1)=\mc{E}^1$ and every component of $\mc{E}^1$ is a component of $f^{-1}(\mc{E}^1)$. 

If $\mc{E}_1^1\neq \mc{E}^1$, we iterate the construction: let $\mc{E}^2=\overline{\bigcup_k\mc{E}_k^1}$ and for $k\geq 1$ define $\mc{E}_k^2$ as the union of all components of $f^{-k}(\mc{E}^2)$ that intersect $\mc{E}^2$. If $\mc{E}_1^2=\mc{E}^2$, then $\mc{E}^2$ is stable. If not, we repeat the procedure with $\mc{E}^2$ playing the role of $\mc{E}_1$, thereby generating a sequence of invariant sets $\mc{E}^1, \mc{E}^2, \ldots, \mc{E}^m,\ldots$. Since the number of components of $\mc{E}^m$ is non-increasing, there exists an index $m_0$ such that
$$\#\textup{Comp}(\mc{E}^m)=\#\textup{Comp}(\mc{E}^{m_0})\textup{
for all }m\geq m_0,$$
where $\tu{Comp}(\cdot)$ denotes the collection of all components of a set. If a component $E$ of $\mc{E}^m$ is strictly contained in a component $E'$ of $\mc{E}_1^m$ for some $m\geq m_0$, then for any point $z\neq v_{\pm}$ in $f(E)\cap \mc{E}^1$ we have 
$$\#(f^{-1}(z)\cap E)<\#(f^{-1}(z)\cap E')\leq 2n.$$ Consequently, there is a large integer $N$ for which $\mc{E}^{N}=\mc{E}_1^N$. It is clear that $\mc{E}^N$ is stable. We call $\mathcal{E}=\mc{E}^N$ the {\bf maximal Fatou chain} of $f$ ({ generated by $\mathcal{U}$}). For every integer $m\geq 1$, the continuum $\mc{E}^m=\overline{\bigcup_k\mc{E}_k^{m-1}}$ is called the {\boldmath \bf level-$m$ Fatou chain\unboldmath}. 

\vspace{3pt}

The following lemma characterize the topological and combinatorial properties of the maximal Fatou chain. 
\begin{lema}\label{classification}
Suppose a PCF McMullen map $f$ has bounded attracting cycles and $\mc{U}$ is the union of immediate basins of all bounded attracting cycles. Let $\mathcal{E}$ be the maximal Fatou chain of $f$ generated by $\mc{U}$. Then 
\begin{enumerate}
\item each component of $\mc{E}$ is a full continuum containing at most one critical value;
\item every two points in a component $E$ of $\mathcal{E}$ can be joined by a curve in $E$ passing through at most countably many points in the Julia set $J$;
\item if $\mc{E}\cap \pa B\neq \emptyset$, then every component of $\mc{E}$ intersects $\partial B$ at a single point. 
\end{enumerate}
\end{lema}

%Recall that for every $1\leq m\leq N$, $$\mc{E}^m $$
% is called the {\bf level-m Fatou chain}. 

The proof of Lemma \ref{classification} relies on the following {rotational disjointness property} of Fatou chains. Recall that $n\geq 2$ is the exponent of $f$. 

\begin{lema}\label{lem:Fatou-cycle}
 For any $1\leq m\leq N$, the following statements hold.
\begin{enumerate}
\item If $n$ is even, then $\omega \EEE^m\cap \EEE^m=\emptyset$ for any $\omega$ with $\omega^{2n}=1$ and $\omega\neq 1$. 
\item If $n$ is odd, then $\omega \EEE^m\cap \EEE^m=\emptyset$ for any $\omega$ with $\omega^{2n}=1$ and $\omega\neq \pm 1$.  
\item If $n$ is odd, then $\EEE^m=-\EEE^m$. Moreover, if $E$ is a component of $\EEE^m$, then $-E$ is a component of $\EEE^m$ disjoint from $E$.
\end{enumerate}
\end{lema}
\begin{proof}[Proof of Lemma \ref{classification} under Lemma \ref{lem:Fatou-cycle}]

\medskip\noindent (1) To prove  that every component of $\EEE$ is a full continuum, it suffices to show that $\overline{\mathbb{C}}\setminus\EEE$ is connected. 

Note first that $0$ and $\infty$ lie in the same component $\Omega_0$ of $\overline{\mathbb{C}}\setminus\EEE$; for otherwise, some component $E$ of $\EEE$ would separate $0$ and $\infty$. Then $e^{\pi \tb{i}/n} E$ would also separate $0$ and $\infty$. By Lemma \ref{lem:separate}(2), we would have $e^{\pi \tb{i}/n} E\cap E\neq\emptyset$, contradicting Lemma \ref{lem:Fatou-cycle}\,(1) and (2) since $n\geq 2$.

Assume that some components of $\mc{E}$ are not full. Then $\overline{\mathbb{C}}\sm\mc{E}$ has a component $\Omega\neq\Omega_0$. For any $k\geq 1$, the set $f^k(\Omega)$ is open and connected with $\partial f^k(\Omega)\subseteq f^k(\partial \Omega)\subseteq\mc{E}$, so $\Omega_0$ is either contained in $f^k(\Omega)$ or disjoint from it. Hence $\Omega$ cannot be a Fatou domain eventually mapped to $B$; and since $\EEE$ is stable and contains $\UUU$, it cannot be a Fatou domain eventually mapped to $\UUU$ either. Therefore, $\Omega\cap J\neq\emptyset$. 

Let $k_0\ge 1$ be the smallest integer such that $0\in f^{k_0}(\Omega)$. Since $\Omega_0$ contains $0,\infty$, we have $\infty\in f^{k_0}(\Omega)$ and $\infty\notin f^{k_0-1}(\Omega)$. Hence $f^{k_0-1}(\Omega)$ contains the pole $0$, a contradiction. Therefore $\overline{\mathbb{C}}\setminus\mc{E}$ is connected, and consequently every component of $\mc{E}$ is a full continuum. 

If a component $E$ of $\mc{E}$ contains both $\pm v$, then $E\cap -E\neq\emptyset$, contradicting Lemma \ref{lem:Fatou-cycle}~(1)(3). This proves Statement (1).

\medskip\noindent (2) By \cite[Theorem 1.5]{CGZ}, for any component $E$ of $\EEE$, there exist a quasiconformal map $\psi$ on $\overline{\mathbb{C}}$ and a PCF polynomial $P$ such that $\psi(E)$ is the filled Julia set of $P$ and $P\circ\psi=\psi\circ f^p$, where $p$ is the period of $E$. Hence, $E$ is locally connected. 

For any two points in $E$, let $\gamma\subset E$ be an arc joining these two  points such that its intersection with any Fatou domain consists of two internal rays. By \cite[Proposition 2.3]{GZ}, $\psi(\gamma)\cap J_P$ is a countable set. Hence every two points in $E$ can be joined by a curve in $E$ passing through at most countably many points in $J$. 

\medskip\noindent (3) Now suppose $\mathcal{E}\cap\partial B\neq \emptyset$. Since $\EEE\cap \partial B$ is $f$-invariant and $\partial B$ is a Jordan curve, each component of $\mathcal{E}\cap\partial B$ is a single point: otherwise, by the expanding property of $J$, we would have $\partial B\subseteq \EEE$, a contradiction to the fact that each component of $\EEE$ is a full continuum.

To show that every component of $\mc{E}$ intersects $\partial B$ at a single point, by Lemma \ref{lem:Fatou-cycle} and Statement (1), it suffices to prove the connectivity of the open set $$W=\overline{\mathbb{C}}\setminus (\tilde{\EEE}\cup \overline{B}) ,$$
where $\tilde{\EEE}:=\bigcup_{\omega^{2n}=1}\omega\EEE.$

Denote by $\Omega_0$ the component of $W$ containing $T$. If $W$ were disconnected,  as in the proof of Statement (1), there would exist another component $\Omega\neq \Omega_0$ of $W$ such that $f(\Omega)$ covers $T$. Since $\tilde{\EEE}\cup \overline{B}$ is invariant under the rotation of $2n$-th root of unity, it follows from Lemma \ref{lem:separate}\,(3) that  $\{\omega\Omega, \omega^{2n}=1\}$ are $2n$ pairwise disjoint components of $W$. The symmetry \eqref{symmetry} implies  $\Omega_0\subseteq f(\omega\Omega)$.

On the other hand, the boundary of $f(\Omega_0)$ is contained in $\partial B\cup\mathcal{E}\cup-\mathcal{E}$. By Lemma \ref{lem:Fatou-cycle}\,(1) and (2), $\pa W\cap \pa B$ contains some arcs. Since $f$ preserves orientation on these arcs, we obtain $f(\Omega_0)\cap\Omega_0\not=\emptyset$. Consequently $T\subseteq \Omega_0\subseteq f(\Omega_0)$. This implies that $T$ has at least $2n+1$ pre-images under $f$. This leads to a contradiction. Hence, $W$ is connected, which implies every component of $\EEE$ intersects $\pa B$ at a single point. 
\end{proof}

\subsection{Disjointness under rotations}
This subsection is devoted to proving Lemma \ref{lem:Fatou-cycle}. Recall that for each $k\geq1$, $\EEE_k^m$ is the union of all components of $f^{-k}(\EEE^m)$ intersecting $\EEE^m$. The following result is a corollary of \cite[Lemmas 6.1 and 6.3]{CGZ}. 

\begin{prop}\label{pro:arc}
Fix an integer $m\geq 1$. Let $E_k$ be a sequence of sets such that each $E_k$ is a component of $\mc{E}_k^{m}$ and $E_k\subset E_{k+1}$. Let $E_\infty=\bigcup_{k\geq 1}E_k$. Then any two points in $\overline{E_\infty}$ can be joined by an open arc within $E_\infty$.
\end{prop}

\begin{proof}[Proof of Lemma \ref{lem:Fatou-cycle}]
The proof goes by induction on the level $m$. We first check that the statements in this lemma hold for $\EEE^1=\overline{\UUU}$.\vspace{2pt}

\medskip\noindent (1) Assume $n$ is even. Suppose for contradiciton that the statement is false. Then there exist a $2n$-th root of unity $\omega_0\neq1$, a component $U$ of $\UUU$, and a point $z\in \overline{U}$ such that $\omega_0 z\in \overline{U'}$ for some component $U'$ of $\UUU$. Lemma~\ref{lem:non-sym-domain} gives $U\neq U'$. By symmetry $f(\omega_0 z)=\omega_0^n f(z)$ and the fact that $n$ is even, we have $f^2(\omega_0 z)=f^2(z)$.
This implies $z\notin U$: for otherwise $f^2(U)=f^2(U')$, contradicting the periodicity of $U$ and $U'$. We then obtain that 
\begin{equation}\label{eq:22}
e^{k\pi \tb{i}/n}\UUU\cap\UUU=\emptyset \text{ for every } k=1,\ldots,2n-1.
\end{equation}

Since $z\in\partial U$ and $ \omega_0 z\in\partial U'$, and $f^2(\omega_0 z)=f^2(z)$, the boundaries of $f^2(U)$ and $f^2(U')$ share the point $f^2(z)$. As $U$ and $U'$ have the same period, $\partial U\cap\partial U'$ contains at least one point, say $x$. %Applying Lemma~\ref{lem:non-sym-domain} to $U$ and $U'$ yields $x\neq \omega^k z$ for all $k\ge 1$.
%Choose open arcs $\gamma\subseteq U$ joining $x$ to $z$ and $\gamma'\subseteq U'$ joining $x$ to $\omega z$. Set $\beta=\gamma\cup\{x\}\cup \gamma'$. 
If $\omega_0$ is non-primitive, then by equation \eqref{eq:22}, we can apply Lemma \ref{lem:separate}\,(4) with $E=\ol{U\cup U'}$, $A=U\cup\{x\}\cup U',\omega=e^{\pi \tb{i}/n}$ and $\omega_0$ to obtain that $\omega_0E\cap E=\emptyset$. This contradicts $\omega_0 z\in E\cap\omega_0E$.

If $\omega_0$ is primitive, then $\omega_0^n=-1$, so $f(z)\in\partial f(U)$ and $-f(z)\in -\partial f(U')$. Since $U$ and $U'$ are periodic, $f(U)\neq f(U')$. Replacing $(U,U',z,\omega_0)$ by $(f(U),f(U'),f(z),-1)$ in the previous argument then gives a contradiction. This completes the proof of (1).\vspace{3pt}

\medskip\noindent (2) Assume $n$ is odd. Suppose for contradiction that the statement fails. Then there exist a $2n$-th root of unity $\omega_0\neq\pm1$, a component $U$ of $\UUU$, and a point $z\in\overline{U}$ such that $\omega_0 z\in\overline{U'}$ for some component $U'$ of $\UUU$. Lemma~\ref{lem:non-sym-domain} implies $U\neq U'$ and similarly $U\neq -U'$ as $-\omega_0z\in -\overline{U'}$. Thus, $\pm U$ and $\pm U'$ are four distinct Fatou domains.

Since $\UUU=-\UUU$, we may replace $(U,U',\omega_0)$ by $(-U,-U',-\omega_0)$ if necessary and assume $\omega_0^n=1$. 
 By symmetry $f(\omega_0 z)=\omega_0^n f(z)$, we have $f(z)=f(\omega_0 z)$.

We claim $z\notin U$. For otherwise, $z\in U$ and $\omega_0z\in U'$ and thus $f(U)=f(U')$, contradicting the periodicity of $U$ and $U'$. The claim implies 
\begin{equation}\label{eq:23}
e^{k\pi \tb{i}/n}\UUU\cap\UUU=\emptyset \text{ for every } 1\leq k\leq 2n-1 \text{ and }k\not=n.
\end{equation}

By the claim, $z\in\partial U$ and $\omega_0 z\in\partial U'$. As in the proof of Statement (1), $\partial U\cap\partial U'$ contains a point $x$.
By equation \eqref{eq:23} and the fact that $\pm U$ and $\pm U'$ are pairwise disjoint when $k=n$, we can apply Lemma \ref{lem:separate}\,(4) with $E=\ol{U\cup U'}$,  $A=U\cup\{x\}\cup U',\omega=e^{\pi \tb{i}/n}$ and the non-primitive $\omega_0$ to obtain that $\omega_0E\cap E=\emptyset$. This contradicts $\omega_0 z\in E\cap\omega_0E$,
 completing the proof of Statement (2).\vskip 0.3cm

\medskip\noindent (3) When $n$ is odd, we have $\mathcal{U}=-\mathcal{U}$, and the same holds for its closure. If a component $E$ of $\EEE^1=\overline{\UUU}$ satisfies $E\cap -E\neq\emptyset$, then $E=-E$. By Lemma \ref{lem:separate}\,(1), $E$ separates $0$ and $\infty$. Lemma \ref{lem:separate}\,(2) then gives $E\cap \omega E\neq\emptyset$ for every $2n$-th root of unity $\omega$, contradicting Statement (2). This proves Statement (3).\vspace{3pt}

We now proceed by induction on the level $m$ of Fatou chains. Assuming the statements of this lemma hold for $\EEE^m$, we shall prove they also hold for $\EEE^{m+1}$. Since $\EEE^{m+1}=\overline{\bigcup_{k\ge 1}\EEE^m_k}$, we first establish the following claim.
\vskip 0.3cm
\noindent{\bf Claim.} \emph{For any $k\geq1$, the statements in the lemma still hold for $\EEE^m_k$.}

\begin{proof}[Proof of the claim]
We prove Statements (1)-(3) only for $\EEE_1^m$; the cases $k\ge 2$ follow by a similar induction and are omitted.  

\medskip\noindent (1) Assume $n$ is even. If $\omega^n=-1$, then $f(\omega \EEE_1^m)=-f(\EEE_1^m)=-\EEE^m$. By induction, we have $\EEE^m\cap -\EEE^m=\emptyset$. Since $f(\EEE_1^m)=\EEE^m$, it follows that $\omega \EEE_1^m\cap \EEE_1^m=\emptyset$. 

If $\omega^n=1$ and $\omega\neq1$, suppose for contradiction that $\omega E\cap E'\neq\emptyset$ for two components $E$ and $E'$ of $\EEE_1^m$. Since $f(\omega E)=f(E)\subseteq\EEE^m$, we have $\omega E\subseteq E'$ and $f(E)\subseteq f(E')$. Hence $f(E)\cap f(E')\neq\emptyset$, and since distinct components of $\EEE_1^m$ map to distinct components of $\EEE^m$, we conclude that $E'=E$. Therefore $\omega E\subseteq E$, which gives $E=\omega E$.
Hence $E=\bigcup_{i=1}^n\omega^i E$. 
By Lemma \ref{lem:separate}\,(1), $E$ separates $0$ and $\infty$. Choose an $n$-th root $\omega_0$ of $-1$. Lemma \ref{lem:separate}\,(2) then yields $E\cap \omega_0 E\neq\emptyset$. However, the first subcase already showed $E\cap \omega_0 E=\emptyset$. This contradiction proves Statement (1).\vspace{2pt}

\medskip\noindent (3) Assume $n$ is odd. Let $E$ be a component of $\EEE_1^m$ containing a component $E_0$ of $\EEE^m$, and let $K=f(E)$. Then $K$ is a component of $\EEE^m$. By the induction hypothesis, $-E_0\neq E_0$ and $-K\neq K$, and both $-E_0$ and $-K$ are components of $\EEE^m$. Since $n$ is odd, $f(-z)=(-1)^n f(z)=-f(z)$, which yields $f(-E)=-K$. Hence, $-E$ is a component of $f^{-1}(-K)$ disjoint from $E$ (as $K\neq -K$). Because $-E_0\subset -E\cap\EEE^m$, the component $-E$ intersects $\EEE^m$, and therefore is a component of $\EEE_1^m$. This proves Statement (3).\vspace{2pt}

\medskip\noindent (2) Assume $n$ is odd. Suppose for contradiction that $\omega E\cap E'\neq\emptyset$ for some components $E,E'$ of $\EEE_1^m$ and some $2n$-th root of unity $\omega\neq\pm1$. By Statement (3), we may replace $(\omega,E,E')$ with $(-\omega,-E,-E')$ if necessary and assume $\omega^n=1$. Then $f(\omega E)=f(E)\subseteq\EEE^m$. The same argument as in Statement (1) gives
$E=\omega E=\bigcup_{i=1}^{2n}\omega^iE.$
By Lemma \ref{lem:separate}\,(1), we conclude that  $E$ separates $0$ and $\infty$.  Lemma \ref{lem:separate}\,(2) then yields $E\cap-E\neq\emptyset$, contradicting Statement (3). This proves the claim.
\end{proof}

It remains to verify that the statements in the lemma hold for $\EEE^{m+1}$. The argument is analogous to the proof for $\EEE^1$ given earlier. Set $\EEE^m_\infty=\bigcup_{k\geq 1} \EEE_k^m$. Then $$f(\EEE_\infty^m)=\EEE_\infty^m\,\tu{ and }\,\EEE^{m+1}=\overline{\EEE_\infty^m}.$$

\noindent (1) Assume $n$ is even.  By the previous claim we have 
\begin{equation}\label{eq:24}
\text{$\EEE^m_\infty\cap e^{k\pi \tb{i}/n} \EEE_\infty^m=\emptyset$ for any $k=1,\ldots,2n-1$.}
\end{equation}

Suppose for contradiction that Statement (1) fails for $\mc{E}^{m+1}$. Then there exist $z,\omega_0 z\in\mc{E}^{m+1}$ for some $2n$-th root of unity $\omega_0\neq 1$. Hence there are sequences $\{E_k\}$ and $\{E'_k\}$ of components of $\mc{E}_k^m$, with $E_k\subset E_{k+1}$ and $E'_k\subset E'_{k+1}$, such that $z$ and $\omega_0z$ are contained in the closures of $ X:=\bigcup_k E_k$ and $Y:=\bigcup_k E'_k$, respectively. 

By symmetry $f(\omega_0 z)=\omega_0^n f(z)$ and the fact that $n$ is even, we have $f^2(\omega_0 z)=f^2(z)$.
Then $f^2(\overline{X})$ and $f^2(\overline{Y})$ share the point $f^2(z)$. As $X$ and $Y$ are periodic with the same period, $\overline{X}\cap\overline{Y}$ contains at least one point, say $x$. 

Suppose first that  $\omega_0$ is non-primitive. Then by relation \eqref{eq:24} and Proposition \ref{pro:arc}, one of the following two cases occurs:

$\bullet$ If $X=Y$, we can apply Lemma \ref{lem:separate}\,(4) with $E=\ol{X}, A=X,\omega=e^{\pi \tb{i}/n}$ and $\omega_0$ to obtain that $\omega_0E\cap E=\emptyset$. This contradicts $\omega_0 z\in E\cap\omega_0E$.

$\bullet$ If $X\not=Y$, we can apply Lemma \ref{lem:separate}\,(4) with $E=\ol{X\cup Y}, A=X\cup\{x\}\cup Y,\omega=e^{\pi \tb{i}/n}$ and $\omega_0$ to obtain that $\omega_0E\cap E=\emptyset$. This contradicts $\omega_0 z\in E\cap\omega_0 E$.

If $\omega_0$ is primitive, then $\omega_0^n=-1$, so that $f(z)\in f(\overline{X})$ and $f(\omega_0z)=-f(z)\in  f(\overline{Y})$.  Replacing $(X,Y,z,\omega_0)$ by $(f(X),f(Y),f(z),-1)$ in the previous argument then gives a contradiction. This completes the proof of Statement (1).\vspace{2pt}

\medskip\noindent (2) Assume $n$ is odd. Suppose for contradiction that the statement fails. Then there exist a $2n$-th root of unity $\omega_0\neq\pm1$, and a point $z\in\overline{X}$ such that $\omega_0 z\in\overline{Y}$. Here $X$ and $Y$ are constructed similarly in (1). The previous claim implies 
\begin{equation}\label{eq:25}
e^{k\pi \tb{i}/n}\EEE_\infty^m\cap\EEE_\infty^m=\emptyset \text{ for } k\in\{1,\ldots,2n-1\}\setminus\{n\} \text{ and } \EEE_\infty^m=-\EEE_\infty^m.
\end{equation}
 
 Then by relation \eqref{eq:25} and Proposition \ref{pro:arc}, one of the following two cases occurs:

Case 1. $X=Y$. Since $X\neq-X$ by the claim, combining with Proposition \ref{pro:arc}, we apply Lemma \ref{lem:separate}\,(4) with $E=\ol{X}, A=X,\omega=e^{\pi \tb{i}/n}$ and $\omega_0$ to obtain that $\omega_0E\cap E=\emptyset$. This contradicts $\omega_0 z\in E\cap\omega_0E$.

Case 2. $X\not=Y$. Clearly $X\neq -Y$, for otherwise, by replacing $\omega_0$ with $-\omega_0$, the above case occurs.
We may replace $(X,Y,\omega_0)$ by $(-X,-Y,-\omega_0)$ if necessary and assume further that $\omega_0^n=1$. Then by symmetry $f(\omega_0 z)=\omega_0^nf(z)=f(z)$. As in (1), $\overline{X}\cap\overline{Y}$ contains a point $x$. 
Note that the four sets $\pm X$ and $\pm Y$ are pairwise disjoint. We apply Proposition \ref{pro:arc} and Lemma \ref{lem:separate}\,(4) with $E=\ol{X\cup Y}, A=X\cup\{x\}\cup Y,\omega=e^{\pi \tb{i}/n}$ and $\omega_0$ to obtain that $\omega_0E\cap E=\emptyset$. This contradicts $\omega_0 z\in E\cap\omega_0E$.
This completes the proof of Statement (2).\vspace{2pt}

\medskip\noindent (3) Assume $n$ is odd. By relation \eqref{eq:25}, $\EEE^{m+1}=-\EEE^{m+1}$. Moreover, if $E$ is a component of $\EEE^{m+1}$, so is $-E$. If $E=-E$, then $E$ separates $0$ and $\infty$. Choose a $2n$-th root of unity $\omega_0\neq\pm1$. Lemma \ref{lem:separate}\,(2) gives $E\cap\omega_0E\neq\emptyset$, contradicting Statement (2). This completes the proof of Statement (3) and hence of Lemma \ref{lem:Fatou-cycle}.
\end{proof}

\section{Blow-up Theorems}\label{sec3}
In this section, we first recall the blow-up theorem for exact subsystems established in \cite[Theorem~1.6]{CGZ}. We then prove a new blow-up theorem for McMullen maps whose two free critical values lie on $\pa T$, blowing up the corresponding Julia sets to Cantor-circles. We also complete the proof of Theorem~\ref{thm1} here.

\subsection{Blow up of exact subsystems}
Let $f$ be a PCF rational map. Recall that  $P_f$ denotes its postcritical set
$\{f^k(c): c \tu{ is a critical point of $f$ and $k\geq 1$}\}$.

Suppose $\VVV_1\subset\VVV$ are open sets such that $\partial\VVV$ is contained in the Julia set of $f$ and each component of $\partial\VVV$ contains more than one point. We say that $f:\VVV_1\to\VVV$ is an {\bf exact subsystem} if
\begin{itemize}
\item [(1)]$ \VVV$ has finitely many components, each of which is finitely connected;

\item [(2)] $\VVV_1$ is the union of some components of $f^{-1}(\VVV)$;

\item [(3)] each component of $\VVV\setminus\VVV_1$ is a full continuum disjoint from $P_f$.
\end{itemize}

\begin{thm}[\textup{\cite[Theorem 1.6]{CGZ}}]\label{thm:blow-upI}
Let $f$ be a PCF rational map and let $f:\VVV_1\to\VVV$ be an exact subsystem such that $\mc{V}$ is connected. Denote
$$
\VVV_n=(f|_{\VVV_1})^{-n}(\VVV)\quad\text{ and }\quad E=\bigcap_{n\geq1}\ol{\VVV_n}.
$$
Then there exists a PCF rational map $g$, a continuum $K_g\supset J_g$ with $g^{-1}(K_g)=K_g$, and a quotient map $\pi:\overline{\mathbb{C}}\to\overline{\mathbb{C}}$ such that
\begin{enumerate}
\item  components of $\overline{\mathbb{C}}\setminus K_g$ are all Jordan domains with pairwise disjoint closures;
\item  $E=\pi(K_g)$ and $f\circ\pi=\pi\circ g$ on $K_g$;
\item  for any point in $\bigcap_{k>0}\VVV_k$, the fiber $\pi^{-1}(z)$ is a single point;
\item  for any component $B_n$ of $\overline{\mathbb{C}}\setminus \VVV_n$, the set $\pi^{-1}(B_n)$ is the closure of a component of $\overline{\mathbb{C}}\sm K_g$;

\item  a point $x$ belongs to $P_g$ if and only if either $\pi(x)\in P_f\cap\VVV$, or $x$ is the B\"{o}ttcher center of a component $D$ of $\overline{\mathbb{C}}\setminus K_g$ with $\pi(\ol{D})\cap P_f\neq\emptyset$. 
\end{enumerate}
Moreover, the map $g$ is unique up to M\"{o}bius conjugacy.
\end{thm}

Intuitively, this theorem constructs a new rational map $g$ whose Julia set $J_g$ is a regularized version of the set $E$. Suppose $f$ admits no attracting cycles in $\mathcal{V}_1$. Then $E=J_f$ and $K_g=J_g$ is a Sierpi\'{n}ski carpet.
The quotient map $\pi$ matches $J_g$ with $J_f$ and preserves the dynamics. It is injective at the buried point set; the only collapsed points are boundary points from Fatou domains.
Consequently, one can first study the simpler dynamics of $g$ on $J_g$ and then transport the results back to $f$ via $\pi$.

\medskip The blow-up theorem (Theorem~\ref{thm:blow-upI}) applies to any PCF rational map admitting an exact subsystem. To use it for McMullen maps, we need to exhibit a set $\mc{E}_0$ that is stable (i.e., $f(\mc{E}_0)\subseteq\mc{E}_0$ and every component of $\mc{E}_0$ is a component of $f^{-1}(\mc{E}_0)$). Two cases provide such an $\mc{E}_0$.

\smallskip{\textbf{Case (S1).} No bounded attracting cycles, and $f^k(v)\in\partial T$ for some $k\ge1$.}
By Proposition~\ref{prop:classify-no-bounded-cycles}(4), the components of $T\cup f^{-1}(T)\cup\cdots\cup f^{-k}(T)$ are Jordan domains with pairwise disjoint closures; the critical values $\pm v$ lie on boundaries of two of them, say $U_1$ and $U_2$, and the preimage of each $\overline{U_i}$ is a full continuum whose interior consists of two Jordan domains. We define $\mc{E}_0=\bigcup_{k\geq 0}\bigcup_{i=1,2}f^k(\overline{U_i}).$ Then $\mc{E}_0$ contains $P_f$.

\smallskip{\textbf{Case (S2).} Bounded attracting cycles, with the maximal Fatou chain $\mathcal{E}$ disjoint from $\partial B$.}
By Lemma~\ref{classification}, each component of $\mathcal{E}$ is a full continuum containing at most one critical value. Then components of $f^{-1}(\mc{E})$ are full. Set $\mathcal{E}_0$ to be the union of $\mathcal{E}$, $\overline{B}$, and (when $n$ is even) the preimage of $\mathcal{E}$ under $f$. Then $\mathcal{E}_0$ contains $P_f$.

\smallskip
In both cases, $\mc{E}_0$ is stable, and every component of $\mc{E}_0$ or $f^{-1}(\mc{E}_0)$ is a full continuum.
Now we define $\mc{V}=\ol{\mb{C}}\setminus\mc{E}_0$ and $\mc{V}_1=\ol{\mb{C}}\setminus f^{-1}(\mc{E}_0)$. Then the restriction $f:\mc{V}_1\to\mc{V}$ is an exact subsystem of $f$ by definition. Applying Theorem \ref{thm:blow-upI} yields the following blow-up theorem for McMullen maps.

Recall that a point $z\in E\subseteq\overline{\mathbb{C}}$ is called buried in $E$ if it does not lie on the boundary of any component of $\overline{\mathbb{C}}\setminus E$.

\begin{cor}\label{cor:sier}
    Let $f$ be a PCF McMullen map and $J$ be its Julia set. Suppose $f$ satisfies Case (S1) or Case (S2). Then $J$ is a Sierpi\'{n}ski-like carpet. Moreover, there exists a Sierpi\'{n}ski rational map $g$ and a quotient map $\pi$ over $\ol{\mb{C}}$ such that
\begin{enumerate}
    \item $\pi(J_g)=J$ and $f\circ\pi=\pi\circ g$ on $J_g$;
    \item $\pi$ is injective on the buried point set of $J_g$;
    \item For any Fatou domain $U$ of $g$, any two points in $\pi(\ol{U})$ can be joined by a curve that meets $J$ in at most countably many points.
\end{enumerate}
\end{cor}
\begin{proof}
    If Satement~(2) is false, then there is a buried point $z$ in $J_g$ such that $\pi^{-1}(\pi(z))$ is a full continnum which intersects at least two buried points in $J_g$, contradicting Theorem \ref{thm:blow-upI}~(3).

    For Statement~(3), by Theorem~\ref{thm:blow-upI}(4), $\pi(\ol{U})$ is the closure of a component of $\overline{\mathbb{C}}\setminus \VVV_k$ for some $k\ge1$, hence a component of $f^{-k}(\mc{E}_0)$, say $E$. 
    
    Suppose $E$ is periodic. In case (S1), $E$ is the closure of a Fatou domain, so any two points of $E$ can be joined by a curve whose interior lies in the Fatou domain. In case (S2), if $E$ is a component of a maximal Fatou chain, then the required property follows from Lemma~\ref{classification}(2). By pullback, Statement~(3) holds. 
\end{proof}
\subsection{\boldmath Blow up of McMullen maps with $v\in\pa T$\unboldmath }

We now turn to Case~(3) of Proposition~\ref{prop:classify-no-bounded-cycles}, where $v\in\partial T$. As shown there, this forces $n\ge3$. In this subsection we shall prove that the Julia set is a necklace by blowing it up to a Cantor-circle.

Fix an integer $n\ge3$. Set $I=[0,1]$, $I_0=[0,1/n]$, $I_1=[1-1/n,1]$, $\mathbb{S}^1=\mathbb{R}/2\pi\mathbb{Z}$, and let $\sigma_i:I_i\to I$ be the affine homeomorphism preserving orientation for $i=1$ and reversing for $i=0$. The {\bf model map} 
$g:(I_0\sqcup I_1)\times\mathbb{S}^1\longrightarrow I\times\mathbb{S}^1$
is defined by $$g(x,t)=(\sigma_i(x),(-1)^{i+1}n\cdot t~\tu{mod}~2\pi)$$ for $(x,t)\in I_i\times\mathbb{S}^1$. Its {\bf Julia set},  defined here as the non-escaping set, is a Cantor set of circles, or simply a Cantor-circle.

Set $\mathcal{A}_0=(1/n,\,1-1/n)\times\mathbb{S}^1$ and $\mathcal{A}_k=g^{-k}(\mathcal{A}_0)$ for $k\ge1$. Then $\mathcal{A}_k$ consists of $2^k$ annuli, each of width $(1-2/n)n^{-k}$. The Julia set of $g$ is
$$J_g=\mathcal{C}\times\mathbb{S}^1=I\times\mathbb{S}^1\setminus\bigcup_{k\ge0}\mathcal{A}_k,$$
a Cantor-circle. 

Recall that $\phi_B:B\to\ol{\mathbb{C}}\setminus\overline{\D}$ and $\phi_T:T\to \mathbb{D}$ are B\"{o}ttcher maps of $B$ and $T$, respectively. Since $\partial B$ and $\partial T$ are Jordan curves, these two maps extend homeomorphically to  $\overline{B}$ and $\overline{T}$, still denoted by $\phi_B$ and $\phi_T$, respectively. 
%A continuous onto map $\pi$ of $\overline{\mathbb{C}}$ is called a \textbf{quotient map} if $\pi^{-1}(z)$ is a singleton or a full continuum for all $z\in\overline{\mathbb{C}}$.

\begin{thm}\label{thm:blow-upII}
Let $f$ be a PCF McMullen map with $v\in\pa T$ and $J$ be its Julia set. Then $J$ is a necklace. Moreover, there exists a continuous onto map $\pi:I\times \mathbb{S}^1\to \overline{\mathbb{C}}\setminus (B\cup T)$ such that:
\begin{enumerate}
\item $\pi(J_g)=J$ and $f\circ\pi=\pi\circ g$ on $J_g$;
\item for any $z\in J$ whose orbit avoids $\{\pm v\}$, $\pi^{-1}(z)\cap J_g$ is a singleton;
\item for any $z\in J$ with $f^k(z)\in\{\pm v\}$ for some $k\ge 1$, $\pi^{-1}(z)\cap J_g$ consists of exactly two points; 
\item $\pi(0,t)=\phi_T^{-1}(e^{\textup{\tb{i}}t})$ and $\pi(1,t)=\phi_B^{-1}(e^{\textup{\tb{i}}t})$ for all $t\in\mathbb{S}^1$.
\end{enumerate}
\end{thm}
\noindent\tb{Remark.} One can show that the map $\pi$ in Theorem~\ref{thm:blow-upII} is uniquely determined. Since this uniqueness is not needed in what follows, we omit the proof.

\begin{proof}
Let $\mathcal{D}$ be the union of two small open disks centered at $\pm v$, respectively, such that $\mathcal{D}\cap\partial T$ consists of two open arcs. Choose a homeomorphism $\alpha:\overline{\mathcal{D}}\to\overline{\mathcal{D}}$ with $\alpha|_{\partial\mathcal{D}}=\textup{id}$ that sends both $\pm v$ into $T\cap\mathcal{D}$.

Modify $f$ near the critical points to obtain a branched covering $f_*$ as follows. Let $\mathcal{D}_1=f^{-1}(\mathcal{D})$ and define
\[
f_*(z)=\begin{cases}
f(z), & z\in\overline{\mathbb{C}}\setminus\mathcal{D}_1,\\[4pt]
\alpha\circ f(z), & z\in\mathcal{D}_1.
\end{cases}
\]
Then  $f_*$ is a branched covering and differs from $f$ only in $\DDD_1$, and all critical points of $f_*$ escape to infinity. As in the case of rational maps, we denote by $P_{f_*}$ the closure of the union of the critical orbits of $f_*$.   

Set $V^*_0=V_0=\overline{\mathbb{C}}\setminus(\overline{B}\cup\overline{T})$. Then $V^*_1:=f_*^{-1}(V^*_0)$ and $V_1:=f^{-1}(V_0)$ each consists of two annuli nested in $V^*_0$ and $V_0$, respectively; see Figure \ref{fig:xi1}.

Let $\xi_0=\textup{id}$ on $\overline{\mathbb{C}}$. We will construct a quotient map $\xi_1$ such that $\xi_0$ and $\xi_1$ are homotopic rel $\overline{\mathbb{C}}\setminus V_0^*$ and
\[
\xi_0\circ f_* = f\circ\xi_1 \quad\text{on } V_1^*.
\]

Define $\xi_1(z)=\xi_0(z)$ for $z\in\overline{\mathbb{C}}\setminus\mathcal{D}_1$. A component $D^*$ of $\mathcal{D}_1^*$ meets $V_1^*$ in two disks $D_+^*,D_-^*$, and $D_0^*=D^*\setminus\overline{V_1^*}$ is also a disk. Similarly, the corresponding component of $\mathcal{D}_1$ yields $D_+,D_-,D_0$, where $D_0$ consists of two disjoint disks in Fatou domains.  For $i\in\{+, -\}$, define $\xi_1:D^*_i\to D_i$ by
$
\xi_1(z) = (f|_{D_i})^{-1}\circ\xi_0\circ f_*(z).
$ 

\begin{figure}[H]
  \centering
  \includegraphics[width=15cm]{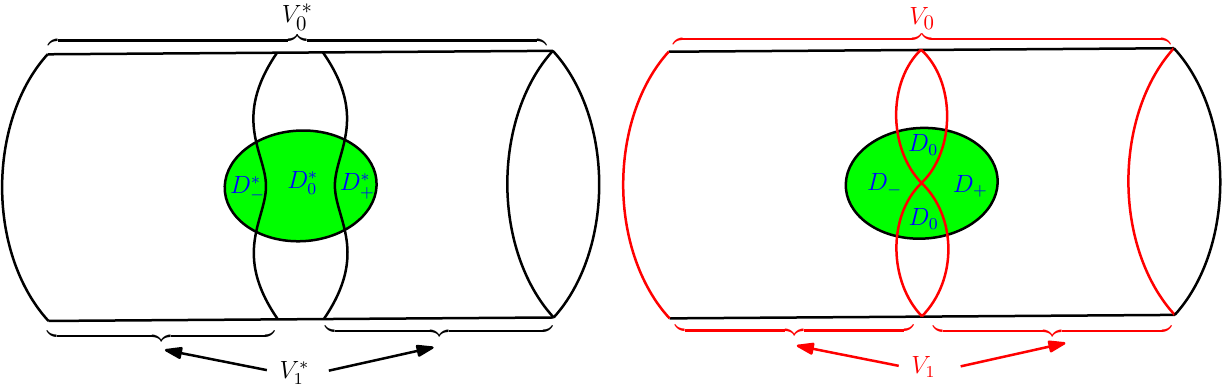}
  \caption{The construction of $\xi_1$}\label{fig:xi1}
\end{figure}

Extend $\xi_1$ to $D_0^*$ as a quotient map: $\xi_1:\overline{D_0^*}\to\overline{D_0}$ sends an arc joining two points on $\partial V_1^*$ to the unique critical point and is injective elsewhere. 

By construction, $\xi_0\circ f_* = f\circ\xi_1$ on $V_1^*$. Moreover, there exists a homotopy $\xi_t:\overline{\mathbb{C}}\to\overline{\mathbb{C}}$, $t\in[0,1]$, such that each $\xi_t$ is a quotient map and $\xi_t(z)=\xi_0(z)$ for all $z\in\overline{\mathbb{C}}\setminus V^*_0$.

\medskip\noindent\textbf{Claim 1.} \emph{There exist a neighborhood $\mathcal{N}$ of infinity,  an isotopy $\phi_t:\overline{\mathbb{C}}\to\overline{\mathbb{C}}$, $t\in [0,1]$ rel. $\NNN\cup P_{f_*}$ and a ratoinal map $\tilde{f}$, such that $\phi_0$ is holomorphic in $\NNN$ and 
$\phi_0\circ f_*=\tilde{f}\circ\phi_1\textup{ on } \overline{\mathbb{C}}.$
Moreover, the Julia set of $\tilde{f}$ is a Cantor-circle.}
\begin{proof}[Proof of Claim 1]
By \cite[Theorem 1.1]{CT}, to prove the existence of $\phi_t$ and $\tilde{f}$, it suffices to check that $f_*$ has no Thurston obstruction. 

Recall that a multicurve is a non-empty finite collection of non-peripheral Jordan curves in $\overline{\mb{C}}\setminus P_{f_*}$ that are pairwise disjoint and  pairwise non-homotopic rel $P_{f_*}$. A multicurve $\gamma$ is called {$f_*$-stable} if each non-peripheral component of the preimage $f_*^{-1}(\Gamma)$ of a curve $\gamma\in\Gamma$ is homotopic rel $P_{f_*}$ to a curve in $\Gamma$. A multicurve is called irreducible if for any two curves $\gamma_i, \gamma_j \in \Gamma$, there exists an integer $k \ge 1$ such that $f_*^{-k}(\gamma_j)$ has a component homotopic to $\gamma_i$ rel $P_{f_*}$. Finally, a {Thurston obstruction} is an $f_*$-stable multicurve whose transition matrix has leading eigenvalue at least $1$ (see \cite{CT}).

By a standard reduction argument, it suffices to check that no irreducible $f_*$-stable multicurve is a Thurston obstruction. Indeed, any Thurston obstruction contains an irreducible stable sub-multicurve with the same leading eigenvalue.

Let $\mathcal{N}$ be a disk containing $\infty$ with $f_*(\mathcal{N})\Subset \mathcal{N}$. For each $k$, let $\NNN_k$ denote the union of the two components of $f^{-k}(\NNN)$ containing $0$ and $\infty$. Then there exists an integer $m>0$ such that $\NNN_m$ consists of two disks in $B$ and $T$ respectively, and contains $P_{f_*}$.

Let $\Gamma$ be an irreducible multicurve in $\overline{\mathbb{C}}\setminus P_{f_*}$. Since $f(\mathcal{N}) \Subset \mathcal{N}$ is an attracting neighborhood of $\infty$, the forward iterates $f_*^k(\mathcal{N}_m)$ converge to $\{\infty\}$ as $k\to\infty$. Thus there exists $k_0>0$ such that $\Gamma\cap f_*^{k_0}(\mathcal{N}_m)=\emptyset$. Among all multicurves isotopic to $\Gamma$ rel $P_{f_*}$, we choose a representative (still denoted $\Gamma$) such that for each $k=0,\dots,k_0$, the intersection $\Gamma\cap f_*^k(\mathcal{N}_m)$ has the minimal number of components. This \emph{minimality} means that no isotopic deformation of $\Gamma$ can further reduce the number of components of these intersections.

We claim that this minimal choice forces $\Gamma$ to be entirely contained in the annulus $\overline{\mathbb{C}}\setminus \overline{\mathcal{N}_m}$. Suppose otherwise, then some curve $\gamma\in\Gamma$ meets $\mathcal{N}_m$. Consider a component $\tilde{\gamma}$ of $f_*^{-1}(\Gamma)$ that is isotopic to $\gamma$ rel $P_{f_*}$. By the minimality of the intersection number, $\tilde{\gamma}$ must also intersect $\mathcal{N}_m$, otherwise we could isotope $\Gamma$ to reduce the number of components. It follows that $\gamma_1=f_*(\tilde{\gamma})\in\Gamma$ intersects $f_*(\mathcal{N}_m)$.

Applying the same argument to $\gamma_1$ and using the minimal intersection with $f_*(\mathcal{N}_m)$, we obtain a curve $\gamma_2\in\Gamma$ intersecting $f_*^2(\mathcal{N}_m)$. Inductively, for every $k\ge1$ we find $\gamma_k\in\Gamma$ intersecting $f_*^k(\mathcal{N}_m)$. Taking $k=k_0$ contradicts $\Gamma\cap f_*^{k_0}(\mathcal{N}_m)=\emptyset$. Therefore, $\gamma\cap\mathcal{N}_m=\emptyset$ for every curve $\gamma\in\Gamma$, as claimed.

Consequently, $\Gamma$ lies in the annulus $\overline{\mathbb{C}}\setminus \overline{\mathcal{N}_m}$. Because $P_{f_*}\subset \mc{N}_m$, any non-peripheral curve in this annulus is isotopic to the core curve. Hence $\Gamma$ consists of a single curve $\gamma$, and $f_*^{-1}(\gamma)$ has exactly two components $\gamma_1,\gamma_2$, each isotopic to $\gamma$ and mapped onto $\gamma$ as a covering of degree $n\ge 3$. The Thurston matrix is therefore the $1\times 1$ matrix $(2/n)$, whose unique eigenvalue is $2/n<1$. Thus, $\Gamma$ is not a Thurston obstruction.

Since no irreducible $f_*$-stable multicurve can be a Thurston obstruction, $f_*$ has no Thurston obstructions. By \cite[Theorem 1.1]{CT}, there exist a rational map $\tilde{f}$ and two isotopic homeomorphisms $\phi_0,\phi_1$ rel $\mathcal{N}\cup P_{f_*}$ such that $\phi_0\circ f_*=\tilde{f}\circ\phi_{1}$ on $\overline{\mathbb{C}}$. 

Using the isotopy lifting theorem for $\tilde{f}$ and $f_*$, we obtain a sequence of homeomorphisms $\{\phi_k\}_{k\ge0}$ such that $\phi_k$ is isotopic to $\phi_{k+1}$ rel $f_*^{-k}(\mathcal{N})\supseteq \mathcal{N}_k$ and $\phi_k\circ f_*=\tilde{f}\circ\phi_{k+1}$ on $\overline{\mathbb{C}}$.  
Set $\mathcal{W}=\phi_{m+1}(\mathcal{N}_m)$ and $\mathcal{W}_1=\phi_{m+1}(\mathcal{N}_{m+1})$. Then $\mathcal{W}\Subset\mathcal{W}_1$ and $\tilde{f}(\mathcal{W}_1)\subseteq\mathcal{W}$.  

Let $A=\overline{\mathbb{C}}\setminus\overline{\mathcal{W}}$; this annulus is disjoint from $P_{\tilde{f}}$. Its preimage $\tilde{f}^{-1}(A)$ consists of two annuli compactly contained in $A$ that separate the two components of $\partial A$. Hence, by \cite[Proposition 2.2]{HP12}, the Julia set of $\tilde{f}$ is a Cantor-circle.
\end{proof} 
We may specify the isotopy $\phi_t$ in the claim such that $\phi_0$ is holomorphic in $B\cup T$ and $\phi_t=\phi_0$ on $B\cup T$ for $t\in[0,1]$. Indeed, both $f_*$ and $\tilde{f}$ are conformally conjugate to $z\mapsto z^n$ on their immediate attracting basins. Claim 1 gives a local conjugacy $\phi_m$ that is holomorphic on $\mc{N}_m$; by the uniqueness of B\"ottcher coordinates, this local conjugacy extends uniquely to the whole basin $B$, and likewise to $T$. We then replace $(\phi_0,\phi_1)$ by a representative in the same isotopy class that is holomorphic on $B\cup T$.

Let $\Omega_0=\phi_0(V_0^*)$ and $\Omega_1=\tilde{f}^{-1}(\Omega_0)$.  
Define $h_t=\xi_t\circ\phi_t^{-1}$ for $t\in[0,1]$. Then $\{h_t\}_{t\in[0,1]}$ is a homotopy on $\overline{\mathbb{C}}$ rel $\overline{\mathbb{C}}\setminus\Omega_0$ satisfying $h_0\circ\tilde{f}=f\circ h_1$ on $\Omega_1$.

Since both $\tilde{f}:\Omega_1\to\Omega_0$ and $f:V_1\to V_0$ are coverings, the general homotopy lifting theorem lifts $h_t:\Omega_0\to V_0$ rel $\partial\Omega_0$ to a homotopy $h_t:\Omega_1\to V_1$ rel $\partial\Omega_1$ for $t\in[1,2]$. This homotopy extends to all of $\overline{\mathbb{C}}$ for $t\in[1,2]$, with each $h_t$ a quotient map and $h_t(z)=h_1(z)$ for $z\in\overline{\mathbb{C}}\setminus\Omega_1$.

By induction, we obtain a sequence of homotopies $\{h_t\}_{t\in[m,m+1]}$ on $\overline{\mathbb{C}}$ for all $m\ge0$. Set $\Omega_m=\tilde{f}^{-m}(\Omega_0)$ and $V_m=f^{-m}(V_0)$. Then for each $m$ the following hold:

\begin{enumerate}
\item $h_t:\overline{\mathbb{C}}\to\overline{\mathbb{C}}$ is a quotient map with $h_t(\Omega_m)=V_m$;
\item $h_t(z)=h_m(z)$ for all $z\in\overline{\mathbb{C}}\setminus\Omega_m$ and $t\in[m,m+1]$;
\item $h_m\circ\tilde{f}=f\circ h_{m+1}$ on $\Omega_{m+1}$;
\item $h_m^{-1}(z)$ is a single point if and only if $z\in\overline{\mathbb{C}}\setminus\bigcup_{1\le k\le m}f^{-k}(\pm v)$;
\item $h_m^{-1}(z)$ is an arc in $\overline{\mathbb{C}}\setminus\Omega_m$ joining two points on distinct boundary components of $\Omega_m$ if and only if $z\in\bigcup_{1\le k\le m}f^{-k}(\pm v)$.
\end{enumerate}

\noindent\textbf{Claim 2.} \emph{The sequence $\{h_m\}$ converges uniformly to a quotient map on $\overline{\mathbb{C}}$.}

\begin{proof}[Proof of Claim 2]
The argument follows the same pattern as in \cite[Theorem~1.1]{CPT}. By \cite[Lemma~3.1]{CPT}, a uniform limit of quotient maps is again a quotient map, so it suffices to find constants $M>0$ and $\rho>1$ such that
\[
\operatorname{dist}(h_{m+1}(z),h_m(z))\le M\rho^{-m}
\quad\text{for all }m\ge1\text{ and }z\in\overline{\mathbb{C}}.
\]

The {homotopic length} of a curve $\gamma$ is the infimum of the lengths of smooth curves homotopic to $\gamma$ rel $P_f$ with fixed endpoints, measured in the orbifold metric (see \cite[Appendix~A.1]{CGZ}).

For $z\in\Omega_0$, set $\gamma_z(t)=h_t(z)$, $t\in[0,1]$. The homotopic length of $\gamma_z$ varies continuously with $z$ and tends to $0$ as $z\to\partial\Omega_0$, so it is bounded by some constant $M_1$ for all $z\in\Omega_0$.

Fix $m\ge1$ and $z\in\overline{\mathbb{C}}$. If $z\notin\Omega_m$, then $\operatorname{dist}(h_m(z),h_{m+1}(z))=0$. If $z\in\Omega_m$, put $w=f^m(z)\in\Omega_0$. Then the path $\beta(t)=h_t(z)$, $t\in[m,m+1]$, is a lift of $\gamma_w$ under $f^m$ starting at $h_m(z)$. Hence,
\[
\operatorname{dist}(h_m(z),h_{m+1}(z))\le C\,L_\omega[\beta]\le CM_1\rho^{-m}
\]
for some constant $C>0$ independent of $m$ and $z$, by \cite[Appendix~A.3 and Lemma~A.1]{CGZ}. This completes the proof of Claim~2.
\end{proof}

Let $h$ be the uniform limit of $\{h_m\}$. Since $J_{\tilde{f}}=\bigcap_{m\ge0}\overline{\Omega_m}$ and $J=\bigcap_{m\ge0}\overline{V_m}$, we have $h(\overline{\Omega_m})=\overline{V_m}$ and $h(\partial\Omega_m)=\partial V_m$ for every $m\ge0$. Hence $h(J_{\tilde{f}})\subseteq J$, and surjectivity yields $h(J_{\tilde{f}})=J$. Moreover, the properties of $h_m$ imply $h\circ\tilde{f}=f\circ h$ on $J_{\tilde{f}}$.

Now consider $\partial B$. By the properties of $h$, the preimage $h^{-1}(\partial B)$ lies in $J_{\tilde{f}}$ and contains $\partial h_0^{-1}(B)$. Take any $z\in h^{-1}(\partial B)\setminus\partial h_0^{-1}(B)$. Since $J_{\tilde{f}}$ is a Cantor-circle, for all sufficiently large $m$ there are two components $H,H'$ of $\overline{\mathbb{C}}\setminus\Omega_m$ separating $z$ from $\partial h_0^{-1}(B)$. The map $h$ agrees with $h_m$ on $H\cup H'$, and $h_m(H)\cap h_m(H')=\emptyset$, contradicting $h(z)\in\partial B$. Therefore $\partial h_0^{-1}(B)=h^{-1}(\partial B)$. Moreover, the expanding property of $\tilde{f}$ and $f$ implies that $h$ is injective on $\partial h^{-1}(B)$. Pulling back, we obtain $h^{-1}(\partial V_m)\cap J_{\tilde{f}}=\partial\Omega_m$ and $h$ is injective on each component of $\partial V_m$ for every $m\ge1$. Thus Statement~(3) follows from property~(5) above.

Take any $z\in\bigcap_{m\ge0}V_m$. Its preimage $h^{-1}(z)\subseteq\bigcap_{m\ge0}\Omega_m$ is a full connected compact set contained in a component of $J_{\tilde{f}}$ (a circle). If $h^{-1}(z)$ were not a single point, it would be an arc and would eventually map onto a whole circle in $J_{\tilde{f}}$. Its image under $h$ would then be a single point in $J$, impossible. Therefore $h^{-1}(z)$ is a singleton.

Finally, $g$ and $\tilde{f}$ are quasiconformally conjugate in an annular neighborhood of $h^{-1}(\partial B)$.  
The dynamics $\tilde{f}:\tilde{f}^{-1}(A)\to A$ (defined in Claim 1) and $g$ are combinatorially equivalent  in the sense of \cite[Appendix A]{McM98}. Hence, by \cite[Theorem A.1]{McM98}, there exists a quasisymmetric conjugacy $\psi:J_g\to J_{\tilde{f}}$.  
Then $\pi=h\circ\psi:J_g\to J$ is the required semi-conjugacy, completing the proof of the theorem.
\end{proof}

\subsection{Proof of Theorem \ref{thm1}}\label{sec4}
The blow-up theorems above reduce complicated Julia sets to two simple models. Together with the properties of Fatou chains from Section~\ref{Fatou chains}, this gives the full classification of Theorem~\ref{thm1}, which we now prove.
\begin{proof}[Proof of Theorem~\ref{thm1}]
Let $f$ be a PCF McMullen map. Recall that $B$ and $T$ are the Fatou domains containing $\infty$ and $0$ respectively, and both are Jordan domains by Lemma~\ref{lem:jordan-domains}. We distinguish two main cases.

\medskip\noindent{\textbf{Case~1.} $f$ has no bounded attracting cycles.}
Then Proposition~\ref{prop:classify-no-bounded-cycles} classifies the dynamics into five subcases, which translate directly into the four topological classes:

\begin{itemize}
\item[(i)] $f^k(v)=\infty$ for some $k\ge2$, or $v$ is buried in $J$. Then $J$ is a Sierpi\'nski carpet (Proposition~\ref{prop:classify-no-bounded-cycles}(1), (5)).
\item[(ii)] $v\in\partial B$. Then $\partial B\cap\partial T=\Crit$ by Lemma~\ref{lem:disjoint}, so $J$ is a cluster (Proposition~\ref{prop:classify-no-bounded-cycles}(2)).
\item[(iii)] $v\in\partial T$. In this case $n\ge3$ and Theorem~\ref{thm:blow-upII} blows $J$ up to a Cantor-circle; hence $J$ is a necklace (Proposition~\ref{prop:classify-no-bounded-cycles}(3)).
\item[(iv)] $f^k(v)\in\partial T$ for some $k\ge1$. Corollary \ref{cor:sier} blows $J$ to a Sierpi\'nski carpet, so $J$ is a Sierpi\'nski-like carpet.
\end{itemize}

\medskip\noindent{\textbf{Case~2.} $f$ has bounded attracting cycles.}
Let $\mathcal{U}$ be the union of the immediate basins of all bounded attracting cycles, and let $\mathcal{E}$ be the maximal Fatou chain generated by $\mathcal{U}$ (see Section~\ref{sec:fatouchain}). Lemma~\ref{classification} provides the following key facts: each component of $\mathcal{E}$ is a full continuum; any two points in a component can be joined by a curve meeting $J$ in at most countably many points; and if $\mathcal{E}\cap\partial B\neq\emptyset$, then each component of $\mathcal{E}$ meets $\partial B$ at a single point.

\begin{itemize}
\item[(i)] $\mathcal{E}\cap\partial B=\emptyset$ and $\mathcal{E}=\overline{\mathcal{U}}$. Then $\mathcal{E}$ has finitely many components, each a Jordan domain, and the closures of any two are disjoint. Every other Fatou domain eventually maps to $\mathcal{E}$ or to $B$, and $P_f\cap J=\emptyset$. Hence, every Fatou domain is a Jordan domain with pairwise disjoint closures, and $J$ is a Sierpi\'nski carpet.
\item[(ii)] $\mathcal{E}\cap\partial B=\emptyset$ but $\mathcal{E}\neq\overline{\mathcal{U}}$. Then by Corollary~\ref{cor:sier}, $J$ can be blown up to a Sierpi\'nski carpet, hence is a Sierpi\'nski-like carpet.
\item[(iii)] $\mathcal{E}\cap\partial B\neq\emptyset$. By Lemma~\ref{classification}(3), each component of $\mathcal{E}$ meets $\partial B$ at a single point. Since $f^{-1}(B\cup\mathcal{E})=B\cup T\cup f^{-1}(\mathcal{E})$, any component of $f^{-1}(\mathcal{E})$ containing a critical point meets both $\partial T$ and $\partial B$, as $f^{-1}(B)=B\cup T$ and the critical point is a branch point between these two preimages. Lemma~\ref{classification}(2) then yields a curve inside that component joining a point of $\partial B$ to a point of $\partial T$ and passing through at most countably many Julia points. Thus, $J$ is a cluster.
\end{itemize}

In every subcase, $J$ falls into exactly one of the four topological classes, and the necklace case occurs only for $n\ge3$. This completes the proof.
\end{proof}

\section{Dynamical relations for Sierpi\'nski-like carpets}\label{sec5}
In this section we prove Theorem~\ref{thm:sierpinski-like}, which gives the dynamical relation forced by quasisymmetries between Sierpi\'nski-like carpet Julia sets. The statement covers general PCF rational maps with Sierpi\'nski models; for McMullen maps, Corollary~\ref{cor:sier} reduces it to the corresponding dynamical relation.

%\subsection{Sierpi\'{n}ski-like carpet case}

%\begin{thm}\label{thm:sierpinski-like}
%Let \( J \) and \( \tilde{J} \) be Sierpi\'{n}ski-like Julia sets of  PCF McMullen rational maps \( f \) and \( \tilde{f} \) with the same exponent $n$, respectively. Suppose \( \xi: J \to \tilde{J} \) is a quasisymmetry. Then there exist integers \( l \geq 1 \) and \( m \geq 0 \) such that
%\[
%\tilde{f}^{m+l} \circ \xi = \tilde{f}^{m} \circ \xi \circ f^l \quad \text{on } J.
%\]
%\end{thm}
\begin{thm}\label{thm:sierpinski-like}
Let \( J \) and \( \tilde{J} \) be the Julia sets of PCF rational maps \( f \) and \( \tilde{f} \), respectively. Suppose there exist Sierpi\'{n}ski rational maps $g$ and $\tilde{g}$, and quotient maps $\pi$ and $\tilde{\pi}$ over $\overline{\mathbb{C}}$ such that
\begin{enumerate}
    \item $\pi\circ g=f\circ \pi$ on $J_g$ and $\tilde{\pi}\circ \tilde{g}=\tilde{f}\circ\tilde{\pi}$ on $J_{\tilde{g}}$;
    \item $\pi$ and $\tilde{\pi}$ are injective on the buried point sets;
    \item for any Fatou domain $U$ of $g$, any two points in $\pi(\ol{U})$ can be joined by a curve that meets $J$ in at most countably many points; the analogous statement holds for $\tilde{g}$ and $\tilde{\pi}$.
\end{enumerate}
Suppose \( \xi: J \to \tilde{J} \) is a quasisymmetry. Then there exist integers \(m', l \geq 1 \) and \( m \geq 0 \) such that
\[
\tilde{f}^{m'} \circ \xi = \tilde{f}^{m} \circ \xi \circ f^l \quad \text{on } J.
\]
\end{thm}
\noindent{\bf Remark.} If $\deg f=\deg\tilde f$, then since $\xi$ is a homeomorphism, comparing degrees on both sides of the relation gives $m'=m+l$.

\vskip 0.2cm
To apply the lifting argument introduced in Section~\ref{sec:outline}, we need a distortion lemma in both directions, controlling diameters under pushforward and pullback. This generalizes \cite[Lemma 4.1]{BLM14}. 

For a constant $C\geq 1$, we write $a\asymp b$ and say $a$ and $b$ are {\bf comparable} for the constant $C$ whenever $b/C\leq a\leq Cb$.
 
\begin{lema}\label{lem:distortion}
Let \( f \) be a subhyperbolic rational map. Fix a constant \( C \geq 1 \). Then for every sufficiently small \( \epsilon_{0} > 0 \), there exists a constant \( C_1=C_1(C) \geq 1 \) such that the following holds:

Let \( \alpha \subseteq \beta \) be small connected sets containing more than one point and intersecting the Julia set of $f$. Let $\alpha_k=f^k(\alpha)$ and $\beta_k=f^k(\beta)$, where $k$ is the largest integer such that $\tu{diam}(\beta_k)<\epsilon_0$. 

\begin{enumerate}
\item (pushforward) If $\tu{diam}(\alpha)\asymp \tu{diam}(\beta)$ for the constant $C$, then $$\tu{diam}(\alpha_k)\asymp\tu{diam}(\beta_k)\tu{ and }\tu{diam}(\beta_k)\asymp \epsilon_0$$ for the constant $C_1$.

\item (pullback) If $\tu{diam}(\alpha_k)\asymp\tu{diam}(\beta_k)$ for the constant $C$, then $\tu{diam}(\alpha)\asymp\tu{diam}(\beta)$ for the constant $C_1$.
\end{enumerate}
\end{lema}
\begin{proof}
Since $f$ is subhyperbolic, $P_f\cap J$ is a finite set. So there exist $\epsilon_0 > 0$ and an integer $N \geq 1$ such that for any point $p$ in the Julia set $J$ and any $m \geq 1$, each component of $f^{-m}(B(p, 2\epsilon_0))$ is a disk on which the degree of $f^m$ is at most $N$.

Let $k$ be as in the statement. Let $M$ denote the supremum of the spherical derivative of $f$. Then $\diam(\beta_k) \geq \epsilon_0/M$. Fix $p \in \alpha \cap J$ and set $B = B(f^k(p), 2\epsilon_0)$. Let $B_k$ be the component of $f^{-k}(B)$ containing $p$.
Choose conformal maps $\psi: B_k \to \mathbb{D}$ and $\phi: B \to \mathbb{D}$ with $\psi(p) = 0$ and $\phi(f^k(p)) = 0$. Then
$$
h := \phi \circ f^k \circ \psi^{-1}: \mathbb{D} \to \mathbb{D}
$$
is a proper map fixing $0$ with $\deg(h) \leq N$. 

\medskip\noindent\textbf{Claim.} \emph{Fix $r\in(0,1)$ and let $U_r$ be the component of $h^{-1}(\mathbb{D}(0,r))$ containing $0$. Then every $x\in\partial U_r$ satisfies
\(
r\le |x|\le 4r^{1/N}.
\)}
\begin{proof}[Proof of the claim]
The lower bound $r\le |x|$ follows from the Schwarz Lemma.

For the upper bound, pick $x_0\in\partial U_r$ with $|x_0|=\max_{x\in\partial U_r}|x|$. The restriction
\[
h:\mathbb{D}\setminus\overline{U_r}\longrightarrow \mathbb{D}\setminus\overline{\mathbb{D}(0,r)}
\]
is a proper map of degree at most $N$, so the monotonicity of conformal modulus gives
\[
\frac{1}{2\pi N}\tu{ln}\frac{1}{r}\le \tu{mod}(\mathbb{D}\setminus\overline{U_r})\le \mu(|x_0|),
\]
where $\mu(t)=\tu{mod}(\mathbb{D}\setminus[0,t])$ for $0<t<1$.

The standard estimate $\mu(t)<\frac{1}{2\pi}\tu{ln}\frac{4}{t}$ holds; see \cite{Ah}. Combining this with the previous inequality yields $|x_0|\le 4r^{1/N}$, and hence $|x|\le 4r^{1/N}$ for every $x\in\partial U_r$.
\end{proof}

Since $\diam(\beta_k) \geq \epsilon_0/M$, it follows from the Koebe distortion theorem that there exists a positive number $\delta_0$ such that $\diam(\phi(\beta_k)) > \delta_0$.

\medskip\noindent (1) Set $r = \diam(\phi(\beta_k))/4$ and let $U$ be the component of $h^{-1}(\mathbb{D}(0,r))$ containing $0$. Since $\diam(\phi(\beta_k)) = 4r$, the set $\phi(\beta_k)$ cannot be contained in $\mathbb{D}(0,r)$; hence $\psi(\beta) \not\subseteq U$, and $\partial U$ meets $\psi(\beta)$. By the claim, any $x \in \partial U \cap \psi(\beta)$ satisfies $|x| \geq r$, so $\diam(\psi(\beta)) \geq r \geq \delta_0/4$.

Since $\diam(\alpha) \asymp \diam(\beta)$, Koebe distortion yields $\diam(\psi(\alpha)) \geq \delta_1$ for some $\delta_1 = \delta_1(C) > 0$. Applying the claim to the component of $h^{-1}(\mathbb{D}(0, \diam(\phi(\alpha_k))/4))$ containing $0$, whose boundary intersects $\psi(\alpha)$, we obtain
\[
\diam(\phi(\alpha_k)) \geq (\delta_1/8)^N.
\]
As $\alpha_k \subseteq \beta_k$, we also have $\diam(\phi(\alpha_k)) \leq \diam(\phi(\beta_k)) \leq 2$. Thus, $\diam(\phi(\alpha_k)) \asymp 1$, and Koebe distortion implies $\diam(\alpha_k) \asymp \epsilon_0 \asymp \diam(\beta_k)$.

\medskip\noindent (2) The argument is similar. By assumption and Koebe distortion, $\diam(\phi(\alpha_k)) \asymp \diam(\phi(\beta_k))$. Let $r = \diam(\phi(\alpha_k))/4$. The component $U$ of $h^{-1}(\mathbb{D}(0,r))$ containing $0$ has boundary intersecting $\psi(\alpha)$. The claim gives $\diam(\psi(\alpha)) \geq r \geq \delta_2 > 0$, hence $\diam(\psi(\alpha)) \asymp \diam(\psi(\beta))$. The conclusion follows by Koebe distortion.
\end{proof}

\begin{proof}[Proof of Theorem \ref{thm:sierpinski-like}]
The proof is divided into the following three steps.

\subsection*{\boldmath Step 1. Inducing the homeomorphism $\mathbf{\tilde{\xi}}$.\unboldmath} 

We shall show that $\xi$ induces a homeomorphism $\tilde{\xi}:J_g\to J_{\tilde{g}}$.

We first match the closures of the Fatou domains. Set $E=\pi(\overline{U})$ and $\tilde{E}=\xi(E)$ for a Fatou domain $U$ of $g$. We claim that $\tilde{E}=\tilde{\pi}(\overline{\tilde{U}})$ for a unique Fatou domain $\tilde{U}$ of $\tilde{g}$.

To see this, note that by condition~(3), any two points of $E$ can be joined by a curve meeting $J$ in at most countably many points. Since $\xi$ is a homeomorphism, $\tilde{E}$ has the same property in $\tilde{J}$. Therefore $\tilde{E}$ cannot meet two distinct sets of the form $\tilde{\pi}(\overline{\tilde{U}})$: any curve in $\tilde{E}$ joining a point in one such set to a point in another would have to intersect $\tilde{J}$ in uncountably many points, which is impossible. Hence $\tilde{E}\subseteq\tilde{\pi}(\overline{\tilde{U}})$ for some Fatou domain $\tilde{U}$ of $\tilde{g}$. Applying the same argument to $\xi^{-1}$ yields the reverse inclusion, so equality holds.

Let $X\subseteq J_g$ and $\tilde{X}\subseteq J_{\tilde{g}}$ denote the buried point sets of $J_g$ and $J_{\tilde{g}}$, respectively. By condition~(2), $\pi$ maps $X$ bijectively onto a subset of the buried points of $J$, and $\pi(X)\cap\pi(\overline{U})=\emptyset$ for every Fatou domain $U$ of $g$. The same holds for $\tilde{\pi}$, $\tilde{X}$, and the Fatou domains of $\tilde{g}$. Since $\xi$ is a homeomorphism, the claim above implies that $\xi$ sends $\pi(X)$ bijectively onto $\tilde{\pi}(\tilde{X})$. Therefore the composition
\[
\tilde{\xi}:=\tilde{\pi}^{-1}\circ\xi\circ\pi:X\to\tilde{X}
\]
is a well-defined bijection between the buried point sets.

Now we extend $\tilde{\xi}$ to the peripheral circles by using prime ends. Since $U$ is a Jordan domain, its boundary $\partial U$ is identified with the Carath\'{e}odory boundary $\mathrm{Cara}(\partial U)$ of $\overline{\mathbb{C}}\setminus\ol{U}$. A prime end can be represented by a fundamental chain $\{L_n\}$ of cross-cuts in $\overline{\mathbb{C}}\setminus\ol{U}$; we may choose each $L_n$ so that its endpoints lie on $\partial U$ and its interior buried in $J$. Since $\pi$ is a quotient map sending $U$ onto a component of $\overline{\mathbb{C}}\setminus J$ with boundary $\pi(\partial U)$, the images $\{\pi(L_n)\}$ form a fundamental chain of cross-cuts in $\overline{\mathbb{C}}\setminus E$, defining a prime end in $\mathrm{Cara}(\partial E)$. Thus, $\pi$ induces a homeomorphism from $\mathrm{Cara}(\partial U)$ to $\mathrm{Cara}(\partial E)$.

The same reasoning applied to $\xi$ and $\tilde{\pi}$ yields homeomorphisms between $\mathrm{Cara}(\partial E)$ and $\mathrm{Cara}(\partial\tilde{E})$, and between $\mathrm{Cara}(\partial\tilde{U})$ and $\mathrm{Cara}(\partial\tilde{E})$. We therefore obtain a homeomorphism $\tilde{\xi}:\partial U\to\partial\tilde{U}$ by composing these maps:
\[
\tilde{\xi}|_{\partial U}=\left(\tilde{\pi}^{-1}\circ\xi\circ\pi\right)|_{\mathrm{Cara}(\partial U)}.
\]

Doing this for every Fatou domain of $g$ yields a bijection $\tilde{\xi}:J_g\to J_{\tilde{g}}$. Since the diameters of the Fatou domains of both $g$ and $\tilde{g}$ tend to zero, the extended map is continuous, hence a homeomorphism.

\medskip\noindent{\boldmath\bf Step 2. Proving quasisymmetry of \(\mathbf{\tilde{\xi}}\).\unboldmath} 
Note that the following diagram commutes for every integer $k\ge1$ by condition (1) in the theorem:
\[
\begin{tikzcd}[row sep=large, column sep=large]
J_g \arrow[r, "g^k", leftarrow] \arrow[d, "\pi"'] & 
J_g \arrow[r, "\tilde{\xi}"] \arrow[d, "\pi"] & 
J_{\tilde{g}} \arrow[r, "\tilde{g}^k"] \arrow[d, "\tilde{\pi}"] & 
J_{\tilde{g}} \arrow[d, "\tilde{\pi}"] \\
J \arrow[r, "f^k", leftarrow] & 
J \arrow[r, "\xi"] & 
\tilde{J} \arrow[r, "\tilde{f}^k"] & 
\tilde{J}
\end{tikzcd}
\]

\noindent{\bf Claim 1.} \emph{Fix a constant \(C \ge 1\). Then there exists a constant \(C' = C'(C)\) such that, if \(\alpha\) and \(\beta\) are curves in \(J_g\) with the same initial point and \(\operatorname{diam}(\alpha) \asymp \operatorname{diam}(\beta)\) for $C$, then \(\operatorname{diam}(\tilde{\xi}(\alpha)) \asymp \operatorname{diam}(\tilde{\xi}(\beta))\) for $C'$. }
\begin{proof}[Proof of Claim 1]
Since \( \tilde{\xi} \) is a homeomorphism, we may assume that the curves \( \alpha \) and \( \beta \) have diameters less than some small \( \delta>0 \).

Fix \( r_0 > \delta \). We iterate $\alpha$ and $\beta$ with $g$ and apply Lemma \ref{lem:distortion}~(1):
there exists an integer $m\geq 0$ such that, setting $\alpha_m=g^m(\alpha)$ and $\beta_m=g^m(\beta)$,
 \[ \operatorname{diam}(\alpha_m) \asymp \operatorname{diam}(\beta_m)\tu{ and } \operatorname{diam}(\beta_m) \asymp r_0 \]
for a constant $C_1=C_1(C)$.

We now pass $\alpha_m$ and $\beta_m$ to the $f$-system via $\pi$. Define  \( \alpha'_m = \pi(\alpha_m) \) and \( \beta'_m =\pi(\beta_m) \). Using the fact that \( \pi \) does not collapse an arc in $J_g$ to a point, we obtain \( \operatorname{diam}(\alpha'_m) \asymp \operatorname{diam}(\beta'_m) \) for a constant \( C_2 = C_2(C, r_0) \). 

 %Let $\alpha':=\pi(\alpha)=f^m(\alpha_m')$ and $\beta':=\pi(\beta)=f^m(\beta_m')$. Lemma \ref{lem:distortion}(2) gives that $$\tu{diam}(\alpha')\asymp\tu{diam}(\beta')$$ for a constant $C_3=C_3(C, r_0)$. 
Let $\alpha':=\pi(\alpha)$ and $\beta':=\pi(\beta)$. Then $f^m(\alpha')=\alpha_m'$ and $f^m(\beta')=\beta_m'$. By Lemma \ref{lem:distortion}~(2), $$\tu{diam}(\alpha')\asymp\tu{diam}(\beta')$$
 for a constant $C_3=C_3(C, r_0)$. 

Let \( k \) be the largest integer such that the diameters of  \( \tilde{\alpha}_k := \tilde{f}^k \circ \xi(\alpha') \) and \( \tilde{\beta}_k := \tilde{f}^k \circ \xi(\beta') \) are less than $r_0$. The distortion property of $\xi$ and Lemma \ref{lem:distortion}~(1) yield that for a constant \( C_4 = C_4(C, r_0) \), 
\[ \operatorname{diam}(\tilde{\alpha}_k) \asymp \operatorname{diam}(\tilde{\beta}_k) \tu{ and } \operatorname{diam}(\tilde{\beta}_k) \asymp r_0. \]

We now lift back $\tilde{\alpha}_k$ and $\tilde{\beta}_k$ to the $\tilde{g}$-system via $\tilde{\pi}$. Set \( \alpha_k^* = \tilde{g}^k \circ \tilde{\xi}(\alpha) \) and \( \beta_k^* = \tilde{g}^k \circ \tilde{\xi}(\beta) \). It follows from the commutative diagram in step 1 that
 \[ \tilde{\pi}(\alpha_k^*) = \tilde{\alpha}_k\tu{ and } \tilde{\pi}(\beta_k^*) = \tilde{\beta}_k .\]
The uniform continuity of $\tilde{\pi}$ allows us to transfer the diameter comparability from $\tilde{\alpha}_k$ and $\tilde{\beta}_k$ to $\alpha^*_k$ and $\beta^*_k$.

Finally, we apply Lemma \ref{lem:distortion}~(2) to the map $\tilde{g}$: forward iterating $\tilde{g}$ by $k$ steps sends $\tilde{\xi}(\alpha)$ and $\tilde{\xi}(\beta)$ onto $\alpha^*_k$ and $\beta^*_k$, respectively. This yields the desired conclusion: the original images $\tilde{\xi}(\alpha)$ and $\tilde{\xi}(\beta)$ have comparable diameters. This proves Claim 1.
\end{proof}
Let \( U \) and $\tilde{U}$ be Fatou domains of $g$ and $\tilde{g}$ such that \( \partial \tilde{U} = \tilde{\xi}(\partial U) \). By \cite[Theorem 1.10]{BLM14}, the boundaries of Fatou domains of \( g \) and \( \tilde{g} \) are uniform quasicircles. Thus, there exist  \( K \)-quasiconformal maps \( \phi \) and \( \tilde{\phi} \) on \( \overline{\mathbb{C}} \) with dilatation $K$ independent of $U$ such that \( \phi(\overline{U}) = \overline{\mathbb{D}} \) and \( \tilde{\phi}(\overline{\tilde{U}}) = \overline{\mathbb{D}} \).

Define the boundary map \( \psi = \tilde{\phi} \circ \tilde{\xi} \circ \phi^{-1}: \partial \mathbb{D} \to \partial \mathbb{D} \). We now show that \( \psi \) is uniformly quasisymmetric. For any two adjacent non-overlapping intervals \( \alpha \) and \( \beta \) of equal lengths on \( \partial \mathbb{D} \), the curves \( \phi^{-1}(\alpha) \) and \( \phi^{-1}(\beta) \) in \( J_g \) satisfy \( \tu{diam}(\phi^{-1}(\alpha)) \asymp \tu{diam}(\phi^{-1}(\beta)) \). By Claim 1, the images under \( \tilde{\xi} \circ \phi^{-1} \) have comparable diameters. The quasisymmetry of \( \tilde{\phi} \) then implies \( \tu{diam}(\psi(\alpha)) \asymp \tu{diam}(\psi(\beta)) \) for a constant depending only on \( K \).

Extend \( \tilde{\xi} \) to \( \overline{U} \) by defining \( \tilde{\xi}|_{\overline{U}} = \tilde{\phi}^{-1} \circ \psi \circ \phi \). This yields a homeomorphism \( \tilde{\xi}: \overline{\mathbb{C}} \to \overline{\mathbb{C}} \) that is uniformly quasisymmetric on each Fatou domain of \( g \).

To prove \( \tilde{\xi}: J_g \to J_{\tilde{g}} \) is quasisymmetric, it suffices to show that there is a constant $C\geq 1$ with the property:

\vskip 0.2cm\emph{For any triple of points \( x, y, z \in J_g \) with \( \tu{dist}(x, y) = \tu{dist}(x, z) \) in the spherical metric, the diameters of $\tilde{\xi}(\alpha)$ and $\tilde{\xi}(\beta)$ are comparable with the constant $C$, where  \( \alpha \) and \( \beta \) are spherical geodesics from \( x \) to \( y \) and from \( x \) to \( z \), respectively.} 

\medskip The intersection of $\alpha$ with the Fatou set of $g$ consists of cross-cuts \( \{\alpha_k\} \). Since Fatou domain boundaries are uniform quasicircles, for each \( \alpha_k \), we may choose an arc \( \alpha'_k \subseteq \partial U \) sharing the same endpoints such that  \(\tu{diam}(\alpha'_k) \asymp \tu{diam}(\alpha_k) \) with a constant $C_1$ uniformly in \( k \). By the quasisymmetry of \( \tilde{\xi} \) on $\overline{U}$, there exists constant \( C_2 = C_2(C_1) \) for which it holds that
\[
\tu{diam}(\tilde{\xi}(\alpha_k)) \asymp \tu{diam}(\tilde{\xi}(\alpha'_k)).
\]
We then modify $\alpha$ to avoid the Fatou set: let \( \alpha' = (\alpha \setminus \bigcup_k \alpha_k) \cup \bigcup_k \alpha'_k \subseteq J_g \). Consequently, $$ \tilde{\xi}(\alpha') = (\tilde{\xi}(\alpha) \setminus \bigcup_k \tilde{\xi}(\alpha_k)) \cup \bigcup_k \tilde{\xi}(\alpha'_k).$$
It follows from the triangle inequality that
$$\tu{diam}(\tilde{\xi}(\alpha))\leq\tu{diam}(\tilde{\xi}(\alpha'))+\tu{max}\{\tu{diam}(\tilde{\xi}(\alpha_k))\}\leq (1+C_2)\tu{diam}(\tilde{\xi}(\alpha'))$$
with the reverse inequality holding symmetrically. This implies that
\[
\tu{diam}(\tilde{\xi}(\alpha)) \asymp \tu{diam}(\tilde{\xi}(\alpha')) \quad \text{and} \quad \tu{diam}(\tilde{\xi}(\beta)) \asymp \tu{diam}(\tilde{\xi}(\beta')).
\]
Since \( \tu{diam}(\alpha) \asymp \tu{diam}(\alpha') \), \( \tu{diam}(\beta) \asymp \tu{diam}(\beta') \), we obtain \( \tu{diam}(\alpha')\asymp\tu{diam}(\beta') \). Both $\alpha'$ and $\beta'$ lie in $J_g$; applying Claim 1 gives \(\tu{diam}(\tilde{\xi}(\alpha')) \asymp \tu{diam}(\tilde{\xi}(\beta')) \) for a constant \( C_3 = C_3(C_1) \). Thus, \( \tu{diam}(\tilde{\xi}(\alpha)) \asymp \tu{diam}(\tilde{\xi}(\beta)) \). Hence, $\tilde{\xi}$ is quasisymmetric on $J_g$.\vspace{3pt}

\subsection*{Step 3. Final Conjugacy.} \vspace{3pt}

Since \( \tilde{\xi}: J_g \to J_{\tilde{g}} \) is a quasisymmetry between Sierpi\'{n}ski carpet Julia sets, by Proposition \ref{pro:sierpinski}, there exist integers \(m', l \geq 1 \) and \( m \geq 0 \) such that
\[
\tilde{g}^{m'} \circ \tilde{\xi} = \tilde{g}^m \circ \tilde{\xi} \circ g^l \quad \text{on } J_g.
\]

Restricting to the buried point set $X$ (which is dense), and using the conjugacies \( f \circ \pi = \pi \circ g \) and \( \tilde{f} \circ \tilde{\pi} = \tilde{\pi} \circ \tilde{g} \), together with the relation \( \xi \circ \pi = \tilde{\pi} \circ \tilde{\xi} \) on $J_g$, we obtain
\[
\tilde{f}^{m'} \circ \xi = \tilde{f}^{m} \circ \xi \circ f^l \quad \text{on } \pi(X).
\]
By the continuity of \( f \), \( \tilde{f} \), and \( \xi \), and the density of \( \pi(X) \) in \( J \), this equality holds on the whole Julia set \( J \). This completes the proof.
\end{proof}

\section{Dynamical relations for necklace Julia sets}\label{sec6}
In this section we prove Theorem~\ref{thm:necklace}, the necklace case of the dynamical relation. 

Recall that the McMullen maps are $f_\lambda(z)=z^n+\lambda/z^{n}$ with $\lambda\in\mathbb{C}^*$ and $n\ge2$. In the parameter plane, the unique hyperbolic component $\mathcal{H}$ containing $0$ is called the \tb{McMullen domain}; equivalently, $\mathcal{H}$ consists exactly of those parameters whose Julia sets are homeomorphic to the standard Cantor-circle $\mathcal{C}\times\mathbb{S}^1$~\cite{QRWY15}.

By Theorem \ref{thm1}, necklace Julia sets occur only when $n\ge3$; hence the discussion in this section concerns $n\ge3$. By Lemma \ref{lem:parameter-sym}, the Julia sets of $f_\lambda$ and $f_{\lambda e^{{2\pi \textbf{i}}/{(n-1)}}}$ differ only by a rotation. We may therefore restrict attention to the fundamental domain $$\mathcal{W}=\{\lambda\in\mathbb{C}^*: \textup{arg}(\lambda)\in [0, {2\pi}/{(n-1)})\}.$$
\begin{prop}[\textup{\cite[Theorem 5.5]{QRWY15}}]\label{prop:fundamental}
The intersection $\mathcal{L}=\partial\mathcal{H}\cap \mathcal{W}$ is an arc, which can be parametrized by $\Phi:\mathcal{L}\to [0, \frac{\pi}{n-1})$ with the dynamical property that $\Phi(\lambda)$ is the angle of the internal ray of $T_\lambda$ that lands at $v_\lambda$. If the Julia set of $f_\lambda$ is a necklace for $\lambda\in\mathcal{W}$, then $\lambda\in\mathcal{L}$. 
\end{prop}

Recall from Section~\ref{sec:symmetry} that $G_\lambda$ is the dihedral group of order $4n$ generated by the rotation $z\mapsto e^{\pi\tb{i}/n}z$ and the involution $z\mapsto \sqrt[n]{\lambda}/z$; every element of $G_\lambda$ preserves the Julia set of $f_\lambda$.

\begin{thm}\label{thm:necklace}
Let \( J \) and \( \tilde{J} \) be necklace Julia sets of PCF McMullen maps \( f_\lambda \) and \( f_{\tilde{\lambda}} \), respectively, with the same exponent $n\geq 3$. Assume $\lambda, \tilde{\lambda}\in\mathcal{W}$. Suppose \( \xi: J \to \tilde{J} \) is a quasisymmetry. Then, up to precomposing $\xi$ with an element of $G_\lambda$, we have 
\[
f_{\tilde{\lambda}}\circ \xi = \xi\circ f_{\lambda} \quad \text{on } J.
\]
\end{thm}
\begin{proof}
Write $f=f_\lambda$ and $\tilde{f}=f_{\tilde{\lambda}}$ to simplify the notation. Let $g$ be the model map from Theorem~\ref{thm:blow-upII}. Then there exist quotient maps $\pi:J_g\to J$ and $\tilde{\pi}:J_g\to\tilde{J}$ such that
\[
f\circ\pi=\pi\circ g \quad\text{and}\quad \tilde{f}\circ\tilde{\pi}=\tilde{\pi}\circ g \quad\text{on }J_g.
\]

By precomposing $\xi$ with an involution in $G_\lambda$ if necessary, we may assume that $\xi$ sends $\partial B$ to $\partial\tilde{B}$ and $\partial T$ to $\partial\tilde{T}$.
We now show that $\xi$ induces a homeomorphism $\tilde{\xi}:J_g\to J_g$ satisfying $\tilde{\pi}\circ\tilde{\xi}=\xi\circ\pi.$

Each component $\gamma$ of $J_g$ is a circle. It is called \textbf{peripheral} if $g$ eventually maps it onto $\{1\}\times\mathbb{S}^1$, and \textbf{buried} otherwise. The image $\pi(\gamma)$ is a buried Jordan curve in $J$ if and only if $\gamma$ is buried. In this case, $\xi\circ\pi(\gamma)$ is a buried Jordan curve in $\tilde{J}$, so there is a unique circle $\tilde{\gamma}\subseteq J_g$ such that $\tilde{\pi}(\tilde{\gamma})=\xi\circ\pi(\gamma)$. We define
\[
\tilde{\xi}|_\gamma=\tilde{\pi}^{-1}\circ\xi\circ\pi.
\]

If $\gamma$ is peripheral, then $\pi$ maps $\gamma$ homeomorphically into the boundary of a component of $f^{-k}(\overline{B})$ for some $k\ge0$. Since $\xi$ sends boundaries of Fatou domains to boundaries of Fatou domains, $\xi\circ\pi(\gamma)$ lies on the boundary of a component of $\tilde{f}^{-k}(\overline{\tilde{B}})$. Thus there is a unique peripheral circle $\tilde{\gamma}$ of $J_g$ with $\tilde{\pi}(\tilde{\gamma})=\xi\circ\pi(\gamma)$, and we define $\tilde{\xi}|_\gamma$ as before.

The continuity of quotient maps $\pi$ and $\tilde{\pi}$ yields that $\tilde{\xi}:J_g\to J_g$ is a homeomorphism. To see that $\tilde{\xi}$ fixes each circle of $J_g$, recall that Proposition~\ref{prop:classify-no-bounded-cycles}\,(3) and the $n$-folding dynamics of $f$ imply that $f^{-k}(\overline{T})$ consists of $2n^k$ Fatou domains arranged in a cycle. Since $\xi$ is a homeomorphism, it preserves this combinatorial structure and hence the level of every peripheral circle; by continuity it also preserves the levels of buried circles. Therefore $\gamma=\tilde{\gamma}$ in both cases above, and
\[
\tilde{\xi}(\{x\}\times\mathbb{S}^1)=\{x\}\times\mathbb{S}^1
\]
for every $x\in\mathcal{C}$.

Recall from Theorem~\ref{thm:blow-upII} that $\mathcal{A}_0=(1/n,1-1/n)\times\mathbb{S}^1$ and $\mathcal{A}_k=g^{-k}(\mathcal{A}_0)$ for $k\ge1$; hence $\partial\mathcal{A}_k$ is a union of peripheral circles of $J_g$. Let $\gamma_k$ be one of these circles. Set
\[
X_k=g^{-k}(\pi^{-1}(\tu{Crit}))\cap\gamma_k \quad\text{and}\quad \tilde{X}_k=g^{-k}(\tilde{\pi}^{-1}(\tilde{\tu{Crit}}))\cap\gamma_k.
\]
Both $X_k$ and $\tilde{X}_k$ consist of $2n^{k+1}$ points evenly distributed on $\gamma_k$. Because $\tilde{\xi}(X_k)=\tilde{X}_k$ and $\tilde{\xi}$ preserves their cyclic order, the restriction $\tilde{\xi}|_{X_k}:X_k\to\tilde{X}_k$ is a rotation
\[
\textup{rot}_{\alpha_k}:(x,\theta)\mapsto (x,\alpha_k+\theta~\textup{mod}~2\pi).
\]
If $\gamma_k$ and $\gamma_k'$ are boundary components of a component of $\mathcal{A}_k$, then $\alpha_k=\alpha_k'~\textup{mod}~2\pi$.

Now fix a circle $\gamma$ of $J_g$ and choose a sequence of peripheral circles $\gamma_k$ converging to $\gamma$. Since $\#X_k=2n^{k+1}\to\infty$, the continuity of $\tilde{\xi}$ implies that the rotation numbers $\alpha_k$ converge to a real number $\alpha$, and therefore $\tilde{\xi}|_\gamma=\textup{rot}_\alpha$.

\medskip\noindent\textbf{Claim 1.} \emph{Globally, $\tilde{\xi}: J_g\to J_g$ is a rotation, that is, the restrictions of $\tilde{\xi}$ to circles of $J_g$ share a common rotation number}. 

\begin{proof}[Proof of Claim 1]
Set $\mathcal{H}_0=I\times\mathbb{S}^1$ and define inductively $\mathcal{H}_{k+1}=\mathcal{H}_k\setminus\mathcal{A}_k$ for $k\ge0$. Each component of $\mathcal{H}_k$ is an annulus of width $n^{-k}$. Suppose the claim is false. Since the two boundary circles of any component of $\mathcal{A}_k$ share a common rotation number, we may assume that some component $H_0$ of $\mathcal{H}_{k_0}$ has positive rotation difference $\Delta_0=\Delta(H_0)$.

The annulus $\mathcal{A}_{k_0}$ cuts $H_0$ into two sub-annuli. One of them, say $H_1\subseteq\mathcal{H}_{k_0+1}$, must satisfy $\Delta(H_1)\geq\Delta_0/2$. Inductively, we obtain nested annuli $H_k\subseteq\mathcal{H}_{k+k_0}$ such that
\[
\Delta(H_k)\ge\frac{\Delta_0}{2^k}.
\]

For each $H_k$, choose points $x,y,z\in\partial H_k$ with $x,y$ on one boundary circle and $z$ on the other, and with $\textup{dist}(x,y)=\textup{dist}(x,z)=\textup{width}(H_k)$. Since $\tilde{\xi}$ acts as a rotation on each boundary circle, $\textup{dist}(\tilde{\xi}(x),\tilde{\xi}(y))=\textup{dist}(x,y)$, while the rotation difference gives $\textup{dist}(\tilde{\xi}(x),\tilde{\xi}(z))\ge\Delta(H_k)$. Hence
\begin{equation}\label{eq:diff}
\frac{\textup{dist}(\tilde{\xi}(x),\tilde{\xi}(z))}{\textup{dist}(\tilde{\xi}(x),\tilde{\xi}(y))}
\ge\frac{\Delta_0\,n^{k+k_0}}{2^k}\to\infty \quad\text{as }k\to\infty.
\end{equation}

To obtain a contradiction, it remains to show that this ratio is uniformly bounded in $k$. Recall that for every integer $l\ge1$ we have the following commutative diagram.
\[
\begin{tikzcd}[row sep=large, column sep=large]
J_g \arrow[r, "g^l", leftarrow] \arrow[d, "\pi"'] & 
J_g \arrow[r, "\tilde{\xi}"] \arrow[d, "\pi"] & 
J_{g} \arrow[r, "g^l"] \arrow[d, "\tilde{\pi}"] & 
J_{g} \arrow[d, "\tilde{\pi}"] \\
J \arrow[r, "f^l", leftarrow] & 
J \arrow[r, "\xi"] & 
\tilde{J} \arrow[r, "\tilde{f}^l"] & 
\tilde{J}
\end{tikzcd}
\]

We first show that $\textup{dist}(\pi(x),\pi(y))\asymp\textup{dist}(\pi(x),\pi(z))$. Fix a small constant $r_0>0$. For each large integer $k$, let $l=l(k)$ be the maximal integer such that
\[
\textup{dist}(g^l(x),g^l(y))=\textup{dist}(g^l(x),g^l(z))\ge r_0.
\]
Since $\pi$ is injective on $\partial\mathcal{H}_m$ for every $m\ge0$, the triple
\[
x'=\pi\circ g^l(x),\quad y'=\pi\circ g^l(y),\quad z'=\pi\circ g^l(z)
\]
is separated by some universal constant $\epsilon_0>0$; hence $\textup{dist}(x',y')\asymp\textup{dist}(x',z')$. Choose a topological disk $D$ centered at $x'$ whose boundary contains $y'$ and $z'$ and with
\[
\textup{shape}(D,x'):=\frac{\max_{w\in\overline{D}}\textup{dist}(x',w)}{\min_{w\notin D}\textup{dist}(x',w)}\le C
\]
for some universal constant $C\ge1$. Let $D'$ be the component of $f^{-l}(D)$ containing $\pi(x)$. Then $\pi(y),\pi(z)\in\partial D'$, and Lemma~\ref{lem:distortion}\,(2) gives
\[
\textup{shape}(D',\pi(x))\le C_1=C_1(C).
\]
Thus, we have $\textup{dist}(\pi(x),\pi(y))\asymp\textup{dist}(\pi(x),\pi(z))$.

Set $\tilde{x}=\xi\circ\pi(x)$, $\tilde{y}=\xi\circ\pi(y)$ and $\tilde{z}=\xi\circ\pi(z)$. Since $\xi$ is quasisymmetric and the previous step shows that $\textup{dist}(\pi(x),\pi(y))\asymp\textup{dist}(\pi(x),\pi(z))$, we have $\textup{dist}(\tilde{x},\tilde{y})\asymp\textup{dist}(\tilde{x},\tilde{z})$. Choose a disk $\tilde{D}$ centered at $\tilde{x}$ whose boundary contains $\tilde{y}$ and $\tilde{z}$ and with $\textup{shape}(\tilde{D},\tilde{x})\le C_2$. Push $\tilde{D}$ forward by $\tilde{f}$ until its diameter first exceeds $r_0$; let $m$ be this minimal iterate and set $\tilde{D}_1=\tilde{f}^m(\tilde{D})$. By Lemma \ref{lem:distortion}, the image $\tilde{D}_1$ contains a round disk $\tilde{D}_0$ of definite size centered at $x_*=\tilde{f}^m(\tilde{x})$.

We claim that $\tilde{D}_0$ avoids the two points $y_*=\tilde{f}^m(\tilde{y})$ and $z_*=\tilde{f}^m(\tilde{z})$. Indeed, $\tilde{f}^m:\tilde{\pi}(\overline{H_k})\to\tilde{f}^m(\tilde{\pi}(\overline{H_k}))$ is a covering, and, because $r_0$ is small, its restriction to $\tilde{D}_0\cap\tilde{\pi}(\overline{H_k})$ is injective. By taking $r_0$ small enough we may also arrange that $\tilde{D}_0$ is compactly contained in $\tilde{D}_1$, so it does not meet the boundary points $y_*$ and $z_*$.
Consequently,
\[
\textup{dist}(x_*,y_*)\asymp\textup{dist}(x_*,z_*)\ge r_1=r_1(r_0,C_2).
\]
The uniform continuity of $\tilde{\pi}$ therefore gives
\[
\textup{dist}(g^m\circ\tilde{\xi}(x),g^m\circ\tilde{\xi}(y))
\asymp\textup{dist}(g^m\circ\tilde{\xi}(x),g^m\circ\tilde{\xi}(z)).
\]
Pulling back by $g^m$, we obtain
$
{\textup{dist}(\tilde{\xi}(x),\tilde{\xi}(z))}\leq C_3\ {\textup{dist}(\tilde{\xi}(x),\tilde{\xi}(y))}
$
for a constant $C_3$ independent of $k$, contradicting \eqref{eq:diff}. This proves Claim~1.
\end{proof} 

Recall that $\phi_T:\partial T\to\partial\mathbb{D}$ is the extended B\"ottcher map. By Theorem~\ref{thm:blow-upII}~(4),
$
\pi(0,\tu{arg}(\phi_T(z)))=z \text{ for }z\in\partial T.
$
The B\"ottcher map is rotationally symmetric, and $\pi$ inherits the same symmetry:
\[
\pi(0,\alpha+\tu{arg}(\omega))=\omega\,\pi(0,\alpha),
\]
for every $\alpha\in\mathbb{R}/2\pi\mathbb{Z}$ and every $2n$-th root of unity $\omega$. The same holds for $\tilde{\pi}$.

By Claim~1, we may assume that $\tilde{\xi}=\tu{rot}_\alpha$ on $J_g$. After precomposing $\xi$ with a rotation in $G_\lambda$ if necessary, we may further assume that $0\le\alpha<\pi/n$.

Both $X:=g(X_0)$ and $\tilde{X}:=g(\tilde{X}_0)$ consist of two points whose images under $\pi$ and $\tilde{\pi}$ are the critical values of $f$ and $\tilde{f}$, respectively. Let $d_0=\tu{dist}(X,\tilde{X})$. Since $\lambda,\tilde{\lambda}\in\mathcal{W}$ by assumption, Proposition~\ref{prop:fundamental} gives $d_0<{\pi}/(n-1)$.

Since $g^{k+1}(X_k)=X$ and $g^{k+1}(\tilde{X}_k)=\tilde{X}$, we have
$
\tu{dist}(X_k,\tilde{X}_k)={d_0}/{n^{k+1}}
$
and $\tu{rot}_\alpha(X_k)=\tilde{X}_k$ for each peripheral circle $\gamma_k$. Because the points of $X_k$ are spaced by $2\pi/(2n^{k+1})$, it follows that
\[
\alpha=\frac{d_0}{n^{k+1}}+\frac{2\pi}{2n^{k+1}}\,m_k
\]
for every $k\ge0$ and some integer $0\le m_k<2n\cdot n^k$. Combining the cases $k=0,1$ gives
\[
d_0=\frac{\pi}{n-1}(m_1-nm_0).
\]
Since $d_0<\pi/(n-1)$, this forces $d_0=0$. The $k=0$ relation reduces to $\alpha=(\pi/n)m_0$, and the normalization $0\le\alpha<\pi/n$ implies $m_0=0$. Hence, $\alpha=0$ and $\tilde{\xi}=\textup{id}$.

Thus, $\xi=\tilde{\pi}\circ\pi^{-1}$ conjugates $f$ and $\tilde{f}$ on the buried point sets of their Julia sets. By continuity, this conjugacy extends to all of $J$, so $\xi\circ f=\tilde{f}\circ\xi$ on $J$. This completes the proof of the theorem.
\end{proof}
%Applying the arguments in the proof of Theorem \ref{thm:necklace} to $\tilde{f}=f$ and $\xi=\tu{id}$, we establish the uniqueness of the semi-conjugacy $\pi$.

%\begin{cor}\label{cor:uniqueness}
%Let $f$ be a PCF McMullen map for $n\geq 3$ whose Julia set $J$ is a necklace. Then the quotient map $\pi$ statisfying conditions (1)-(4) in Theorem \ref{thm:blow-upII} is unique.
%\end{cor}

\section{Dynamical relations for cluster Julia sets}\label{sec7}

In this section we establish the dynamical relation for PCF McMullen maps whose Julia set is a cluster; the precise statement is as follows.

\begin{thm}\label{thm:intersection-I}
Let $J$ and $\tilde{J}$ be cluster Julia sets of PCF McMullen maps $f_\lambda$ and $f_{\tilde{\lambda}}$, 
respectively. Suppose $\xi: J \to \tilde{J}$ is an orientation-preserving homeomorphism. 
Then, by precomposing $\xi$ with an element of $G_\lambda$ if necessary, we have
\begin{equation}\label{eq:semi-conjugacy}
f_* \circ \xi = \xi \circ f_\lambda \quad \text{on } J,
\end{equation}
where $f_*(z) = \omega \cdot f_{\tilde{\lambda}}(z)$ for some $2n$-th root of unity $\omega$.
\end{thm}

\noindent\textbf{Remark.} (1) The rational maps $f_{\tilde{\lambda}}$ and $f_*$ share the same Julia set $\tilde{J}$ by Lemma~\ref{lem:parameter-sym}. (2) The conclusion of Theorem~\ref{thm:intersection-I} only requires $\xi$ to be an orientation-preserving homeomorphism; quasisymmetry is not needed.

\subsection{Tower structure}
Let $f$ be a PCF McMullen map of exponent $n\geq 2$ whose Julia set is a cluster. The discussion in Subsection~\ref{sec4} shows that, according to the location of the free critical value $v$, the dynamics of $f$ falls into one of the following two cases:

\medskip
\textbf{Case~(I).} $v\in\partial B$. Then $\partial B \cap\partial T$ consists of the $2n$ free critical points.

\medskip\textbf{Case~(II).} $v\in K_v$ and $K_v\cap\partial B\neq\emptyset$, where $K_v$ is a component of the maximal Fatou chain. Then $-v\in -K_v$ and $K_v\cap -K_v=\emptyset$ by Lemma \ref{classification}(1). 

\medskip We set $K_v:=\{v\}$ by convention in Case (I). In both cases, let $\mathcal{K}$ be the union of all iterated preimages of $K_v$ and $-K_v$. 

An open or closed set $E$ is said to \textbf{connect} two disjoint sets $E_1$ and $E_2$ if $\overline{E}\cap\overline{E_1}\neq \emptyset$ and $\overline{E}\cap\overline{E_2}\neq\emptyset$. 
%$E$ is called a \textbf{connecting component} of $E_1$ and $E_2$. 
 
We now construct a tower structure of the cluster Julia set associated to $f$, which will be the key to encoding the dynamics and recognizing the quasisymmetry rigidity.

\medskip\noindent\textbf{Level 0.}
Set $\mathcal{V}_0:=V_0:= \C \setminus \overline{B}$ and $\mathcal{T}_0:=T_0:=T$.
Let $\mathcal{K}_0 := \mathcal{K}(T_0)$ denote the union of all components of 
$\mathcal{K}$ that connect $\pa B$ and $\pa T$ (i.e. $\pa V_0$ and $\pa T_0$).
By Lemma~\ref{classification}(3), $\mathcal{K}_0$ consists of exactly $2n$ components, 
each containing a free critical point and meeting each of $\partial B$ and $\partial T_0$ at a single point.
\medskip

\noindent\textbf{Level 1.}
Set $\mathcal{V}_1 := \mathcal{V}_0 \setminus (\overline{\mathcal{T}_0} \cup \mathcal{K}_0)$.
Then $\mathcal{V}_1$ consists of $2n$ components, each bounded by:
(i) one boundary arc of $\partial T$, (ii) one boundary arc of $\partial B$, 
and (iii) two adjacent components of $\mathcal{K}_0$.
The restriction of $f$ to each component of $\mathcal{V}_1$ is a conformal map onto 
$V_0 \setminus (K_v \cup -K_v)$.

Let $\mathcal{T}_1 := f^{-1}(\mathcal{T}_0)$. Since $f$ is conformal on each component of $\mathcal{V}_1$, 
the components of $\mathcal{T}_1$ and $\mathcal{V}_1$ are in one-to-one correspondence: 
each component $V_1$ of $\VVV_1$ contains exactly one component $T_1$ of $\mc{T}_1$.

Let $\mathcal{K}(T_1)$ denote the union of all components of $\mathcal{K}$ connecting $\pa V_1$ and $\pa T_1$ and $\mc{K}_{1}:= \bigcup \mathcal{K}(T_{1})$, where $T_{1}$ runs over all components of $\mathcal{T}_{1}$.
From the conformality of $f$ on $V_1$, exactly one of the following holds:
\begin{itemize}
\item[(1)] $K_v \not\subset \mathcal{K}_0$. Then $\mathcal{K}(T_1)\cap \mathcal{K}_0=\emptyset$. Each component of $\mathcal{K}(T_1)$ lies in $\overline{V_1}$ and meets its boundary at a single point.
\item[(2)] $K_v \subseteq \mathcal{K}_0$. Then $\mathcal{K}(T_1) \cap \mathcal{K}_0 \neq \emptyset$, 
consisting of at most two components, each of which is disjoint from $V_1$, while each component 
$K'$ in $\mathcal{K}(T_1)\setminus \mc{K}_0$ satisfies $K' \subseteq \overline{V_1}$ and $K' \cap \partial V_1$ 
is a single point. This configuration never occurs for McMullen maps in Case (I): $v\in\pa B$.
\end{itemize}
In either case, we have $f(\mathcal{K}(T_1)) = \mathcal{K}(T_0)$ and $\#\tu{Comp}(\mathcal{K}(T_1)) = 2n$. The components of $\mathcal{K}(T_1)\setminus\mathcal{K}_0$ are called \textbf{new} for $\mathcal{K}(T_1)$. It is clear that

\medskip\noindent$\textbf{Fact 1.}$
For each component $V_1$ of $\mathcal{V}_1$, the two boundary arcs of $\partial V_1 \setminus \mathcal{K}(T_0)$ 
are from distinct Fatou domains.
Each of these arcs meets at least one new component of $\mathcal{K}(T_1)$ at a single point.

\medskip\noindent{\boldmath\textbf{Level $i \geq 2$.}}
Having defined $\mathcal{T}_{i-1}$, $\mathcal{V}_{i-1}$, and $\mathcal{K}_{i-1}$, we set 
$$\mathcal{T}_i := f^{-1}(\mathcal{T}_{i-1})\text{\,\,\, and \,\,\,}\mathcal{V}_i := \mathcal{V}_{i-1} \setminus (\overline{\mathcal{T}_{i-1}} \cup \mathcal{K}_{i-1}).$$
Then $f$ restricted to each component $V_i$ of $\mathcal{V}_i$ satisfies:
\begin{equation}\label{eq:conformal-Vi}
f: V_i \to V_{i-1} \setminus (K_v \cup -K_v) \quad \text{is conformal.}
\end{equation}
It may occur that both $\pm K_v$ are disjoint from $V_{i-1}$ in the above. These conformal maps establish bijections between components of $\mathcal{V}_i$ and $\mathcal{T}_i$: 
each $V_i$ contains a unique component $T_i$.

Let $\mathcal{K}(T_i)$ denote the union of all components of $\mathcal{K}$ connecting $\partial V_i$ and $\pa T_i$. Set $\mc{K}_i:=\bigcup\mc{K}(T_i)$, where $T_i$ runs over all components of $\mc{T}_i$. 
By \eqref{eq:conformal-Vi}, we have $f(\mathcal{K}(T_i)) = \mathcal{K}(T_{i-1})$ and 
$\#\tu{Comp}(\mathcal{K}(T_i)) = 2n$. Components of $\mathcal{K}(T_i)\setminus(\mathcal{K}_0\cup\ldots\cup \mathcal{K}_{i-1})$ are called \textbf{new} for $\mc{K}(T_i)$. Analogous to Fact 1, by pullback, we have

\vskip 0.3cm
\noindent$\textbf{Fact 2.}$
For each component $V_i$ of $\mathcal{V}_i$, 
$\partial V_i \setminus (\mc{K}_0 \cup \ldots \cup \mc{K}_{i-1})$ 
has two boundary arcs from distinct Fatou domains, each meeting at least one new component of $\mathcal{K}(T_i)$ at a single point.
\vskip 0.3cm

%Finally, set $\mathcal{K}_i := \bigcup K(T_i)$ where the union is over all components $T_i$ of $\mc{T}_i$, and $\mathcal{T} := \bigcup_{i \geq 0} \mathcal{T}_i$.

The tower $(\mathcal{V}_i, \mathcal{T}_i, \mathcal{K}_i)_{i \geq 0}$ captures the nested self-similarity of the cluster Julia set; see Figure \ref{fig:connect}. The conformal maps \eqref{eq:conformal-Vi} at each level 
are the key to transferring dynamical information across scales.

Similarly, for another PCF McMullen map $\tilde{f}$ with the cluster Julia set $\tilde{J}$, 
we define the corresponding tower $(\tilde{\mathcal{V}}_i, \tilde{\mathcal{T}}_i, \tilde{\mathcal{K}}_i)_{i \geq 0}$.

% ====================================================================
%\subsubsection{The main theorem and proof strategy}
% ====================================================================

\subsection{Proof of Theorem \ref{thm:intersection-I}}
The proof proceeds in four steps:

Step 1. Using the {connecting number} $\ell(\Omega, \Omega')$ as a topological invariant, we show that $\xi$ sends the distinguished pair $\{B, T\}$ to $\{\tilde{B}, \tilde{T}\}$, possibly swapping them.

Step 2. Inductively, we show that $\xi$ preserves the tower structure: $\xi(\mathcal{V}_i) = \tilde{\mathcal{V}}_i$, $\xi(\mathcal{T}_i) = \tilde{\mathcal{T}}_i$, 
and $\xi(\mathcal{K}_i) = \tilde{\mathcal{K}}_i$ for all $i \geq 0$.

Step 3. We introduce a labeling system on the components of $V_i$ and prove that $\xi$ respects these labels.

Step 4. We construct the modified map $f_*$ and verify 
the conjugacy \eqref{eq:semi-conjugacy} by a limit argument.

We now carry out each step in detail. Write $f:=f_\lambda$ and $\tilde{f}:=f_{\tilde{\lambda}}$ for simplicity. 

\subsection*{Step 1. Identifying the distinguished Fatou domains.} The map $\xi$ obviously extends continuously to the Fatou set. We may assume $\xi$ is a homeomorphism over $\overline{\mb{C}}$. 

\begin{lema}\label{claim:preserve-BT}
As a collection of Fatou domains, $\{\xi(B),\, \xi(T)\}=\{\tilde{B},\, \tilde{T}\}.$
%In other words, $\xi$ either preserves $B$ and $T$ individually or swaps them.
\end{lema}
\begin{proof}[Proof of Lemma \ref{claim:preserve-BT}] 
Let $\Omega$ and $\Omega'$ be Fatou domains of $f$ whose closures are disjoint. 
The \textbf{connecting number} $\ell(\Omega, \Omega')$ is defined as follows, corresponding to two topological configurations:

\medskip{\textbf{Case (A).} Non-Jordan case:} no component of $\mathcal{K}$ is a closed Jordan domain. Then $\ell(\Omega, \Omega')$ is the number of components of $\mathcal{K}$ connecting $\Omega$ and $\Omega'$.

\textbf{Case (B).} Jordan case: all components of $\mathcal{K}$ are closed Jordan domains (hence closed Fatou domains). Then $\ell(\Omega, \Omega')$ 
is the number of Fatou domains connecting $\Omega$ and $\Omega'$.

\medskip The connecting number is a topological invariant: it depends only on the topological intersection pattern of the boundaries of $\Omega$ and $\Omega'$, and hence is preserved under homeomorphisms, i.e.,
$\ell(\Omega, \Omega')=\ell(\xi(\Omega), \xi(\Omega')).$

We first prove the lemma for Case (A). Set $\mathcal{T} := \bigcup_{i \geq 0} \mathcal{T}_i$. Since each component $K$ of $\mathcal{K}$ intersects the boundary of any Fatou domain in $\mathcal{T}\cup B$ in at most one point, and since this intersection point avoids the boundaries of the Fatou domains contained in $K$, the image $\xi(K)$ must be a component of $\tilde{\mc{K}}$. Consequently, $\xi$ establishes one-to-one correspondence between $\mc{K}$ and $\tilde{\mc{K}}$, and between $\mc{T}\cup B$ and $\tilde{\mc{T}}\cup\tilde{B}$.

By definition,
$\ell(B,T)=\#\tu{Comp}(\mathcal{K}(T_0))=2n.$
Thus it suffices to show $\ell(\Omega, \Omega')<2n$ whenever $\{\Omega, \Omega'\}\neq \{B, T\}$.

\medskip\noindent\textbf{Subcase (A-1).} Both $\Omega,\Omega'\subset \mathcal{T}_{i}$ with $i\ge 1$. Write $\Omega\subset V_i$ and $\Omega'\subset V_i'$ where $V_i$ and $V_i'$ are components of $\mathcal{V}_{i}$. Then $V_i\neq V_i'$ since the components of $\mc{T}_i$ and $\mc{V}_i$ are in one-to-one correspondence. Any component $K$ of $\mc{K}$ connecting $\Omega$ and $\Omega'$ must intersect both $\partial V_i$ and $\partial V_i'$. Recall that $\mathcal{K}(\Omega)$ is the union of all components of $\mathcal{K}$ connecting $\partial V_i$ and $\partial \Omega$, and similarly for $\mathcal{K}(\Omega')$. Therefore $K$ belongs to both $\mathcal{K}(\Omega)$ and $\mathcal{K}(\Omega')$. By Fact 2, $K$ cannot be new. It follows that $\ell<\#\tu{Comp}(\mathcal{K}(\Omega))=2n.$

\medskip\noindent\textbf{Subcase (A-2).} $\Omega\subset \mathcal{T}_{i}$ and $\Omega'\subset \mathcal{T}_{i'}$ with $i'>i\geq0$. Write $\Omega\subset V_i$ and $\Omega'\subset V_{i'}$ as before. If $V_i\cap V_{i'}=\emptyset$, then any component connecting $\Omega$ and $\Omega'$ must meet $\partial V_i$, $\partial V_{i'}$, $\partial\Omega$, and $\partial\Omega'$; by Fact 2 again there are fewer than $2n$ such components. If $V_{i'}\subset V_i$, then $\Omega\cap V_{i'}=\emptyset$. Hence, any component connecting $\Omega$ and $\Omega'$ must exit $V_{i'}$ before reaching $\Omega$. By Fact 2, some new component of $\mathcal{K}(\Omega')$ is disjoint from $\overline{\Omega}$. It forces $\ell<\#\tu{Comp}(\mathcal{K}(\Omega'))=2n$.

\medskip\noindent\textbf{Subcase (A-3).} One domain is $B$ and the other lies in $\mathcal{T}_{i}$ with $i\ge 1$. Recall that $V_0=\C\setminus\overline{B}$ contains $T_0=T$. Thus this case is entirely analogous to Subcase (A-2), with the outer domain $V_i$ replaced by $V_0$ and the inner domain $V_{i'}$ replaced by $V_i$ containing $\Omega'$. Since $V_i\subset V_0$, any component of $\mathcal{K}$ connecting $B$ and $\Omega'$ must exit $V_i$ before reaching $B$; the same counting argument as in Subcase (A-2) gives $\ell<2n$.

\medskip\noindent\textbf{Subcase (A-4).} At least one domain lies in $\mathcal{K}$. Suppose $\Omega$ is contained in a component $K$ of $\mathcal{K}$. Since the components of $\mathcal{K}$ are pairwise disjoint, any component of $\mathcal{K}$ connecting $\Omega$ and $\Omega'$ must coincide with $K$ itself; no other component of $\mathcal{K}$ can meet the boundary of $\Omega$. Hence there is at most one connecting component, and $\ell\leq 1$.

Combining the above arguments, we obtain the lemma for Case (A). 

\medskip We now turn to Case (B) where all components of $\mathcal{K}$ are closed Fatou domains.
The difference from Case (A) is that here the components of $\mathcal{K}$ are themselves closed Fatou domains.
Thus the map $\xi$ need not send each component of $\mathcal{K}$ to a component of $\tilde{\mathcal{K}}$. We cannot use the component-by-component correspondence from Case (A). We simply regard the components of $\mathcal{K}$ as closed Fatou domains and count them accordingly.

Observe that in Case (B), any Fatou domain connecting two domains in $\mathcal{T}\cup B$ must lie in $\mathcal{K}$. Therefore, the same arguments as above imply $\ell<2n$ whenever $\Omega,\Omega'\subset\mathcal{T}\cup B$.

Since $\ell=0$ whenever $\Omega\subset \mc{K}$ and $\Omega'\subset \mc{T}\cup B$ (because $\ol{\Omega}\cap\ol{\Omega'}=\emptyset$ by definition and any Fatou domain connecting  them lies neither in $\mc{K}$ nor in $\mc{T}\cup B$, it suffices to consider the case that $\Omega, \Omega'\subset \mathcal{K}$. Then any connecting component for $\Omega'$ and $\Omega$ lies in $\mc{T}\cup B$.
Write $\Omega\subset \mc{K}_i$ and $\Omega'\subset \mc{K}_{i'}$ for minimal integers $i, i'\geq 0$.
\begin{figure}[http]
\centering
\begin{tikzpicture}
\node at (0,0){ \includegraphics[width=11cm]{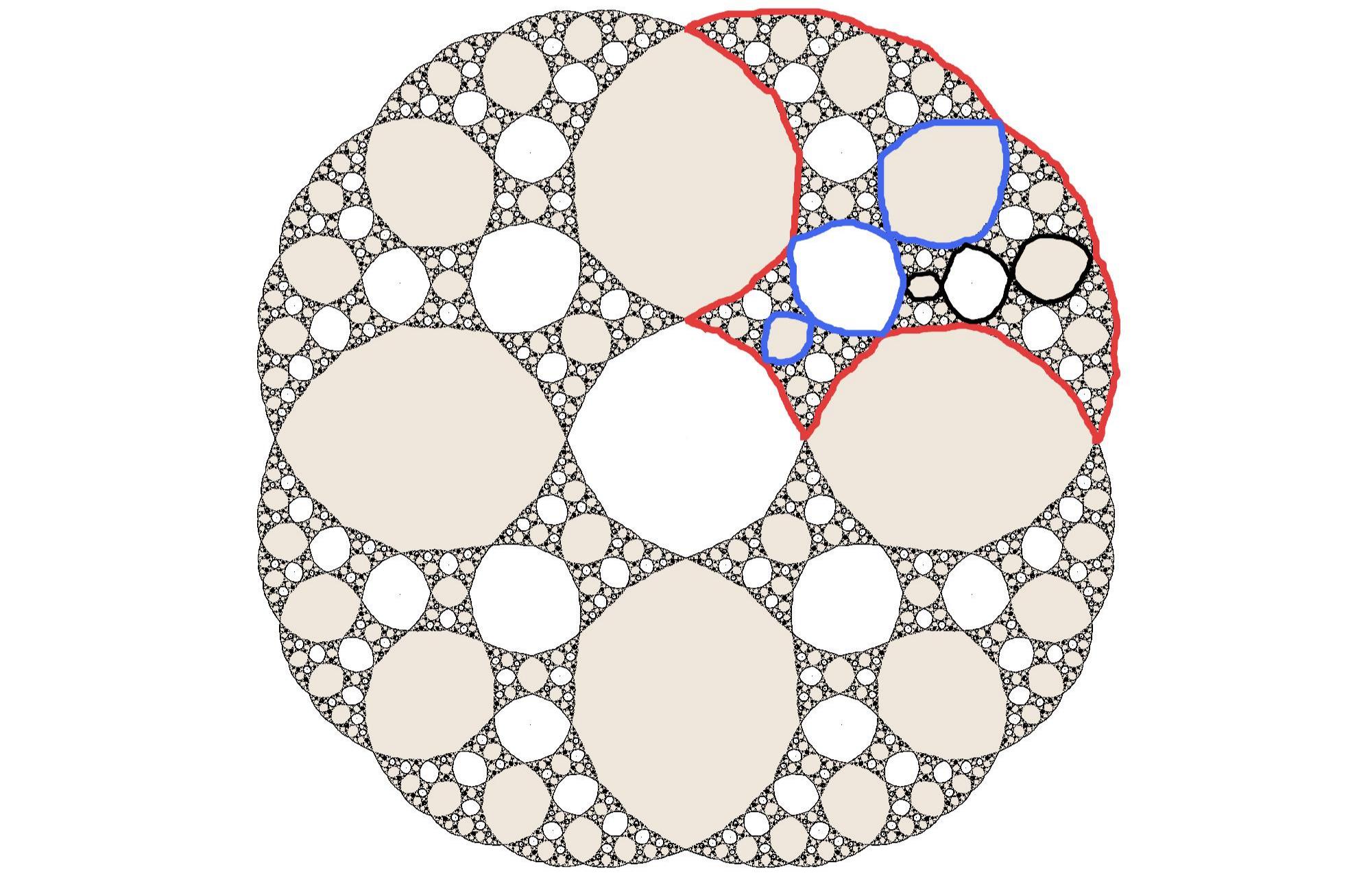}};
\node at (1.25, 1.25){$T_1$};
\node at (2.3, 1.25){$T_2$};
\node at (5.0, 1.25){$B$};
\node at (0, 0){$T$};
%\ruler{3}{4};
\end{tikzpicture}
\caption{Illustration of the proof of Case (B).}\label{fig:connect}
\end{figure}
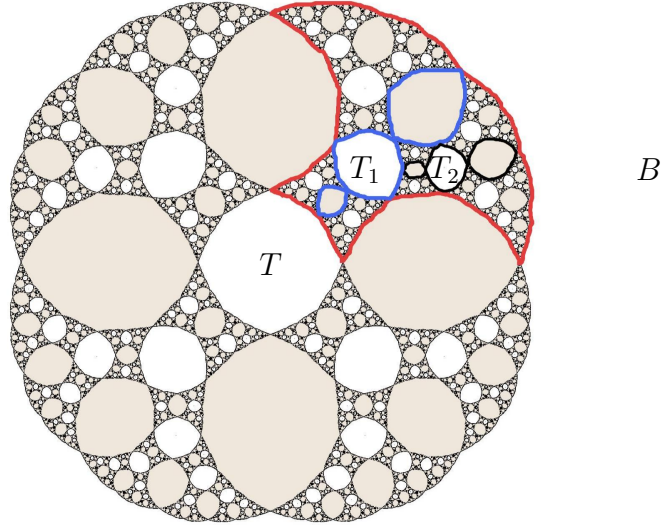

\medskip\noindent\textbf{Subcase (B-1).} $i=i'=0$. Then $\Omega$ and $\Omega'$ can be connected by $B$, $T$, and possibly some components $T_k$ of $\mc{T}_k$ with $k\geq 1$. If such a component $T_k$ exists, write $T_k\subset V_1$. By Fact 1, no components of $V_1\setminus(\overline{T_1}\cup\mc{K}(T_1))$ connects $\Omega$ and $\Omega'$. It then follows that $T_k=T_1$; see Figure~\ref{fig:connect}. Thus, $\ell\leq 3<2n$.
\vskip 0.2cm

\medskip\noindent\textbf{Subcase (B-2).} $i=i'\geq 1$. If $\Omega$ and $\Omega'$ lie in distinct components of $\mc{V}_1$, then only $B$ and $T$ possibly connect them. Thus, $\ell\leq 2<2n.$ By assumption $\Omega$ and $\Omega'$ are new for $\mc{K}_i$. Note that the images $f(\Omega)$ and $f(\Omega')$ are also new for $\mc{K}_{i-1}$. Since $f$ is conformal on components of $\mc{V}_1$, $\ell(\Omega, \Omega')=\ell(f(\Omega), f(\Omega'))$. By induction, we are reduced to Subcase (B-1). Thus, $\ell<2n$.

\medskip\noindent\textbf{Subcase (B-3).} $i'>i\geq 0$. By applying the argument of (B-2) inductively and using the conformality of $f$, we are reduced to the base situation $i=0$ and $i'\geq 1$. 
Suppose $\overline{\Omega'}$ is new for $\mc{K}(T_{i'})$ with $\Omega'\cup T_{i'}\subset V_{i'}\subset  V_1$. 

We consider the base case $i'=1$. Set $\mc{U}=V_1\setminus (\overline{T_1}\cup \mc{K}(T_1))$. If no component of $\mc{U}$ connects $\Omega$ and $\Omega'$, then any component connecting $\Omega$ and $\Omega'$ can involve at most two Fatou domains $B$ and $T_1$ (or, symmetrically, $T$ and $T_1$); hence $\ell\leq 2<2n$.

Otherwise, let $U$ be the unique component of $\mc{U}$ that connects $\Omega, \Omega'$. Applying Fact 2 iteratively inside $U$, we obtain a component $V_k$ of some deeper level $\mc{V}_k(k\geq 2)$ nested in $U$ and a domain $T_k\subset V_k$ such that $\overline{T_k}\cup \mc{K}(T_k)$ separates $\Omega$ from $\Omega'$. Consequently, any component connecting $\Omega$ and $\Omega'$ can involve at most three Fatou domains, namely $B, T_1, T_k$, or the symmetric triple $T, T_1, T_k$; see Figure~\ref{fig:connect}. Thus $\ell\leq 3<2n$.

For $i'\geq 2$, since $\Omega'\subset \mc{K}(T_{i'})\subset V_{i'}\subset V_{i'-1}$, there exists a component $U$ of $V_{i'}\setminus\overline{T_{i'}}\cup \mc{K}(T_{i'})$ connecting $\Omega$ and $\Omega'$; otherwise, only $B, T_{i'}$ and $T_{i'-1}$ (or, symmetrically, $T, T_{i'}$ and $T_{i'-1}$) are available as connecting components, giving $\ell\leq 3<2n$; see Figure~\ref{fig:connect}. By Fact 2, inside $U$, we find a deeper wall $\overline{T_k}\cup\mc{K}(T_k)$ separating $\Omega$ from $\Omega'$ within $\ol{V_k}$, whence any connecting path involves at most three Fatou domains. Therefore, $\ell<2n$.

In every subcase we have shown that $\ell(\Omega, \Omega')<2n$ unless $\{\Omega, \Omega'\}=\{B, T\}$. Since $\xi$ preserves the connecting number $\ell$, the pair $\{B, T\}$ is mapped onto itself. The proof of Lemma \ref{claim:preserve-BT} is complete.
\end{proof}

%---------------------------------------------------------------------
\begin{rmk}
The proof shows that the connecting number $\ell$ not only distinguishes $\{B,T\}$ from all other pairs, but also encodes the level of a Fatou domain in the tower $(\mc{V}_{i},\mc{T}_{i})$. This observation will be crucial in the inductive argument of Proposition~\ref{prop:preserve-tower} below.
\end{rmk}

\subsection*{Step 2. Preserving the tower structure}

Having established that $\xi$ preserves $B$ and $T$, by precomposing $\xi$ with an involution in $G_{\lambda}$ if necessary, we may assume $
\xi(B)=\tilde{B}\text{ and } \xi(T)=\tilde{T}.$ We now show inductively 
that $\xi$ preserves the entire tower:

\begin{prop}\label{prop:preserve-tower}
For all $i \geq 0$, we have
\begin{equation}\label{eq:preserve-tower}
\xi(\mathcal{V}_i) = \tilde{\mc{V}}_i, \quad
\xi(\mc{T}_i) = \tilde{\mc{T}}_i, \quad
\xi(\mathcal{K}_i) = \tilde{\mathcal{K}}_i.
\end{equation}
\end{prop}

\begin{proof}
The proof is by induction on $i$. We begin with the base cases $i=0, 1$.

For $i=0$, we have $\mc{V}_0=\C\setminus\overline{B}$ and $\tilde{\mc{V}}_0=\C\setminus\overline{\tilde{B}}$. By Lemma~\ref{claim:preserve-BT} and the inductive assumption, $\xi(\mc{V}_0)=\tilde{\mc{V}}_0$ and $\xi(\mc{T}_0)=\tilde{\mc{T}}_0$. Moreover, $\mc{K}_0$ consists exactly of those components of $\mc{K}$ that connect $\pa B=\pa V_0$ and $\pa T=\pa T_0$. Since this connecting property is preserved by $\xi$, we obtain $\xi(\mathcal{K}_0)=\tilde{\mathcal{K}}_0$, and consequently $\xi(\mc{V}_1)=\tilde{\mc{V}}_1$.

For $i=1$, fix a component $V_1$ of $\mc{V}_1$. For a Fatou domain $\Omega\subset V_1$, define the generalized connecting number by
\[\ell_1(\Omega):=\left\{
\begin{array}{ll}
\#\{\text{components of }\mathcal{K}\text{ connecting }\partial V_1\text{ and }\Omega\}, & \hbox{\tu{in Case (A)};} \\[5pt]
\#\{\text{Fatou domains connecting }\partial V_1\text{ and }\Omega\}, & \hbox{in Case (B).}
\end{array}
\right.
\]
Since $f$ is conformal on $V_1$, the arguments in Lemma~\ref{claim:preserve-BT} show that $\ell_1(\Omega)\leq 2n$, with equality if and only if $\Omega=T_1$ for some component $T_1\subset\mc{T}_1$. As $\ell_1$ is a topological invariant and $\xi(V_1)=\tilde{V}_1$, the map $\xi$ sends each such $T_1$ to a component $\tilde{T}_1$ of $\tilde{\mc{T}}_1$; hence $\xi(\mathcal{K}(T_1))=\mathcal{K}(\tilde{T}_1)$. Recalling that $\mc{K}_1=\bigcup\mc{K}(T_1)$, where the union is taken over all components $T_1\subset\mc{T}_1$, we conclude that \eqref{eq:preserve-tower} holds for $i=1$.

Assume \eqref{eq:preserve-tower} holds for all $j \leq i - 1$.
We have $V_i = V_{i-1} \setminus (\overline{T_{i-1}} \cup \mathcal{K}_{i-1})$.
By the induction hypothesis, $\xi$ maps $V_{i-1} \setminus \overline{T_{i-1}}$ 
bijectively onto $\tilde{V}_{i-1} \setminus \overline{\tilde{T}_{i-1}}$.
The components of $\mathcal{K}(T_{i-1})$ are characterized 
as the components of $\mathcal{K}$ connecting $\partial V_{i-1}$ and $\pa T_{i-1}$, and this property is preserved by $\xi$.
Hence $\xi(V_i) = \tilde{V}_i$.

The equality $\xi(T_i) = \tilde{T}_i$ follows by analogous argument: defining a generalized connecting number $\ell_i$ for Fatou domains in $V_i$, which is preserved by $\xi$; similar arguments to those in Lemma \ref{claim:preserve-BT} imply that $\ell_i(\Omega)\leq 2n$, with equality if and only if $\Omega=T_i$. 

Finally, $\xi(\mathcal{K}(T_i)) ={\mathcal{K}}(\tilde{T}_i)$ since $\mathcal{K}(T_i)$ 
consists of the components of $\mathcal{K}$ connecting $\partial V_i$ and $\pa T_i$. Hence $\xi(\mathcal{K}_i) = \tilde{\mathcal{K}}_i$. This completes the proof of the proposition.
\end{proof}

% ====================================================================
\subsection*{Step 3. The labeling system}
% ====================================================================

We now encode the geometry of the tower using symbolic labels to its components.
\vskip 0.2cm

\tb{Level 1.} The $2n$ components $V_1$ of $\mc{V}_1$ inherit a cyclic order from their 
counterclockwise arrangement around $T$. We label them by 
$\eta(V_1) \in \{0, 1, \ldots, 2n-1\}$ such that two components 
$V_1$ and $V_1'$ are connected by a component of $\mc{K}_0$ if and only if 
$\eta(V_1) = \eta(V_1') \pm 1\tu{ mod }{2n}$.

\vskip 0.2cm
{\boldmath\textbf{Level $i \geq 2$.}}
Fix a component $V_i$ of $\mc{V}_i$. By \eqref{eq:conformal-Vi}, 
the conformal map $f: V_i \to V'_{i-1} \setminus (K_v \cup -K_v)$ is a bijection onto its image. We use $f$ to lift the existing label on $\mc{V}_{i-1}$. Specifically, suppose $\eta(V_{i-1}')=\tau_2\ldots\tau_{i}$. Let $\tau_1=\eta(V_1)$ where $V_1$ is the unique component of $\mc{V}_1$ containing $V_i$. Then we assign to $V_i$ the label
\begin{equation}\label{eq:label}
\eta(V_i) = \tau_1\tau_2\cdots \tau_{i}.
\end{equation}

In general, each component $V_i \subset \mc{V}_i$ carries a label in \eqref{eq:label} satisfying the following two consistency conditions:
\begin{enumerate}
\item[(L1)] If $V_i \subset V_{i-1}$ with $i\geq 2$, 
then $\eta(V_{i-1}) = \tau_1 \ldots \tau_{i-1}$;
\item[(L2)] If $f(V_i) \subset V_{i-1}'$ where $V_{i-1}'$ is a component of $\mc{V}_{i-1}$, 
then $\eta(V'_{i-1}) = \tau_2 \ldots \tau_i$ (the left shift of $\eta(V_i)$).
\end{enumerate}

In general, $\xi(K_v)$ need not coincide with $K_{\tilde{v}}$, because $\xi$ may place $K_v$ at a symmetrically different position. Here $K_{\tilde{v}}$ is the component of the maximal Fatou chain of $\tilde{f}$ containing its critical value $\tilde{v}$ (with $K_{\tilde{v}}=\{\tilde{v}\}$ in Case~(I)). To align the markings, we replace $\tilde{f}$ by a suitable rotation $f_*=e^{k\pi\tb{i}/n}\tilde{f}$; by Lemma~\ref{lem:parameter-sym}, this preserves the Julia set and dynamics. The corresponding set $K_{v_*}$ in the tower of $f_*$ is then chosen so that $\xi(\pm K_v)=\pm K_{v_*}$.  This implies that, for any component $K$ of $\mc{K}_0$,
\begin{equation}\label{eq:conjugate}
  \xi\circ f(K)=f_*\circ \xi(K).
\end{equation}

We define labels $\eta_*$ for the tower of $f_*$ in the same manner. Since $\xi$ maps $\mc{V}_1$ 
bijectively onto $\tilde{\mc{V}}_1$ and $\xi$ preserves orientation, we may assume $\eta(V_1)=\eta_*(\xi(V_1))$, i.e., $\eta(V_1)=\eta_*(\tilde{V}_1)$ if and only if $\xi(V_1)=\tilde{V}_1$ for all components $V_1, \tilde{V}_1$ of $\mc{V}_1$, $\tilde{\mc{V}}_1$, respectively.

\begin{lema}\label{lem:label-alignment}
$\eta(V_i) = \eta_*(\xi(V_i))$ for all $i\geq 1$ and all components $V_i$ of $\mc{V}_i$.
\end{lema}
\begin{proof}[Proof of Lemma~\ref{lem:label-alignment}]
The base case $i=1$ holds, and equation \eqref{eq:conjugate} holds for every component of $\mc{K}_0$, by the above setting for $\eta_*$.

We proceed by induction on $i$. Assume that the lemma holds for all levels up to $i-1$ and that \eqref{eq:conjugate} holds for every component of $\mc{K}_0\cup\cdots\cup\mc{K}_{i-2}$, where $i\geq 2$. Let $V_i\subset V_{i-1}$ be a component at level $i$, and set $\tilde{V}_i=\xi(V_i)$ and $\tilde{V}_{i-1}=\xi(V_{i-1})$. By the induction hypothesis, $\eta(V_{i-1})=\eta_*(\tilde{V}_{i-1})$, so condition (L1) implies that the first digits of $\eta(V_i)$ and $\eta_*(\tilde{V}_i)$ coincide with those of $\eta(V_{i-1})$.

The maps
$f: V_{i-1}\to V_{i-2}\setminus(K_v\cup -K_v)\text{ and }f_*: \tilde{V}_{i-1}\to \tilde{V}_{i-2}\setminus(K_{v_*}\cup -K_{v_*})$
are conformal. Since $\eta(V_{i-1})=\eta_*(\tilde{V}_{i-1})$, condition (L2) gives $\eta(V_{i-2})=\eta_*(\tilde{V}_{i-2})$ (for $i\geq 3$); hence $\xi(V_{i-2})=\tilde{V}_{i-2}$ by the induction hypothesis. For $i=2$, this statement is trivial since $V_0=\tilde{V}_0=\C\setminus\overline{B}$.

The choice of $f_*$ ensures that $\xi$ sends $\{\pm K_v\}\cap V_{i-2}$ onto $\{\pm \tilde{K}_{v_*}\}\cap \tilde{V}_{i-2}$ (if any). Therefore, the induced homeomorphism
\[
\xi_* := f_*^{-1}\circ \xi\circ f: V_{i-1}\to \tilde{V}_{i-1}
\]
is well defined.

Since $\xi$ preserves the tower, it maps the cyclic arrangement of the $2n$ ``bridges" in $\mc{K}(T_{i-1})$ around $T_{i-1}$ to the corresponding arrangement in $\mc{K}(\tilde{T}_{i-1})$. The conjugation defining $\xi_*$ inherits this property. Moreover, by \eqref{eq:conjugate}, $\xi$ and $\xi_*$ agree on $\mc{K}_0\cup\cdots\cup\mc{K}_{i-2}$, whose components intersect $\partial V_{i-1}$ and determine the positions of the gaps inside $V_{i-1}$. Therefore $\xi_*$ sends $V_i$ to the same gap as $\xi(V_i)$, i.e., $\xi_*(V_i)=\tilde{V}_i$. From $\xi_*=f_*^{-1}\circ\xi\circ f$ we obtain $\xi\circ f(V_i)=f_*(\tilde{V}_i)$. By induction and condition (L2), it then follows that $V_i$ and $\tilde{V}_i$ share the remaining $i-1$ digits. Hence, $\eta(V_i)=\eta_*(\xi(V_i))$.

Moreover, the same argument shows that \eqref{eq:conjugate} holds for every component of $\mc{K}_{i-1}$, completing the induction.
\end{proof}

% ====================================================================
\subsection*{Step 4. The conjugacy}
% ====================================================================

We now complete the proof of Theorem~\ref{thm:intersection-I}.

\begin{proof}[Completion of the proof of Theorem \ref{thm:intersection-I}]
For $i \geq 1$, label the components of $\mc{T}_i$ by 
$\eta(T_i) := \eta(V_i)$ where $V_i$ is the unique component of $\mc{V}_i$ 
containing $T_i$. By condition (L2), the left shift operator $\sigma$ satisfies 
$\sigma(\eta(T_i)) = \eta(f(T_i))$. A similar label applies to 
$\tilde{\mc{T}}_i$ for $f_*$.

Lemma~\ref{lem:label-alignment} gives $\eta(T_i) = \eta_*(\xi(T_i))$ 
for all $i \geq 1$. Therefore, for $i\geq 2$, we have
$$\eta_*(f_* \circ \xi(T_i))= \sigma(\eta_*(\xi(T_i))) = \sigma(\eta(T_i)) = \eta(f(T_i))= \eta_*(\xi \circ f(T_i)).$$

Since the labeling is faithful (distinct components have distinct labels), 
we conclude $f_* \circ \xi(T_i) = \xi \circ f(T_i)$ for every component $T_i$ of $\mc{T}_i$ and every $i \geq 1$.

Since both $f$ and $\tilde{f}$ are postcritically finite, 
$\diam(T_i)\to 0$ and $\diam(\tilde{T}_i)\to 0$ as $i\to\infty$.
Every point $z\in J$ is the limit of a sequence of components 
$T_i\subset\mc{T}_i$. By continuity, taking limits in the identity 
$f_*\circ\xi=\xi\circ f$ on each $T_i$ yields 
$f_*\circ\xi(z)=\xi\circ f(z)$. Since $z$ was arbitrary, this holds for all $z\in J$. This completes the proof of Theorem~\ref{thm:intersection-I}.
\end{proof}

\section{Quasisymmetric rigidity of the Julia sets}\label{sec8}
In this section, we complete the proof of Theorem~\ref{thm:mobius}.
Before the proof, we outline the proof of Proposition \ref{thm:BLM}, with further details referring to \cite[Theorem 1.4]{BLM14} or \cite[Theorem 5.2]{QYZ}.

\begin{proof}[Sketch of the proof of Proposition~\ref{thm:BLM}]
Since $f$ is postcritically finite, there exists a system of B\"{o}ttcher maps on its Fatou domains that are conformally conjugate to $z\mapsto z^{\delta}$ on $\mathbb{D}$ piecewise \cite[Lemma 3.3]{BLM14}. By the dynamical relation in the condition and \cite[Lemma 7.1]{BLM14}, $\xi$ can be conformally extended to the periodic Fatou domains; and further extended to the entire Fatou set via iterative liftings of $f$ and $\tilde{f}$.

The extended $\xi$ is a homeomorphism of $\overline{\mathbb{C}}$, conformal on the Fatou set, and quasisymmetric on the Julia set. Moreover, it maps Fatou centers to Fatou centers.

Next, we only need to verify that $\xi$ is a M\"{o}bius transformation.  The key observation is that boundaries of the Fatou domains of $f$ and $\tilde{f}$ are uniform quasicircles \cite[Proposition 3.5]{QYZ}. Let $U$ be a Fatou domain of $f$ and $V=\xi(U)$. By the removability of quasicircles for quasiconformal maps and the Schwarz reflection principle on $\partial \mathbb{D}$, the restriction $\xi: \overline{U}\to \overline{V}$ is uniformly quasisymmetric.

Now, the maps restricted to closures of Fatou domains are uniformly quasisymmetric, and to $J$ is also quasisymmetric. As shown in the latter part of the proof of \cite[Proposition 5.1]{Bon11}, $\xi$ is quasiconformal on $\overline{\mathbb{C}}$.

Since the Julia set $J$ has measure zero \cite[Theorem 3.8]{QYZ}, $\xi$ is a quasiconformal map that is conformal on $\overline{\mathbb{C}}\setminus J$, which has full measure in $\ol{\mathbb{C}}$. Thus, $\xi$ is 1-quasiconformal, and this implies $\xi$ is a M\"{o}bius transformation.
\end{proof}

\begin{proof}[Proof of Theorem \ref{thm:mobius}]
By Lemma~\ref{lem:parameter-sym}, multiplying $\lambda$ by $e^{2\pi{\bf i}k/(n-1)}$ only rotates the Julia set by $e^{\pi{\bf i}k/(n(n-1))}$. Thus we may normalize: $\arg(\lambda), \arg(\tilde{\lambda})\in[0,{2\pi}/(n-1))$.
Moreover, if $J$ is a cluster, we replace $\tilde{f}$ with $f_*$ as in Theorem~\ref{thm:intersection-I}. \vspace{3pt}

For Sierpi\'{n}ski carpet Julia sets, Proposition~\ref{pro:sierpinski} gives the dynamical relation \eqref{eq:26}; for Sierpi\'{n}ski-like carpet Julia sets, Corollary~\ref{cor:sier} reduces them to the Sierpi\'{n}ski model and Theorem~\ref{thm:sierpinski-like} yields \eqref{eq:26}; for necklace and cluster Julia sets, Theorems~\ref{thm:necklace} and \ref{thm:intersection-I} give, after precomposing $\xi$ with an element of $G_{\lambda}$ if necessary, conjugacy relations of the form \eqref{eq:semi-conjugacy}, which are special cases of \eqref{eq:26}. Thus the quasisymmetry $\xi$ satisfies \eqref{eq:26} on $J$ in all four topological cases. Proposition~\ref{thm:BLM} then implies that $\xi$ extends to a M\"{o}bius transformation on $\overline{\mathbb{C}}$, and the relation still holds on $\overline{\mathbb{C}}$. Moreover, $\xi$ maps each Fatou domain of $f$ and its center to those of $\tilde{f}$. \vspace{5pt}

Since $\xi$ is a homeomorphism and $\deg f=\deg\tilde f$, comparing degrees in \eqref{eq:26} gives $m'=m+l$, i.e.,
$
\tilde{f}^{m+l}\circ\xi=\tilde{f}^{m}\circ\xi\circ f^{l}.
$
Furthermore, for any integer $q\ge1$, applying the relation repeatedly shows that
\[
\tilde{f}^{q(m+l)}\circ\xi=\tilde{f}^{qm}\circ\xi\circ f^{ql};
\]
indeed, the induction step follows by composing the relation on the left with $\tilde f^{qm}$. Hence we may replace $(l,m)$ by $(ql,qm)$ and assume $l\ge5$ and $m\ge0$. 

\medskip\noindent\textbf{Claim 1.} \emph{As a collection of Fatou domains, $\xi\{B, T\}=\{\tilde{B}, \tilde{T}\}$.}
\begin{proof}[Proof of Claim 1]
For necklace and cluster Julia sets the claim is clear, since $\xi$ is a conjugacy (Theorems~\ref{thm:necklace} and \ref{thm:intersection-I}). Thus we assume $J$ is a Sierpi\'{n}ski or Sierpi\'{n}ski-like carpet.

Suppose $U:=\xi^{-1}(\tilde T)\notin\{B,T\}$. From the dynamical relation we have
\[
\tu{deg}(\tilde f^{m+l}\circ\xi,U)=\tu{deg}(\tilde f^m\circ\xi\circ f^l,U)=n^{m+l},
\]
which forces $\deg(f^l,U)=n^l$. Hence $f$ has degree $n$ on each domain in the orbit segment $U,f(U),\dots,f^{l-1}(U)$, so each contains a critical point of local degree $n$. Since free critical points of $f$ are simple, this forces $n=2$.

Since $f^{-1}(T)$ contains no critical point (as the critical values $\pm2\sqrt{\lambda}\neq 0$), the orbit segment avoids $f^{-1}(T)$.
With $l\ge5$ and only four free critical points, two terms in the orbit segment coincide; thus the segment contains a Fatou domain cycle. Moreover, the two free critical values of $f$ are critical points. Since $n=2$, any two free critical points $c$ and $c'$ satisfy $f^2(c)=f^2(c')$. Thus, the four Fatou domains $V_i$ ($i=0,1,2,3$) containing free critical points, one of them being $U$, satisfy
\[
f^l(V_i)=f^l(U)\text{ and }\tu{deg}(f^l,V_i)=n^l.
\]
Combining this with the dynamical relation, we obtain
\[
\tu{deg}(\tilde f^{l+m},\xi(V_i))=\tu{deg}(\tilde f^m\circ\xi\circ f^l,V_i)=\tu{deg}(\tilde f^{l+m}\circ\xi,U)=n^{l+m},
\]
and
$
\tilde f^{m+l}(\xi(V_i))=\tilde f^m(\xi(f^l(U)))=\tilde B.
$
Hence, at least two of $V_0,\dots,V_3$ are sent by $\xi$ to Fatou domains other than $\tilde B$ and $\tilde T$, and these domains are iterated onto $\tilde B$. This is impossible because $\tilde f^{-1}(\tilde T)$ contains no critical points.
Therefore, $\xi^{-1}(\tilde T)\in\{B,T\}$. The same argument applies to $\tilde B$. This completes the proof of Claim~1.
\end{proof}
Next, by precomposing an element of $G_{\lambda}$ if necessary, we may assume $\xi(B)=\tilde{B}$ and $\xi(T)=\tilde{T}$. Thus, $\xi$ fixes $0$ and $\infty$. Write $\xi(z)=az$. We perform Laurent expansions for $\tilde{f}^{m+l}\circ\xi$ and $\tilde{f}^m\circ\xi\circ f^l$ in a punctured neighborhood of infinity.
Since $\tilde{f}^{m+l}\circ\xi=\tilde{f}^m\circ\xi\circ f^l$, the coefficients of their highest-order terms  coincide. It implies $|a|=1$.\vspace{5pt}

\noindent\textbf{Claim 2.} \emph{$\xi(f^{-1}(T))=\tilde{f}^{-1}(\tilde{T})$. }
\begin{proof}[Proof of Claim 2]
Still, it suffices to consider the case where $J$ is a Sierpi\'nski or Sierpi\'{n}ski-like carpet. Let $U$ be a Fatou domain in $f^{-1}(T)$ and $\tilde{U}=\xi(U)$. Since $$\tu{deg}(\tilde{f}^m\circ\xi\circ f^l, U)=\tu{deg}(\tilde{f}^{m+l}, \tilde{U})=n^{m+l-1},$$  we have $\tilde{U}\subseteq \tilde{f}^{-1}(\tilde{T})$ whenever $n\geq 3$.
For $n=2$, either $\tilde{U}\subseteq \tilde{f}^{-1}(\tilde{T})$ or $\tilde{U}$ is critical and $\tilde{f}^2(\tilde{U})=\tilde{T}$. It suffices to exclude the latter case. 

By considering the inverse $\xi^{-1}$, $U$ is critical and $f^2(U)=T$. A direct computation shows $\lambda=\tilde{\lambda}=-{1}/{16}$; the free critical points of $f$ and $\tilde{f}$ are $\{\pm\frac{1}{2}e^{\pi \tb{i}/4}, \pm\frac{1}{2}e^{-\pi\tb{i}/4}\}$, and the roots of $f$ and $\tilde{f}$ are $\{\pm{1}/{2}, \pm\tb{i}/2\}$. 

Note that the union of roots and free critical points of $f$ (resp. $\tilde{f}$) is evenly distributed on the circle $\{|z|={|\lambda|}^{1/4}\}$. The dynamical relation implies that $\xi(f^{-1}(T))$ consists of four critical Fatou domains of $\tilde{f}$. By considering the inverse $\xi^{-1}$, similarly $\xi^{-1}(\tilde{f}^{-1}(\tilde{T}))$ consists of four critical Fatou domains of $f$. Thus $\xi$ interchanges the roots of $f$ with the free critical points of $\tilde f$. 

By composing an element of $G_{\lambda}$ if necessary, we may assume $\xi(z)=e^{\pi\tb{i}/4}z$. Suppose $f(z_0)={1}/{2}$. Then the local degree of $\tilde{f}^m\circ\xi\circ f^l$ at the point $z_0$ equals $2^{m+l-2}$. However, since neither $\xi(z_0)$ nor $\tilde{f}(\xi(z_0))$ is a critical or root point of $\tilde{f}$, the local degree of $\tilde{f}^{m+l}\circ\xi$ at $z_0$ is strictly less than $2^{m+l-2}$, contradicting the dynamical relation.
\end{proof}

By Claim 2, the $2n$-th roots of $-\lambda$ are mapped onto those of $-\tilde{\lambda}$ under the affine transformation $\xi$. Since $|a|=1$, we have $\tilde{\lambda}=a^{2n}\lambda$ and
$a=e^{\tb{i}\theta}$ with $$\theta=\frac{\tu{arg}(\tilde{\lambda})-\tu{arg}(\lambda)}{2n},$$ up to precomposing $\xi$ with a rotation in $G_{\lambda}$.

It remains to show that $a^{2n}$ is an $(n-1)$-th root of unity. Suppose otherwise. Then $a^{-2n(n-1)}\neq1$, so $a^{-(n-1)}$ is not a $2n$-th root of unity. Define $h=\xi\circ f\circ\xi^{-1}$. A direct computation gives
\begin{equation}\label{eq:h}
h(z)={a^{-(n-1)}}\tilde f(z).  
\end{equation}

Since $h$ is conjugate to $f$ via $\xi$, we have $J_h=\xi(J)=\tilde J$. Thus, the identity map $\tu{id}:J_h\to\tilde J$ is a quasisymmetry. By the corresponding rigidity theorem (Proposition~\ref{pro:sierpinski}, Theorems~\ref{thm:sierpinski-like}, \ref{thm:necklace}, or \ref{thm:intersection-I}), there exist integers $l'\ge5$ and $m'\ge0$ such that
$$
\tilde f^{l'+m'}=\tilde f^{m'}\circ h^{l'}\text{ on }\overline{\mathbb{C}}.
$$

Pick a point $z_0$ that is not a critical point of $\tilde{f}$ but whose image $\tilde{f}(z_0)$ is a root of $\tilde{f}$. Such a point $z_0$ exists since $\tilde{f}$ has only two critical values and $2n$ roots. Thus, the local degree of $\tilde{f}^{l'+m'}$ at $z_0$ satisfies
$$\tu{deg}(\tilde{f}^{l'+m'}, z_0)=n^{l'+m'-2}.$$

Note that the roots and critical points of $h$ coincide with those of $\tilde{f}$ by \eqref{eq:h}. Since $a^{-(n-1)}$ is not a $2n$-th root of unity by assumption, the point $h(z_0)=a^{-(n-1)}\tilde{f}(z_0)$ is not a root of $h$. Hence, the local degree of $\tilde{f}^{m'}\circ h^{l'}$ at $z_0$ satisfies
$$\tu{deg}(\tilde{f}^{m'}\circ h^{l'}, z_0)\leq 2n^{l'+m'-3}.$$
The dynamical relation forces these two degrees to coincide, which can happen only if $n=2$, $h(z_0)$ is a critical point of $h$, and $h^2(z_0)$ is a root of $h$. Since $h$ is conjugate to $f$ via $\xi$, the point $c_0:=\xi^{-1}(h(z_0))$ is a critical point of $f$, and $f(c_0)$ is a root of $f$. A direct computation gives $f(z)=z^2-{1}/{(16z^2)}$. 

Now pick a point $w_0$ that is not a critical point of $h$ but whose image $h(w_0)$ is a root of $h$. A similar argument yields $\tu{deg}(\tilde{f}^{m'}\circ h^{l'}, w_0)=2^{l'+m'-2}$. By the assumption that $a^{n-1}$ is not a $2n$-th root of unity,  the point $\tilde{f}(w_0)=a^{n-1}h(w_0)$ is not a root of $\tilde{f}$. Thus, $\tu{deg}(\tilde{f}^{l'+m'}, w_0)\leq 2^{l'+m'-2}$. The dynamical relation forces equality. Thus, $\tilde{f}(w_0)$ is a critical point of $\tilde{f}$ and $\tilde{f}^2(w_0)$ is a root of $\tilde{f}$. This implies $\tilde{f}(z)=z^2-1/(16z^2)$.

The roots of both $f$ and $\tilde{f}$ are $\{\pm1/2, \pm\textbf{i}/2\}$. By Claim 2, $\xi$ maps the roots of $f$ onto those of $\tilde{f}$. Thus, $a^{2n}=1$, which is obviously an $(n-1)$-th root of unity. This completes the proof of the theorem.
\end{proof}

%----------------------------------------------------------------------------------------------------------------


\begin{thebibliography}{99}

%\bibitem[AIM09]{AIM09} K. Astala, T. Iwaniec, G. M. Martin, \textit{Elliptic partial differential equations and quasiconformal mappings in the plane}, Princeton Univ. Press, Princetion, NJ, 2009.

%\bibitem[Bea83]{Bea83}A. F. Beardon, \emph{The geometry of discrete groups}, Springer, New York, 1983.
\bibitem[Ah]{Ah} L. V. Alhfors, Lectures on quasi-conformal mappings, Wadsworth \& Brook/Cole, Advanced Books \& Software, Monsterey 1987.

\bibitem[BF]{BF} J. Belk and B. Forrest, Quasisymmetries of finitely ramified Julia sets, \emph{Math. Ann.}, (2025) 393:1683-1740.

\bibitem[Bon06]{Bonk2006ICM} M. Bonk, The quasiconformal geometry of fractals, in \textit{Proceedings of the International Congress of Mathematicians, Madrid 2006}, Vol.~II, 1349--1373, Eur. Math. Soc., Z\"{u}rich, 2006.

\bibitem[Bon11]{Bon11}M. Bonk, Uniformization of Sierpi\'{n}ski carpets in the plane, \textit{Invent. Math.} \textbf{186} (2011), 559-665.

\bibitem[BM13]{BM1}M. Bonk and S. Merenkov, Quasisymmetric rigidity of square Sierpi\'{n}ski carpets, \textit{Ann. of Math.} (2)177(2013) 591-643.

%\bibitem[BM1]{BM1}M. Bonk and S. Merenkov, Quasisymmetric rigicity of square Sierpi\'{n}ski carpets, \textit{Ann. of Math.} (2)177(2013) 591-643.


\bibitem[BM20]{BM2}M. Bonk and S, Merenkov, Square Sierpi\'{n}ski carpets and Latt\`{e}s maps, \textit{ Math. Z.} 296(2020), 695--718.

\bibitem[BLM]{BLM14}M. Bonk, M. Lyubich and S. Merenkov, Quasisymmetries of Sierpi\'{n}ski carpet Julia sets, \textit{Adv. Math.} {301}(2016) 383-422.

\bibitem[BKM]{BKM} M. Bonk, B. Kleiner, S. Merenkov, Rigidity of Schottky sets, \textit{Amer. J. Math}. 131 (2009) 409--443.

%\bibitem[Bou97]{Bou97}M. Bourdon, Immeubles hyperboliques, dimension conforme et rigidit\'{e} de Mostow, \textit{Geom. Funct. Anal.} \textbf{7} (1997), 245-268.

%\bibitem[BF]{BF14}B. Branner and N. Fagella, \emph{Quasiconformal Surgery in Holomorphic Dynamics}, Cambridge University Press, New York, 2014.

%\bibitem[BP]{BP02} M. Bourdon and H. Pajot, Quasiconformal geometry and hyperbolic geometry, in \emph{Rigidity in dynamics and geometry} (Cambridge, 2000), pages 1-17, Springer, Berlin, 2002.

\bibitem[CGZ]{CGZ} G. Cui, Y. Gao and J. Zeng, Invariant graphs in Julia sets and decompositions of rational maps, arXiv: 2408.12371v2[math.DS], 23 Nov 2024.

\bibitem[CPT1]{CPT} G. Cui, W. Peng, and L. Tan, On a theorem of Rees-Shishikura, \textit{Ann. Fac. Sci. Toulouse Math.} \textbf{21} (2012), no. 5, 981-993.
\bibitem[CPT2]{CPT2} G. Cui, W. Peng, and L. Tan, Renormalizations and wandering
Jordan curves of rational maps, \textit{Commun. Math. Phys}. 344 (2016), 67--115.

\bibitem[CT]{CT} G. Cui, and L. Tan, A characterization of hyperbolic rational maps, \textit{Invent. math.}. 183 (2011), 451--516.
%\bibitem[DFGJ14]{DFGJ14}R. L. Devaney, N. Fagella, A. Garijo and X. Jarque, Sierpi\'{n}ski curve Julia sets for quadratic rational maps, \textit{Ann. Acad. Sci. Fenn. Math.} \textbf{39} (2014), 3-22.

%\bibitem[DH]{DH}


\bibitem[DLU]{DLU05}R. L. Devaney, D. Look and D. Uminsky, The escape trichotomy for singularly perturbed rational maps, \textit{Indiana Univ. Math. J.} \textbf{54} (2005), 1621-1634.

\bibitem[DP]{DP}R. L. Devaney and K. Pilgrim, Dynamic classification of escape-time Sierpi\'{n}ski curve
Julia sets. \emph{Fund. Math.}, \textbf{202}(2009), 181-198.

\bibitem[GZ]{GZ} Y. Gao, J. Zeng, Polynomials in molecules, arXiv: 2601. 18605v1.

%\bibitem[Hat]{Hat02}A. Hatcher, \emph{Algebraic topology,} Cambridge Univ. Press, Cambridge, 2002.

\bibitem[HP]{HP12}P. Ha\"{\i}ssinsky and K. Pilgrim, Quasisymmetrically inequivalent hyperbolic Julia sets, \textit{Revista Math. Iberoamericana}, \textbf{28} (2012), 1025-1034.

%\bibitem[Kle06]{Kle06}B. Kleiner. The asymptotic geometry of negatively curved spaces: uniformization, geometrization and rigidity, In \emph{International Congress of Mathematicians. Vol. II}, pages 743-768. Eur. Math. Soc., Z\"{u}rich, 2006.

%\bibitem[KL09]{KL09}J. Kahn and M. Lyubich, Quasi-Additivity Law in conformal geometry, \textit{Ann. Math.} \textbf{169} (2009), 561-193.

%\bibitem[LV73]{LV73}O. Lehto and K. I. Virtanen, \textit{Quasiconformal Mappings in the Plane}, Springer Verlag, Berlin, Heidelberg, New York, 1973.

%\bibitem[Lev90]{Lev90} G. Levin, Symmetries on Julia sets (Russian), \textit{Mat. Zametki} \textbf{48} (1990), 72-79, 159; translation in Math. Notes \textbf{48} (1990), 1126-1131 (1991).


%\bibitem[Lev01]{Lev01} G. Levin, Letter to the editors: ``Symmetries on Julia sets" (Russian), Mat. Zametki \textbf{69} (2001), 479-480; translation in Math. Notes \textbf{69} (2001), 432-433.
\bibitem[LLMM]{LLMM}R. Lodge, M. Lyubich, S. Merenkov and S. Mukherjee, On dynamical gaskets generated by rational maps, Kleinian groups, and Schwarz reflections. \emph{Conform. Geom. Dyn.}, 27(01), 1-54(2023).
%\bibitem[LP]{LP97}G. Levin and F. Przytycki, When do two rational functions have the same Julia set? \textit{Proc. Amer. Math. Soc.} \textbf{125} (1997), 2179-2190.

%\bibitem[Lyu97]{Lyu97}M. Lyubich, Dynamics of quadratic polynomials I-II, \textit{Acta Math.} \textbf{178} (1997), 185-297.

\bibitem[LM]{LM} M. Lyubich and S. Merenkov, Quasisymmetries of the basilica and the Thompson group. \emph{Geom. Funct. Anal.}, 28(3),727-754(2018).

\bibitem[LMM]{LMM} Y. Luo, M. Mj and S. Mukherjee, Universality of the Basilica, arXiv:2601.13553v1.% [math.DS] 20 Jan 2026.

%\bibitem[LN]{LN} Y. Luo and D. Ntalampekos, Uniformization of gasket Julia sets, arXiv:2411.17227v1.% [math.DS] 26 Nov 2024.

%\bibitem[Ma\~{n}93]{Man93} R. Ma\~{n}\'{e}, On a lemma of Fatou, \textit{Bol. Soc. Bras. Mat.} \textbf{24} (1993), 1-11.

%\bibitem[McM94]{McM94}C. T. McMullen, \textit{Complex Dynamics and Renormalization}, Ann. Math. Studies \textbf{135}, Princeton U. Press, Princeton, NJ, 1994.

%\bibitem[McM95]{McM95}C. T. McMullen, The classification of conformal dynamical system. In \emph{Current developments in mathematics, 1995 (Cambridge, MA),} pages 323-360. Internat. Press, Cambridge, MA 1994.

%\bibitem[Ma]{Man93} R. Ma\~{n}\'{e}, On a lemma of Fatou, \textit{Bol. Soc. Bras. Mat.} \textbf{24} (1993), 1-11.

\bibitem[McM]{McM98}C. T. McMullen, Self-similarity of Siegel disks and Hausdorff dimension of Julia sets, \textit{Acta Math.}, \textbf{180}(1998), 247-292.

\bibitem[Mer10]{Mer}S. Merenkov, A Sierpi\'{n}ski carpet with the co-Hopfian property, \textit{Invent. Math.} 180 (2010), 361-388.

\bibitem[Mer12]{Mer12}S. Merenkov, Planar relative Schottky sets and quasisymmetric maps, \textit{Proc. Lond. Math.
Soc.} (3) \textbf{104} (2012), no. 3, 455-485

\bibitem[Mer14]{Mer14}S. Merenkov, Local rigidity of Schottky maps, \textit{Proc. Amer. Math. Soc.} \textbf{142} (2014), no. 12, 4321-4332.

%\bibitem[Mil93]{Mil93} J. Milnor, Geometry and dynamics of quadratic rational maps, with an appendix by J. Milnor and L. Tan, \textit{Exper. Math.} \textbf{2} (1993), Vol 1, 37-83.

%\bibitem[Mil00]{Mil00} J. Milnor, Periodic orbits, externals rays and the Mandelbrot set: An expository account, in ``\textit{G\'{e}om\'{e} Complexe et Syst\`{e}mes Dynamiques, Colloque en l'honneur d'Adrien Douady}'', \textit{Ast\'{e}risque}, \textbf{261} (2000), 277-333.

%\bibitem[Mil]{Mil06} J. Milnor, \textit{Dynamics in One Complex Variable}, 3rd ed., Princeton Univ. Press, Princeton, NJ, 2006.

%\bibitem[Moo25]{Moo25}R. L. Moore, Concerning upper semicontinuous collections of compacta, \emph{Trans. Amer. Math. Soc.} \textbf{27} (1925), 416-426.

%\bibitem[MS]{MS} C. T. McMullen and D. Sullivan, Quasiconformal homeomorphisms and dynamics. III. The Teichm\"{u}ller space of a holomorphic dynamical system. \emph{Adv. Math.}, 135(1998), 351-395.


\bibitem[QRWY]{QRWY15} W. Qiu, P. Roesch, X. Wang and Y. Yin, Hyperbolic components of McMullen maps. \emph{Ann. Sci. \'{E}c. Norm. Sup\'{e}r (4)}, 48(3)(2015), 703-737.

\bibitem[QWY]{QWY12}W. Qiu, X. Wang and Y. Yin, Dynamics of McMullen maps, \textit{Adv. Math.} \textbf{229} (2012), 2525-2577.



\bibitem[QYZ]{QYZ} W. Qiu, F. Yang, and J. Zeng, Quasisymmetric geometry of Sierpi\'{n}ski carpet Julia sets, \textit{Fund. Math.} 244 (2019), no. 1, 73-107.


%\bibitem[Ste06]{Ste06}N. Steinmetz, On the dynamics of the McMullen family $R(z)=z^m+\lambda/z^\ell$, \textit{Conform. Geom. Dyn.} \textbf{10} (2006), 159-183.

%\bibitem[Sul85]{Sul85}D. Sullivan, Quasiconformal homeomorphisms and dynamics I: Solution of the Fatou-Julia problem on wandering domains, \textit{Ann. of Math.} \textbf{122} (1985), 401-418.

%\bibitem[Tio15]{Tio15}G. Tiozzo, Topological entropy of quadratic polynomials and dimension of sections of the Manderbrot set, \textit{Adv. Math.} \textbf{273} (2015), 651-715.

%\bibitem[Why58]{Why58}G. Whyburn, Topological characterization of the Sierpi\'{n}ski curves, \textit{Fund. Math.} \textbf{45} (1958), 320-324.

%\bibitem[XQY14]{XQY14}Y. Xiao, W. Qiu and Y. Yin, On the dynamics of generalized McMullen maps, \emph{Ergod. Th. $\&$ Dynam. Sys.} {\bf 34} (2014), 2093-2112.

%\bibitem[Yin99]{Yin99}Y. Yin, On the Julia sets of semi-hyperbolic rational maps, \textit{Chinese J. Contemp. Math.} \textbf{20} (1999), no. 4, 469-476.

%\bibitem[Yin]{Yin}Y. Yin,  Geometry and dimension of Julia sets, in \textit{The Mandelbrot Set, Theme and Variations}, London Math. Soc. Lecture Note Ser. \textbf{274}, Cambridge Univ. Press, Cambridge, 2000, pp. 281-287.

%\bibitem[Zak03]{Zak03}S. Zakeri, External rays and the real slice of the Mandelbrot set, \textit{Ergod. Th. $\&$ Dynam. Sys.} \textbf{23} (2003), 637-660.

\end{thebibliography}
\end{document}